\documentclass[12pt]{amsart}
\usepackage[utf8]{inputenc}
\usepackage{amssymb,amsbsy,amsmath,amscd,amsthm,amsfonts,MnSymbol,mathrsfs}
\usepackage{cite}
\usepackage{natbib}
\usepackage[colorlinks]{hyperref}
\usepackage{graphicx}
\usepackage{verbatim}
\usepackage{float}

\usepackage{bbm}
\usepackage{subfigure}
\usepackage{enumerate}

\usepackage{xcolor}
\usepackage[normalem]{ulem}

\usepackage{pdfsync}

\usepackage{mathtools}

\newcommand{\vep}{\varepsilon}

\newcommand{\R}{\mathbb{R}}
\newcommand{\bbr}{\mathbb{R}}
\newcommand{\vague}{\stackrel{v}{\rightarrow}}

\newcommand{\LFe}{F_{e(X)}}
\newcommand{\cov}{\mathrm{cov}}
 
 \newcommand{\bU}{\mathbf{U}}

\newcommand{\one}{\mathbf{1}}

\newcommand{\argmin}{\mathrm{argmin}}

\newcommand{\E}{\mathbb{E}}

\newcommand{\var}{\mathrm{Var}}

\theoremstyle{plain}
\newtheorem{theorem}{Theorem}[section]
\newtheorem{lemma}[theorem]{Lemma}
\newtheorem{proposition}[theorem]{Proposition}
\newtheorem{corollary}[theorem]{Corollary}
\theoremstyle{definition}

\newtheorem{remark}[theorem]{Remark}

\newtheorem{example}[theorem]{Example}
\newtheorem{assumption}[theorem]{Assumption}

\newcommand{\bbP}{\mathbb{P}}
\newcommand{\bx}{\mathbf{x}}
\newcommand{\cip}{\overset{P}\to}

\definecolor{Blue}{rgb}{0,0,1}

\definecolor{red}{rgb}{1,0,0}

\title{Estimating quantile treatments  without strict overlap}

\author{Marco Avella Medina}
 \address{Department of Statistics \\
 	Columbia University}
 \email{marco.avella@columbia.edu}
 
 \author{Richard A. Davis}
 \address{Department of Statistics \\
 	Columbia University}
 \email{rdavis@stat.columbia.edu}
 
 \author{Gennady Samorodnitsky}
 \address{School of Operations Research and Information Engineering\\
 	Cornell University}
 \email{gs18@cornell.edu}

 	\thanks{This research was partially
 	supported by  NSF grants DMS-2310973 (Avella Medina and Davis) at Columbia and DMS- 23109734 (Samorodnitsky) at Cornell. }

 
\begin{document}

\maketitle

\begin{abstract}
We consider the problem of estimating quantile treatment effects without assuming strict overlap, i.e., we do not assume that the propensity score is bounded away from zero. More specifically, we consider an inverse probability weighting (IPW) approach for estimating quantiles in the potential outcomes framework and pay special attention to scenarios where the propensity scores can tend to zero as a regularly varying function. Our approach effectively considers a heavy-tailed objective function for estimating the quantile process. We introduce a truncated  IPW estimator that is shown to outperform the standard quantile IPW estimator when strict overlap does not hold. We show that the limiting distribution of the estimated quantile process follows an infinitely divisible law  and converges at the  rate $n^{1-1/\gamma}$, where $\gamma>1$ is the tail index of the propensity scores when they tend to zero.  We propose a practical, data‑driven procedure for selecting the truncation parameter, grounded in our asymptotic theory. The performance of our estimators is illustrated in numerical experiments and in a dataset that exhibits the  presence of extreme propensity scores.

\end{abstract}

\section{Introduction}

Inverse probability weighting (IPW) estimators have become ubiquitous in the theory and practice of causal inference with observational data since the seminal work of \cite{rubin1974,rosenbaum:rubin1983,robins:rotnitzky:zhao1994}. These methods work well under assumptions of unconfoundedness and overlap. The former requires the treatment assignment mechanism to only depend on observed covariates. The latter, also known as positivity or common support,  requires all units to have a non-zero propensity score i.e. a non-zero probability of being assigned to each treatment condition. In fact a stronger condition called strict overlap is frequently assumed in the literature. It requires the propensity score to be bounded away from $0$ and $1$. Its popularity stems from the fact that it is a  necessary condition for the existence of  semiparametric estimators of the average treatment effect that are uniformly $\sqrt{n}$-consistent over a nonparametric class of models \cite{khan:tamer2010}. 

At a high level the main conceptual contribution of our work consists of showing how ideas from extreme value theory can be leveraged to  construct  IPW-type estimators that work well under relaxed overlap conditions.  In particular, we focus on estimating quantile treatment effects for fixed quantiles $\tau\in(0,1)$ and intermediate quantiles that change as the sample size $n$ increases as $\tau=\tau_n\to 0$ and  $\tau_nn\to\infty$ as $n\to \infty$. We relax the quantile IPW estimation frameworks of \cite{firpo2007} and \cite{zhang2018} as we do not assume strict overlap but an alternative regular variation condition of the distribution of the propensity score near $0$. This assumption effectively allows the summands of the IPW quantile objective function to be heavy-tailed which in turn connects this work to the extreme value literature. Our approach can be viewed as an extension of the mean IPW estimators of \cite{ma:wang2020} and \cite{heiler:kazak2021} which also allowed for extreme propensity scores. Our proposal is closer to \cite{heiler:kazak2021} since our quantile estimator is a stabilized IPW estimator similar to the one they considered for mean estimation. However, our study of quantile  estimators lead to the following main  four methodological and theoretical contributions.

\begin{enumerate}[(i)]
    \item
    \textbf{Quantile process estimation:} we consider the problem of jointly estimating  quantile treatment effects for a sequence of quantile levels $\{\tau_j\}_{j=1}^m\in(\eta,1-\eta)$ where  $m\in\mathbb{N}$ and  $\eta\in (0,1/2)$. We propose IPW-type estimators that behave well under a relaxation of  the strict overlap condition. In particular, we allow  the inverse propensity score to be heavy-tailed as  introduced in \cite{ma:wang2020,heiler:kazak2021}. We derive the joint asymptotic distribution of a finite collection of quantile effects, effectively extending the results of \cite{firpo2007} to multiple quantile effects while relaxing his strict overlap condition. We extend our analysis to intermediate  quantiles that change as the sample size $n$ increases, namely we consider the case where   $\tau=\tau_n\to 0$ and $\tau_nn\to\infty$ as $n\to \infty$. This analysis extends the intermediate quantile results of \cite{zhang2018}  to the scenario where there are extreme inverse propensity scores. 
    \item 
    
    \textbf{Truncation:} we propose a new mean stabilized IPW estimator that uses truncated propensity scores as an alternative to the stabilized IPW considered in \cite{heiler:kazak2021}. This allows to trade off bias and variance, and leads to an improved mean squared error, especially in the very heavy-tailed regimes. This truncation approach is somehow related to the trimming idea considered in \cite{ma:wang2020}  for their mean IPW estimator. The inherent reliability of our stabilized IPW estimator and the ``soft-rejection'' nature of truncation, as opposed to trimming observations associated with small propensity scores, lead to a better MSE. 
    
    \item 
    \textbf{Convex objectives with stable limit:} our proof techniques depart from the direct analysis of the  closed form mean estimators proposed in \cite{ma:wang2020,heiler:kazak2021}. Our work extends the previously known limiting distribution results for quantile estimators  of \cite{firpo2007,zhang2018} to the heavy-tailed case where the limit is a stable distribution. The core idea of our analysis is to characterize the limiting distribution of the empirical process corresponding to the convex objective function that defines our estimator.  By convexity the limiting distribution of our estimator is the distribution of the argmin of the limiting distribution of this empirical process. Similar ideas have been used by \cite{davis1992m, knight2002limiting, goh2009nonstandard} in the context of parameter estimation of heavy-tailed processes and quantile estimation. 

    \item 
     \textbf{Data-driven selection of the truncation parameter:} we propose a practical, data‑driven procedure for selecting the truncation parameter, grounded in our asymptotic theory. The asymptotic  mean of the truncated IPW estimator admits an explicit representation in terms of a parameter $\theta$, which is directly linked to the truncation level. By estimating this function via  regression and applying bootstrapping methods, we obtain nearly unbiased estimates across a range of truncation levels. We then choose the truncation level that minimizes the mean squared error of these nearly unbiased estimates. The full procedure, along with an empirical assessment of its performance, is presented in Section \ref{sec:numexp}.  
\end{enumerate}

\subsection{Related literature}

There is a very large body of literature studying the problem of causal inference with the potential outcomes framework. Representative books of this literature include \cite{imbens:rubin2015,rosenbaum2017,ding2024}. The assumptions of unconfoundedness and overlap are closely  tied to the use of IPW estimators for causal inference. This goes back to the seminal papers of \cite{rubin1974, rosenbaum:rubin1983,robins:rotnitzky:zhao1994}. There is a vast literature discussing various issues related to the importance  of overlap in observational studies and we just limit ourselves to a modest overview of some of those lines of work. Several authors have dealt with the important practical concern of limited overlap  by resorting to trimming observations that have estimated propensity scores that are either too small or too large \cite{dehejia:wahba1999,crump:hotz:imbens:mitnik2009,busso:dinardo:mccrary2014,yang:ding2018, matsouaka:zhou2024,liu:li:zhou:matsouaka2024}. Others have shown the importance of overlap  in the rates of convergence attainable by semiparametric estimators \cite{khan:tamer2010,ma:wang2020,armstrong:kolesar2021,khan:nekipelov2022}. \cite{damour:ding:feller:lei:jasjeet2021} pointed out that recent efforts to incorporate machine learning ideas to incorporate covariate adjustments for enhanced nonparametric flexibility, such approaches also lead to poor overlap.

Quantile treatment effects became  particularly popular in econometrics due to their ability to capture the heterogeneous effect of a   treatment in  a population \cite{abadie:angrist:imbens2002,chernozhukov:hansen2006,firpo2007}. In many applications extreme quantiles are of interest \cite{chernozhukov2005,zhang2018}. Examples can include measuring the effect of smoking on low birth weights \cite{abrevaya:dahl2008}, survival analysis at extreme durations \cite{koenker:geling2001}, and high risk in finance \cite{tsay2005}. The study of extreme quantiles requires special care since standard theory for quantiles does not hold in such scenarios \cite{chernozhukov2005,wang:li:he2012,wang:li2013, zhang2018,deuber:lo:engelke:maathuis2023}.

\section{Notation and background}
\label{sec:notation} 

We consider a standard potential outcome framework where $D\in\{0,1\}$ is a binary treatment indicator and $Y\in\mathbb{R}$  is the response variable of interest.  We denote by $Y(1)$ and $Y(0)$ the potential outcomes of $Y$ under treatment $D=1$ and $D=0$ respectively, and assume that we only observe $D$ as well as the response $Y=DY(1)+(1-D)Y(0)$. In this context a typical population parameter of interest is the average treatment effect $\mathbb{E}[Y(1)]-\mathbb{E}[Y(0)]$. In this work, we
 focus instead on the problem  of estimating the causal effects of a binary choice $D=1$ vs. $D=0$ on quantiles of $Y$. More specifically, we consider the quantile treatment effects (QTE) defined as
 \begin{equation*}
     \label{eq:QTE}
     \delta(\tau):=q_1(\tau)-q_0(\tau),
 \end{equation*}
where $\tau\in(0,1)$ and $q_j(\tau):=\inf\{y\in \mathbb{R}\,:\,F_j(y)\geq\tau\}$ denotes the $\tau$-quantile of the potential outcome $Y(j)$ and $F_j(y)$ denotes its cumulative distribution function. We will be interested both in the problem of  estimating  the QTE corresponding to fixed quantiles $\tau$ as well as in the case where $\tau$ is close to $0$.

We start by recalling how IPW-type estimators work in the more conventional situation of estimating and comparing the mean responses, when the propensity scores are relatively likely to take values close to 0 and/or to 1.  
The usual assumption is that of a random sample $\bigl( (Y_i(1),Y_i(0)),
D_i, X_i), \, i=1,2,\ldots\bigr)$ of i.i.d. vectors with the response $Y_1\in\mathbb{R}$, the covariate vector $X_1\in\mathbb{R}^d$ and  the treatment indicator $D_1\in\{0,1\}$. The propensity score is defined as 
$e(X_1)=  \bbP(D_1=1|X_1)$, takes values in $(0,1)$ and its law is denoted by $\LFe$. In most cases in the distributional assumptions on the sample the covariate vector $X_1$ is replaced by the corresponding propensity score $e(X_1)$ as we indicate below. 
 
\begin{assumption}[Unconfoundedness]\label{ass:unconfoundedness}
    The variables $D_1$ and $(Y_1(1),Y_1(0))$ are conditionally independent
given $e(X_1)$. 
\end{assumption}

\begin{assumption}[Regular variation]\label{ass:RV_prop_score}
    For some $\gamma_1>1$, the distribution of the propensity score has regularly varying tails with index $\gamma_1-1$ at zero, i.e.,
    \begin{equation*}
        \lim_{t\to 0^+}\frac{\bbP(e(X_1)\leq t x)}{\bbP(e(X_1)\leq t)}=x^{\gamma_1-1}, \quad \forall x>0.
    \end{equation*}
\end{assumption}
In many cases we also impose a similar assumption at 1, i.e.,
\begin{equation*}
        \lim_{t\to 0^+}\frac{\bbP(e(X_1)> 1-t x)}{\bbP(e(X_1)> 1-t)}=x^{\gamma_0-1}, \quad \forall x>0
    \end{equation*}
for some $\gamma_0>1$  that may differ from $\gamma_1$.

Assumption \ref{ass:RV_prop_score} removes the strong overlap condition which requires that 
$$\mathbb{P}(e(X_i)\in (\xi,1-\xi) )=1 \quad \mbox{ for some } \quad \xi>0.$$
The above setup is  similar to the one in \cite{ma:wang2020, heiler:kazak2021}. Both papers independently introduced  Assumption \ref{ass:RV_prop_score}. In \cite{ma:wang2020} the authors suggested a robust  IPW estimator of the mean $\theta=\mathbb{E} (Y_1)$  (where now $Y_i$ means  $Y_i(1)$), given by
\begin{equation} \label{e:ma.wang}
    \hat\theta_{n,b_n}^{\mathrm{IPW}}= \frac{1}{n}\sum_{i=1}^n \frac{D_iY_i}{\hat e(X_i)}\one (\hat e(X_i)>b_n)\,,
\end{equation}
with $\hat e(X_i)$ being an estimator of the propensity score $e(X_i)$, and $b_n$ a sequence of positive numbers converging to 0.  They show that, under certain conditions,  the robust IPW estimator, properly centered and scaled, has an infinitely divisible limiting distribution.  In the absence of thresholding, i.e., $b_n=0$, the limit has a stable distribution.  

\cite{heiler:kazak2021} also consider a condition related to Assumption \ref{ass:RV_prop_score} but propose instead to use the stabilized IPW estimator
\begin{equation*}
    \hat\theta_{n}^{\mathrm{SIPW}}= \left(\sum_{i=1}^n \frac{D_i}{\hat e(X_i)}\right)^{-1}\sum_{i=1}^n \frac{D_iY_i}{\hat e(X_i)}\,,
\end{equation*}
where $\hat e(X_i)$ is, again, an estimator of the propensity score $e(X_i)$. The authors showed that under certain additional assumptions   the stabilized IPW estimator has a stable limit, which is similar to the estimator of \cite{ma:wang2020} in the case of light (or altogether absent) thresholding. 
We note that the stabilized estimator of \cite{heiler:kazak2021} can be viewed as the minimizer of the weighted least squares objective function,
\begin{equation}
\label{eq:weighted_LS}
    \hat{\theta}^{\mathrm{SIPW}}_{n}= \mbox{argmin}_u \sum_{i=1}^n \frac{D_i}{\hat e(X_i)}(Y_i-u)^2\,.
\end{equation}

In this paper we introduce a robust IPW method of estimating quantiles. In a sense, this  can be viewed as an extension of the robust mean  IPW methods of \cite{ma:wang2020, heiler:kazak2021} to the problem of estimation of quantiles. It will 
involve a minimization problem as in \eqref{eq:weighted_LS}, but the quadratic function in \eqref{eq:weighted_LS} will be replaced by the check loss function, suitable for the quantile problem, as in  \cite{firpo2007,zhang2018}.  

Our approach involves another step, which differs from the estimator \eqref{e:ma.wang} even when estimating the mean in that we do not discard observations with propensity scores below the truncation level $b_n$. Indeed, in the context of mean estimation, our estimator amounts to solving the minimization problem 
\begin{equation*}
    \label{eq:trunc_SIPW}
    \tilde \theta_{n,b_n}= \mbox{argmin}_u \sum_{i=1}^n \frac{D_i}{\max\{\hat e(X_i),b_n\}}(Y_i-u)^2\,.
\end{equation*}

In Section \ref{sec:IPW-quantile}, we consider similar estimates for the $\tau$-quantile of the distribution for $Y_i$. Such an estimator can also be defined as the minimizer of an objective function suitably weighted by the truncated propensity scores.  We show that this robust estimate of the $\tau$-quantile, properly normalized, has a limiting distribution that is infinitely divisible.  This infinitely divisible distribution  depends on a parameter $\theta$  related to the level of truncation with $\theta=0$ corresponding to no truncation and $\theta>0$ corresponding to an increasing level of truncation.  If  $\gamma_1\in(1,2)$, then the limit distribution is $\gamma_1$-stable if $\theta=0$ and non-stable infinitely divisible without a Gaussian component if $\theta>0$.  If $\gamma_1>2$, then the limit is Gaussian.

In Section \ref{sec:intermediate}, similar results are obtained for the robust IPW estimates of the intermediate $\tau$-quantiles in which the level of the quantile $\tau_n$ converges to zero with the sample size at rate $n\tau_n\to \infty$.  The proofs of the main results in Sections \ref{sec:IPW-quantile} and \ref{sec:intermediate} rely on a functional convergence of an objective function related to quantile estimation and whose argument corresponds to a rescaling of the parameter in a local neighborhood of the true quantile.  Since the limit process is quadratic in the argument,  the limit of the IPW quantile estimate follows directly from the continuous mapping theory, and is explicitly computable.    

The joint limit distribution of the IPW $\tau$-quantiles for treatments 0 and 1 is derived in Section \ref{sec:qte}.  Interestingly, as long as $\gamma_0$ and $\gamma_1$ are less than 2, these quantile estimates are asymptotically independent.  The estimate of the quantile treatment effect, defined as the difference in the two quantiles, follows immediately from the joint limit distribution.

  Section \ref{sec:numexp} contains simulation results illustrating  the behavior of the IPW estimates of quantiles for different choices of threshold $b_n$.  Section \ref{sec:empirical} applies the methodology established in earlier sections to the well studied National Supported Work program data.  Finally, all the proofs and technical arguments are provided in the Appendix.

\section{IPW estimation for fixed quantiles} \label{sec:IPW-quantile}

We use the same setup as in  \cite{ma:wang2020}. Let $\bigl( (e(X_i),
D_i, Y_i), \, i=1,2,\ldots\bigr)$ be i.i.d. random vectors such that
$e(X_1)$ takes values in $(0,1)$ and has the law $\LFe$, while $D_1$ takes
values in $\{0,1\}$. We further assume that unconfoundedness and regular variation (Assumptions \ref{ass:unconfoundedness}--\ref{ass:RV_prop_score}) hold and in this section also $Y_i$ means  $Y_i(1)$. 

For $0<\tau<1$ suppose that there is a unique number $q(\tau)\in\bbr$
such that
\begin{equation} \label{e:quant.def}
    \bbP(Y_1\leq q(\tau))=\tau.
\end{equation}
Then $q(\tau)$ is the $\tau$-quantile of the distribution of $Y_1$,
and the goal is to estimate $q(\tau)$  based on observations $\bigl( (e(X_i),
D_i, Y_i), \, i=1,2,\ldots, n\bigr)$. We assume that
\begin{eqnarray} \label{e:dens,tau}
& &\text{$Y_1$ has a density  $g_{Y}$ in a neighborhood of $q(\tau)$ and  $g_Y(q(\tau))>0$.} 
\end{eqnarray}
We will introduce additional assumptions on the conditional distributions of $Y_1$ given $e(X_1)$ in the sequel. We will often consider multiple quantiles at the same
time. When doing so, we will always assume that all the  assumptions
hold at each one of the quantiles considered.

Let $\{b_n\}$ be a positive sequence  of numbers converging to zero and for a fixed $\tau\in (0,1)$, consider the stochastic process  
\begin{equation} \label{e:X.n}
X_n(\tau;t) = \sum_{i=1}^n \frac{D_i}{\max(e(X_i),b_n)} \rho_\tau(Y_i-t),
 \ t\in\bbr,
\end{equation}
 where $\rho_\tau(y)$ is the ``check'' function (see \cite{koenker1982tests}),
\begin{equation*} \label{e:check.f}
  \rho_\tau(y)=y(\tau-\one(y\leq 0)), \ y\in\bbr.
\end{equation*}
We view the stochastic process \eqref{e:X.n} as an objective function, and define 
a robustified version of the IPW estimate of the quantile $q(\tau)$   by
\begin{equation} \label{e:argmin.n}
  \hat q_n(\tau) = {\rm argmin}_t \, X_n(\tau;t),
 \end{equation} 
which, in case the argmin is not unique, is defined as the leftmost value of the argmin. The objective here is to establish a limit theorem for the
difference $\hat q_n(\tau)-q(\tau)$, suitably normalized and to consider joint convergence for several different values of $\tau$.

In order to  normalize $\hat q_n(\tau)-q(\tau)$ properly so that the resulting sequence has a nondegenerate limit distribution, we introduce a  sequence of constants $\{h_n\}$.   In view of Assumption \ref{ass:RV_prop_score}, the function $w^{-1} \mathbb{P}\bigl( e(X_1)\leq w^{-1})$ is regularly varying at infinity with index $-\gamma_1$. We define $h_n$ to be the $1/n$ ``quantile'' of this function, i.e.,
\begin{equation} \label{e:h_n}
h_n=\inf\bigl\{ w>0:\, w^{-1} \bbP\bigl( e(X_1)\leq w^{-1})\leq
1/n\bigr\}, \ n=1,2,\ldots,
\end{equation}
and since inverses of regularly varying functions are also regularly varying, it follows that $h_n$ is regularly varying with exponent
$1/\gamma_1$. That is $h_n=n^{1/\gamma_1}L(n)$ for some slowly varying function $L$. Our basic assumption for the threshold sequence $\{b_n\}$ is that there exists 
$\theta\in [0,\infty)$, such that 
\begin{equation} \label{e:moderate.ass}
    h_nb_n\to \theta \ \text{as} \ n\to\infty.
  \end{equation}
  For future reference we record here the straightforward limits 
from the definition of $h_n$, \eqref{e:moderate.ass}, and 
regular variation:
\begin{equation} \label{e:key.lim}
  \lim_{n\to\infty} nh_n^{-1}\bbP(e(X_1)\leq b_n) =\theta^{\gamma_1-1}, \ \ 
  \lim_{n\to\infty} nb_n\bbP(e(X_1)\leq b_n) =\theta^{\gamma_1}.   
\end{equation}
Note that the threshold $b_n$ is, approximately, $\theta/h_n$. This, of course, carries less information if $\theta=0$. The latter case can be thought of as  ``soft thresholding", and it requires somewhat different analysis in the sequel. In particular, it requires somewhat different assumptions, as we will see momentarily. 

We assume that there is a family 
${\mathcal L}(Y_1|e(X_1)=z), \, 0<z<1$
of regular conditional distributions of $Y_1$ given $e(X_1)$ such that
\begin{align} \label{e:small.p}
  &{\mathcal L}(Y_1|e(X_1)=z) \Rightarrow G_0 \ \text{in} \ \bbr, \ \text{as}
  \ z\to 0 \ \text{for some probability law $G_0$ on $\bbr$ }, \\
  \label{e:small.p1}
&\bbP \bigl( |Y_1- q(\tau)|\leq \vep \big| e(X_1)=z\bigr) \to 0 \
\text{as} \ \vep\to 0, \ \text{uniformly in $z$ near 0.}
\end{align}
 Here and elsewhere the double arrow $\Rightarrow$ denotes convergence in distribution. In the case $\theta=0$, we strengthen the assumption \eqref{e:small.p1} into 
\begin{align} \label{e:small.p-2}
  &\bbP \bigl( |Y_1- q(\tau)|\leq \vep \big| e(X_1)=z\bigr)=O(\varepsilon)
\ \text{as} \ \vep\to 0 \ \text{uniformly in $z$ near 0.}
\end{align}

A bit more formally, the   assumption \eqref{e:small.p1} 
means that
$$
\lim_{\vep\to 0, \, z\to 0} \, \sup_{z^\prime\leq z}\,
\bbP \bigl( |Y_1- q(\tau)|\leq \vep \big| e(X_1)=z^\prime\bigr) =0,
$$
with the appropriate modification for the assumption 
\eqref{e:small.p-2}.  The choice $b_n=0$ (no thresholding at all) is possible and, trivially, leads to $\theta=0$. In this special case the additional assumption  \eqref{e:small.p-2} is not needed.

The following theorem forms the main result of this section. 
 
 \begin{theorem} \label{th:Z1} Assume that  Assumptions \ref{ass:unconfoundedness} and \ref{ass:RV_prop_score} and conditions  
 \eqref{e:dens,tau}, \eqref{e:moderate.ass}, \eqref{e:small.p}
 and 
 \eqref{e:small.p1} (\eqref{e:small.p-2} if $\theta=0$) hold with $\gamma_1>1$, and for each $0<\tau<1$. Then for $\gamma_1\in (1,2)$ and with $\{h_n\}$ as specified in \eqref{e:h_n},  
\begin{equation} \label{e:key.fixed}
   \bigl[ (n/h_n)\bigl( \hat q_n(\tau) -q(\tau)\bigr), \,
   0<\tau<1\bigr] \Rightarrow \bigl( Z_\tau/g_Y(q(\tau)), \, 0<\tau<1\bigr)
   \end{equation}
   in finite-dimensional distributions. Here 
$\bigl(
   Z_\tau, \, 0<\tau<1\bigr)$ is an infinitely divisible
   stochastic process whose finite-dimensional distributions are
   specified in Proposition \ref{pr:z1}.   For $\gamma_1>2$ or $\gamma_1=2$ and $\E [1/e(X_1)]<\infty$, then \eqref{e:key.fixed} remains valid with $h_n=n^{1/2}$ and $\bigl(
   Z_\tau, \, 0<\tau<1\bigr)$  a Gaussian process as specified in Proposition \ref{pr:z1}, i.e.,
   \begin{equation*} \label{e:key.fixed2}
   \bigl[ n^{1/2}\bigl( \hat q_n(\tau) -q(\tau)\bigr), \,
   0<\tau<1\bigr] \Rightarrow \bigl( Z_\tau/g_Y(q(\tau)), \, 0<\tau<1\bigr).
   \end{equation*}
   
   \end{theorem}

\begin{remark}
\label{rem:Gaussian_case}
       The limit process in the theorem is  $\gamma_1$-stable if $1<\gamma_1<2$ and $\theta=0$, and non-stable infinitely divisible without a Gaussian component if $1<\gamma_1<2$ and $\theta>0$.  The assumption $\gamma_1>2$ or $\gamma_1=2$ and $\E(1/e(X_1))<\infty$ imply that $\operatorname{Var}(D_1/e(X_1))<\infty$, which is needed for the Gaussian limit.  A more delicate calculation also shows that a Gaussian limit is obtained without requiring $\E[1/e(X_1)]<\infty$, but in this case $h_n=n^{1/2}L(n)$, where $L(n)$ is a slowly varying function converging to infinity. 
   \end{remark}

\begin{remark} \label{rk:some.tau}

If the assumptions of Theorem 
\ref{th:Z1} hold only for a subset of $\tau\in (0,1)$, then the convergence result in the theorem will hold for $\tau$ restricted to that set. 
\end{remark}

\begin{remark} \label{rk:structure}
 Since $h_n$ is regularly varying with exponent $1/\gamma_1$, the convergence rate in Theorem \ref{th:Z1} is (up to a slowly varying function), $n^{-(1-1/\gamma_1)}$, which becomes slower as $\gamma_1$ decreases to 1. On the other hand, as $\gamma_1$ increases to 2, the scaling approaches the standard rate of $n^{-1/2}$ as expected with a Gaussian limit.

Overall, 
the structure underlying the proof of Theorem \ref{th:Z1} is as follows. 
Since $X_n(\tau,q(\tau))$ does not depend on $t$, we also have 
\begin{equation*} \label{e:main.split}
  \hat q_n(\tau) = {\rm argmin}_t \{X_n(\tau,t) - X_n(\tau,q(\tau))\}.  
\end{equation*}
Writing $u=(t-q(\tau))n/h_n$ or $t=q(\tau) +(h_n/n)u$, we have  
\begin{equation}
    \frac{n}{h_n}(\hat q_n(\tau)-q(\tau))={\rm argmin}_u \{X_n (\tau, q(\tau)+uh_n/n)- X_n(\tau,q(\tau))\}.  
\end{equation}
In other words, defining for each $\tau$ a sequence of continuous stochastic processes on $\R$ by
\begin{equation*}\label{eqn:zprocess}
 Z_{n,\tau}(u)= X_n (\tau, q(\tau)+uh_n/n)- X_n(\tau,q(\tau)), \, u\in\R \,, 
\end{equation*}
we have
\begin{equation} \label{e:exr.diff}
(n/h_n)\bigl( \hat q_n(\tau) -q(\tau)\bigr) = {\rm argmin}_{u}
Z_{n,\tau}(u), \ 0<\tau<1.
\end{equation}

 We will show that $ (n/h_n^2)Z_{n,\tau}(\cdot)$ converges weakly in $C(\R)$ from which an easy argument allows us to conclude that $(n/h_n)\bigl( \hat q_n(\tau)-q(\tau)\bigr)$ converges weakly to the argmin of the limiting process.  This is the same idea as used by \cite{davis1992m}, \cite{knight2002limiting}, \cite{goh2009nonstandard} and others on similar problems.   

In fact, we will prove the weak joint convergence for a finite number of distinct values of $\tau$. Specifically, we will prove that 
   for distinct $\tau_1,\ldots, \tau_k$ in $(0,1)$, we have, as $n\to\infty$,
 \begin{equation} \label{e:big.conv}
 \left\{\left( \frac{n}{h_n^2} Z_{n,\tau_j} (u),\, j=1,\ldots, k\right), u\in\R \right\}
 \Rightarrow \bigl\{\bigl( -uZ_{\tau_j}+\frac{u^2}{2}g_Y(q(\tau_j)), \, j=1,\ldots, k\bigr), u\in\R\bigr\}
\end{equation}
in $C^k(\R)$. Since the limiting processes have quadratic sample paths, the minimization, in the limit, is  explicit.
In order to obtain this weak convergence, we will use the identity,  (see \cite[p. 121]{koenker2005}). 
 \begin{equation} \label{e:identity}
  \rho_\tau(u-v)-\rho_\tau(u) = -v\bigl( \tau-\one(u\leq 0)\bigr) +
  \int_0^v \bigl( \one(u\leq s)-\one(u\leq 0)\bigr)\, ds,
\end{equation}
which is valid for all $u,v\in\bbr$.  It follows that
\begin{align} \label{e:split.Zn}
\frac{n}{h_n^2}Z_{n,\tau}(u) =& -uh_n^{-1} \sum_{i=1}^n \frac{D_i}{\max(e(X_i),b_n)}
  \bigl( \tau-\one(Y_i\leq q(\tau))\bigr) \\
\notag +& \frac{n}{h_n^2}\sum_{i=1}^n \frac{D_i}{\max(e(X_i),b_n)}\int_0^{uh_n/n}
          \bigl( \one(Y_i\leq q(\tau)+s)-\one(Y_i\leq q(\tau))\bigr)\,
          ds \\
\notag =:& Z^{(1)}_{n,\tau}(u)+ Z^{(2)}_{n,\tau}(u). 
\end{align} 
The first term, $Z^{(1)}_{n,\tau}(u)$, which is linear in $u$, will be shown to converge to 
$-uZ_{\tau}$, where $Z_\tau$ has an infinitely divisible distribution  when $\gamma_1\in (1,2)$, and Gaussian when $\gamma_1\geq 2$.  On the other hand, the second term will be shown to converge in probability to $(u^2/2)g_Y(q(\tau_j))$, from which finite dimensional convergence of  $\bigl((n/h_n^2)Z_{n,\tau}(u)\bigr)$ to $\bigl(-uZ_{\tau}+\frac{u^2}{2}g_Y(q(\tau))\bigr)$ follows easily.  Since the sample paths of $\bigl(Z_{n,\tau}(u)\bigr)$ are convex (in $u$), the finite dimensional convergence of $\bigl((n/h_n^2)Z_{n,\tau}(u)\bigr)$ can be upgraded to weak convergence in $C(\R)$  and the result will then follow by an application of the continuous mapping theorem.  The details are  relegated to the Appendix.  
\end{remark}

 \section{IPW estimation for intermediate quantiles} \label{sec:intermediate}

 It is often of interest to concentrate on quantiles close to 0 or to 1. 
 In this section we consider the so-called intermediate quantiles. Intermediate quantiles that are near 0 are defined as a sequence $\tau=\tau_n\to 0$, with
 \begin{equation} \label{e:interm.q}
   \tau_n\to 0 \ \ \text{and}\ \  n\tau_n\to\infty \ \ \text{as} \ n\to\infty\,, 
 \end{equation}
 with an appropriately modified definition for intermediate quantiles close to 1. Our results in this section are formulated for the case of intermediate quantiles close to 0, but the modifications needed to treat the intermediate quantiles close to 1 instead are obvious. 

 When $\tau_n$ approaches 0 too fast for the second requirement in 
 \eqref{e:interm.q} to hold, one typically speaks of extreme quantiles \cite{wang:li:he2012,deuber:lo:engelke:maathuis2023}, a situation not considered in this paper. We still estimate an intermediate quantile $\tau_n$ using the estimator $\hat q_n(\tau) =\hat q_n(\tau_n)$
 in \eqref{e:argmin.n}.  However, the specific assumptions  for this scenario need to be modified somewhat.

We still assume that each quantile is well defined, in the sense that
\eqref{e:quant.def} holds for each $\tau_n$, as well as the existence of a density of $Y$ at each $q(\tau_n)$ (see \eqref{e:dens,tau}).  Now we also require a smoothness condition.   Specifically, we assume there 
is a positive sequence $\{\theta_n\}$ with $\theta_n\to 0$ such that
\begin{equation} \label{e:theta.n}
g_Y(q(\tau_n)+y) = (1+o(1)) g_Y(q(\tau_n)) \ \text{uniformly in} \
|y|\leq\theta_n
\end{equation}
and  
\begin{equation} \label{e:hn.thetan}
  h_n/n=o(\theta_n), \, n\to\infty.
\end{equation}
Here $\{h_n\}$ is the  
sequence needed for proper normalization in the weak limit theorem below. It is now 
defined by
\begin{equation} \label{e:h_n.interm}
h_n=\frac{1}{g_Y(q(\tau_n))}\inf\bigl\{ w>0:\, w^{-1} \bbP\bigl( e(X_1)\leq w^{-1})\leq
1/(n\tau_n)\bigr\}, \ n=1,2,\ldots. 
\end{equation}
We can write
\begin{equation} \label{e:hn.rv}
  h_n=\frac{r_{1/\gamma_1}(n\tau_n)}{g_Y(q(\tau_n))},
\end{equation}
where the function $r_{1/\gamma_1}$ is regularly varying at infinity
with exponent $1/\gamma_1$. Furthermore, 
\begin{equation} \label{e:express.n}
n \sim \frac{g_Y(q(\tau_n))   h_n}{\tau_n} \bigl[ \bbP\bigl(e(X_1)\leq \bigl(
g_Y(q(\tau_n))   h_n\bigr)^{-1}\bigr)\bigr]^{-1}. \end{equation} 

Because we are considering the intermediate quantiles close to $0$, the assumption \eqref{e:small.p} is now  replaced by the
assumption
\begin{equation} \label{e:small.p.interm}
  \lim_{y\downarrow y_*  } \frac{\bbP(Y_1\leq y\ | \ e(X_1)=z)}{\bbP(Y_1\leq
  y)}=\beta\in (0,\infty)
\end{equation}
uniformly in $z$ near 0, where $y_*=\inf\{ y\in\bbr:\, \bbP(Y_1\leq y)>0\}\in [-\infty, \infty)$.

Finally, the truncation levels $\{b_n\}$ are now assumed to satisfy 
\begin{equation} \label{e:bn.ass.interm}
 g_Y(q(\tau_n))   h_nb_n\to \theta>0 \ \text{as} \ n\to\infty,
  \end{equation}
instead of \eqref{e:moderate.ass}. 
 
Most of the assumptions imposed in this section are analogous to those used in the previous section in the case of fixed quantiles. The assumption on existence of a sequence $\{\theta_n\}$ satisfying \eqref{e:theta.n} and \eqref{e:hn.thetan}, however, has no analogue in the case of fixed quantiles, and the reader may wonder how restrictive this assumption is. Specifically, does it impose restrictions on the intermediate quantile sequence $\{\tau_n\}$? To clarify this point we consider several examples of possible distributions of $Y$. 

\begin{example}
\label{ex:example1}
  Suppose that for some $\alpha>0, C>0$ and $y_0\in\bbr$,
\begin{equation*} 
 g_Y(y) = C\alpha e^{\alpha y}, \, y<y_0. 
\end{equation*}
In this case \eqref{e:theta.n} holds for any sequence $\{\theta_n\}$
converging to 0. Furthermore, for large $n$, 
$$
g_Y(q(\tau_n))= \alpha \tau_n.
$$
Therefore, from \eqref{e:hn.rv},  we get that
$$
h_n/n = \frac{1}{\alpha n\tau_n} r_{1/\gamma_1}(n\tau_n) \to 0, \ n\to\infty,
$$
since $\gamma_1>1$. Therefore, we can choose $\theta_n\to 0$ such that 
\eqref{e:hn.thetan} holds. That is, in this case \eqref{e:theta.n} and
\eqref{e:hn.thetan} do not impose any additional constraints on
$\{\tau_n\}$. 
\end{example}

\begin{example}
\label{ex:example2}
  Suppose that for some $p>0, C>0$ and $y_0\in\bbr$,
\begin{equation*} 
 g_Y(y) = Cp|y|^{-(p+1)}, \, y<y_0. 
\end{equation*}
For large $n$, 
$$
q(\tau_n) = -\text{const.}\, \tau_n^{-1/p}, \ \ 
g_Y(q(\tau_n))= \text{const.}\,  \tau_n^{(1+1/p)}.
$$
Therefore, for such $n$,
$$
h_n/n =  \text{const.} \frac{1}{n \tau_n^{(1+1/p)}}
r_{1/\gamma_1}(n\tau_n) = \text{const.}
\frac{r_{1/\gamma_1}(n\tau_n)}{n\tau_n} \tau_n^{-1/p}.
$$
Thus, assumming  $\tau_n=n^{-a}$ and $a<\frac{1-1/\gamma_1}{1-1/\gamma_1+1/p}$, we have $h_n/n\to 0$. In particular, we can take $\theta_n=\sqrt{h_n/n}$, which verifies \eqref{e:theta.n} and \eqref{e:hn.thetan}.
\end{example}

\begin{example}
\label{ex:example3}
  Suppose that for some $C>0$ and $y_0\in\bbr$,
\begin{equation*}
 g_Y(y) = C|y|e^{-y^2/2},  \, y<y_0. 
\end{equation*}
In that case \eqref{e:theta.n} holds for any sequence
$\{\theta_n\}$ such that
$$
\theta_nq(\tau_n)\to 0.
$$
However, for large $n$, 
$$
q(\tau_n) = -\bigl( 2\log (C/\tau_n)\bigr)^{1/2}, \ \ 
g_Y(q(\tau_n))= \text{const.} \, \tau_n\bigl( \log (C/\tau_n)\bigr)^{1/2}. 
$$
In particular, \eqref{e:theta.n} holds for any sequence
$(\theta_n)$ such that
\begin{equation} \label{e:theta.G}
\theta_n \bigl( \log (1/\tau_n)\bigr)^{1/2}\to 0.
\end{equation} 
Furthermore, for large $n$, 
$$
h_n/n =  \text{const.} \,\frac{1}{n \tau_n}\bigl( \log (C/\tau_n)\bigr)^{-1/2}
r_{1/\gamma_1}(n\tau_n)  = o\bigl[ \bigl( \log (C/\tau_n)\bigr)^{-1/2}\bigr]
$$
since $\gamma_1>1$. Therefore, we can choose $\theta_n$ satisfying
\eqref{e:theta.G} such that  
\eqref{e:hn.thetan} holds. Again, in this case \eqref{e:theta.n} and
\eqref{e:hn.thetan} do not introduce any additional constraints on
$\{\tau_n\}$. 
\end{example}

The following result is the analogue of Theorem \ref{th:Z1} in the case of intermediate quantiles. A process-level version is less natural now, so we state this theorem for a single sequence of intermediate quantiles. 

 \begin{theorem} \label{t:quant.lim.interm}
   Suppose that  Assumptions \ref{ass:unconfoundedness} and \ref{ass:RV_prop_score} and conditions \eqref{e:interm.q}, \eqref{e:theta.n}, \eqref{e:hn.thetan}, \eqref{e:small.p.interm} and 
   \eqref{e:bn.ass.interm} 
   hold. Then  
   $$
   (n/h_n)\bigl( \hat q_n(\tau_n) -q(\tau_n)\bigr) \Rightarrow Z, 
   $$
   where $Z$ is the infinitely divisible random variable described in
   Proposition \ref{pr:Z1.inter}. 
  \end{theorem}

  \begin{remark} \label{rk:interm.structure}
The proof of Theorem \ref{t:quant.lim.interm} has a structure similar to that of Theorem \ref{th:Z1} and described in Remark \ref{rk:structure}. We use now $\tau=\tau_n$  and the equation \eqref{e:split.Zn} is now changed (with a new normalization) to 
\begin{align} \label{e:split.Zn.mod}
\frac{n}{g_Y(q(\tau_n))h_n^2}&Z_{n,\tau_n}(u) = -u\bigl(g_Y(q(\tau_n))h_n\bigr)^{-1} \sum_{i=1}^n  \frac{D_i}{\max(e(X_i),b_n)} 
  \bigl( \tau_n-\one(Y_i\leq q(\tau_n))\bigr) \\
\notag +& \frac{n}{g_Y(q(\tau_n))h_n^2}\sum_{i=1}^n \frac{D_i}{\max(e(X_i),b_n)}\int_0^{uh_n/n}
          \bigl( \one(Y_i\leq q(\tau_n)+s)-\one(Y_i\leq q(\tau_n))\bigr)\,
          ds \\
\notag =:& Z^{(1)}_{n}(u)+ Z^{(2)}_{n}(u). 
\end{align}
Once again, we will show that $Z^{(1)}_{n}(u)$ converges to $-uZ$  and that $Z^{(2)}_{n}(u)$ converges in probability to $u^2/2$.  This is the content of Propositions  \ref{prop:cip.inter} and
\ref{pr:Z1.inter} in the Appendix.
\end{remark}

\section{Quantile treatment effect process}  \label{sec:qte}

The results of Sections \ref{sec:IPW-quantile} 
and \ref{sec:intermediate} can be 
viewed as leading to comparing the effect of two different treatments on quantiles of interest of the clinical outcomes  (see \cite{firpo2007}).
In this section we take the final step necessary to that end.   We switch back to the notation of Section \ref{sec:notation}, so that   $Y_i(1)$ and $Y_i(0)$ denote the  potential outcomes under treatment 1 and treatment 0, respectively. It is the quantiles of the distributions of these random variables that we wish to compare. 

We apply the robustified version of the IPW estimator \eqref{e:argmin.n} to the two quantiles above. That is, we let 
\begin{equation}
\label{eq:IPW_quant}
    \begin{split}
\hat{q}_1(\tau)&:=\mbox{argmin}_{t\in\mathbb{R}} \sum_{i=1}^n\frac{D_i}{ \max(e(X_i),b_n^{(1)})}\rho_\tau(Y_i(1)-t)  \\
\hat{q}_0(\tau)&:=\mbox{argmin}_{t\in\mathbb{R}} \sum_{i=1}^n\frac{1-D_i}{\max(1- e(X_i),b_n^{(0)})}\rho_\tau(Y_i(0)-t) \,.
    \end{split}
\end{equation}
for two threshold sequences $\{b_n^{(1)}\}$ and $\{b_n^{(0)}\}$, corresponding to the two treatments.  The two sequences may or may not be the same, or even similar. The choice of the sequences is directly related to the regular variation properties of the distribution of the propensity score near 0 and near 1. The exact assumptions are stated in the Theorem \ref{thm:joint} below. 
The interesting feature of this theorem is that, under certain conditions, the estimated quantiles $\hat{q}_1(\tau)$ and $\hat{q}_0(\tau)$ are asymptotically independent. The reason for this phenomenon is that, in Theorem \ref{thm:joint}, the joint distribution of properly centered and normalized quantiles is, asymptotically, infinitely divisible. Under the assumptions of the theorem the Gaussian component of the infinitely divisible distribution disappears, and no part of the Poissonian component affects both estimated quantiles at the same time. This is established  both in the case of fixed quantiles  and  in the case of intermediate quantiles. The result requires the conditions of
Theorem  \ref{th:Z1} in the case of fixed quantiles or the conditions of Theorem \ref{t:quant.lim.interm} in the case of intermediate quantiles. In this case the assumptions need to apply to both extreme points of the propensity scores, i.e., to the situation when $e(X_i)$ is near 0 and to the situation when $e(X_i)$ is near 1. All of our assumptions have been previously stated for the situation when $e(X_i)$ is near 0. The corresponding versions near 1 are entirely analogous; sometimes it only requires switching from $e(X_i)$ to $1-e(X_i)$.

In the statement of Theorem \ref{thm:joint} below the sequences $\{h_n^{(1)}\}$ and $\{h_n^{(0)}\}$ are defined by \eqref{e:h_n} and by \eqref{e:h_n.interm}
applied to the laws of $e(X_i)$ and of $1-e(X_i)$ correspondingly.

\begin{theorem}\label{thm:joint}

Suppose that  Assumption  \ref{ass:unconfoundedness} is satisfied and that the propensity scores $e(X_i)$ and $1-e(X_i)$ satisfy Assumption 2.2, with regular variation indices, $\gamma_1$ and $\gamma_0$, respectively,  with $\gamma_1,\gamma_0\in (1,2).$  
 
   \smallskip 
    \noindent 1. (Fixed quantiles) Further assume that the conditions   
\eqref{e:dens,tau}, \eqref{e:moderate.ass}, \eqref{e:small.p}
 and 
 \eqref{e:small.p1}    
 hold both with $Y_1$ equal to $Y(1)$ and with $Y_1$ equal to $Y(0)$. Let $\theta_1$
and $\theta_0$ be the corresponding limits in \eqref{e:moderate.ass}, and replace the assumption \eqref{e:small.p1} with the assumption 
 \eqref{e:small.p-2} if the corresponding $\theta$
vanishes.  Then
\begin{align*}
 &\bigl[\bigl((n/h_n^{(1)})( \hat q_1(\tau) -q_1(\tau)),(n/h_n^{(0)})( \hat q_0(\tau) -q_0(\tau))\bigr), \, 0<\tau<1\bigr] \\
 \Rightarrow&  
 ~~~ \bigl[\bigl((Z_{1,\tau}/g_{Y(1)}(q_1(\tau))),    
  (Z_{0,\tau}/g_{Y(0)}(q_0(\tau)))\bigr), \, 0<\tau<1\bigr]   
  \end{align*}
in finite-dimensional distributions, 
where the limiting processes are independent and as described in Theorem \ref{th:Z1}. 

\smallskip 
    \noindent 2. (Intermediate quantiles) \ Assume conditions  \eqref{e:interm.q}, \eqref{e:theta.n}, \eqref{e:hn.thetan},\eqref{e:small.p.interm} and 
\eqref{e:bn.ass.interm} hold with $Y_1$ equal to $Y(1)$ and with $Y_1$ equal to $Y(0)$. 
Let $\theta_1$
and $\theta_0$ be the corresponding limits in \eqref{e:bn.ass.interm}. 
Then 
\begin{align*}
 &\bigl((n/h_n^{(1)})( \hat q_1(\tau_n) -q_1(\tau_n)),(n/h_n^{(0)})( \hat q_0(\tau_n) -q_0(\tau_n))\bigr) 
 \Rightarrow (Z_1,Z_0),
  \end{align*}
  where the limiting random variables are independent and as described in Theorem \ref{t:quant.lim.interm}. 
\end{theorem}

In the potential outcomes framework for quantiles, the object of interest is the quantile treatment effect (QTE), defined as the difference between 
$q_1(\tau)$ and $ q_0(\tau)$. Therefore the QTE is naturally estimated by $\hat\delta(\tau)=\hat q_1(\tau)-\hat q_0(\tau)$. 
The limit distribution of this estimator can be found directly from Theorem \ref{thm:joint}. It is important to observe that the rate of convergence of $\hat q_1(\tau)$ and $\hat q_0(\tau)$ to their population counterparts strongly depends on the behavior of the distribution of the propensity scores near the endpoints of the interval $(0,1)$. In particular, there is strong dependence on the exponents $\gamma_1$ and $\gamma_0$ in the regular variation assumption in Assumption \ref{ass:RV_prop_score}. Unless this behavior near the two endpoints is balanced, the deviations of one of the two quantile estimators may dominate the other, which will, in turn, affect the behavior of the estimated QTE. The balance condition is
\begin{equation} \label{eq:balance}
  \lim_{t\to 0} \frac{\bbP(1-e(X_1)\le t)}{\bbP(e(X_1)\le t)} = c\,, 
\end{equation}
where $c\in (0,\infty)$. If this condition holds, then $\gamma_1=\gamma_0$. 

The following corollary  describes the limit behavior of the estimate of the QTE.

\begin{corollary} \label{cor:qte}
For each of the two parts below assume that the conditions of the corresponding part in Theorem \ref{thm:joint} hold. In addition, 
assume that \eqref{eq:balance} holds, where we also allow $c=0$. 

\smallskip 
\noindent 1.  (Fixed quantiles) 
\begin{align}\label{eq:qte1}
    &\bigl[ \bigl( n/h_n^{(1)}\bigr)\left( \hat q_1(\tau)-\hat q_0(\tau) -(q_1(\tau)-q_0(\tau))\right), \, 0<\tau<1\bigr] \\
    \notag \Rightarrow  
&\bigl[ Z_{1,\tau}/g_{Y(1)}(q_1(\tau))    
  -c^{1/\gamma_1}Z_{0,\tau}/g_{Y(0)}(q_0(\tau)), \, 0<\tau<1\bigr]
\end{align}
in finite-dimensional distributions. 

\smallskip 
\noindent 2.   (Intermediate quantiles)  Assume that 
\begin{equation} \label{eq:g-ratio}
    \lim_{n\to\infty}\frac{g_{Y(1)}(q_1(\tau_n))}{g_{Y(0)}(q_0(\tau_n))}=\rho \geq 0.
\end{equation}
Then
\begin{equation} \label{eq:qte2}
\bigl( n/h_n^{(1)}\bigr)\left( \hat q_1(\tau_n)-\hat q_0(\tau_n) -(q_1(\tau_n)-q_0(\tau_n))\right)\Rightarrow  
Z_{1} 
  -c^{1/\gamma_1}\rho Z_{0}\,.
\end{equation}
\end{corollary}

\begin{proof} For Part 1, it follows from the definition of $h_n^{(1)}$ and $h_n^{(0)}$ and assumption \eqref{eq:balance} that $h_n^{(0)}/h_n^{(1)}\to c^{1/\gamma_1}$, and hence  \eqref{eq:qte1} follows immediately from Theorem \ref{thm:joint}. Note that when $c=0$, it may or may not be true that $\gamma_0=\gamma_1$, but it does not affect the result. 

For Part 2, the argument is similar.  Assumptions \eqref{eq:balance} and \eqref{eq:g-ratio} imply that $h_n^{(0)}/h_n^{(1)}\to c^{1/\gamma_1}\rho$, and \eqref{eq:qte2} again follows from Theorem \ref{thm:joint}. 
\end{proof}

    The case $c=\infty$ in \eqref{eq:balance} is also possible.
    For fixed quantiles, the appropriate normalization is then $n/h_n^{(0)}$, and the limit consists only of the treatment $D=0$ component. For intermediate quantiles, additional information on the relative behavior of  $\frac{g_{Y(1)}(q_1(\tau_n))}{g_{Y(0)}(q_0(\tau_n))}=\rho_n $ is required to determine the appropriate normalization.

\section{Practical considerations}

Our theoretical analysis relies on a construction that is not directly applicable in practice because it assumes a known propensity score and truncation parameter. In this section we present some intuitive ways to address both of  these issues as well as the question of debiasing the truncated IPW estimator.

A practical way to address the problem of estimating the unknown propensity score is to estimate it with a parametric model. If such a parametric model is well specified and satisfies some regularity conditions similar to Assumptions 3 and 4 in \cite{ma:wang2020}, our theoretical results will still apply to quantile IPW estimators with estimated propensity scores. We will see in our numerical studies that this approach works well in moderate sample size scenarios assuming a logistic regression model for the propensity score function $e(x)=\mathbb{E}[D_i|X_i=x]=e(x, \beta)$. We estimate $\beta$ by maximum likelihood and work with the estimated propensity scores $e(X_i, \hat\beta)$
in our IPW constructions.

\subsection{Debiasing the truncated IPW estimator}

The limit specified in Theorem \ref{th:Z1} does not have zero mean unless $\theta=0$. Therefore the IPW quantile estimator is asymptotically biased and the bias depends on the truncation parameter $b_n$. We note that given an observed sample, the standard  IPW estimator and the naive quantile estimator are special cases of our truncated IPW estimator. Indeed, if one takes $b_n=\min_{1\leq i \leq n}e(X_i,\hat\beta)$ one recovers the standard IPW estimator, while $b_n=\max_{1\leq i \leq n}e(X_i,\hat\beta)$ recovers the naive quantile estimator that assigns equal weight to all  observations. Furthermore, one can compute an arbitrarily dense set of intermediate estimators that lie in between the low bias but high variance standard IPW  estimator and the naive quantile estimator which has high bias but low variance. We suggest a simple data-driven bias correction strategy based on this observation. 

 More specifically, let $\hat{q}_{1}(\tau, b_n)$ be the estimated $\tau$-quantile $Y_i(1)$ obtained with $b_n\in G_L=\{b_{n,1},b_{n,2}\dots,b_{n,L}\}$, where $G_L$ is a grid of $L$ truncation values that increase exponentially from $b_{n,1}=\min_{1\leq i \leq n}e(X_i,\hat\beta)$ to $b_{n,L}=\max_{1\leq i \leq n}e(X_i,\hat\beta)$. That is, we take an equispaced grid between $\log  b_{n,1}$ and $\log b_{n,L}$ and then exponentiate. The explicit grid values are then given by $b_{n,j}=b_{n,1}h^{j-1},~~j=1,\ldots,L$ where $h=\exp\{\frac{1}{L-1}(\log  b_{n,L}- \log b_{n,1})\}$.
The goal of this is to explore a denser grid of values close to $b_{n,1}$ which corresponds to the IPW estimators that do not perform truncation.

\begin{figure}[h!]
    \centering
          \includegraphics[scale=0.35]{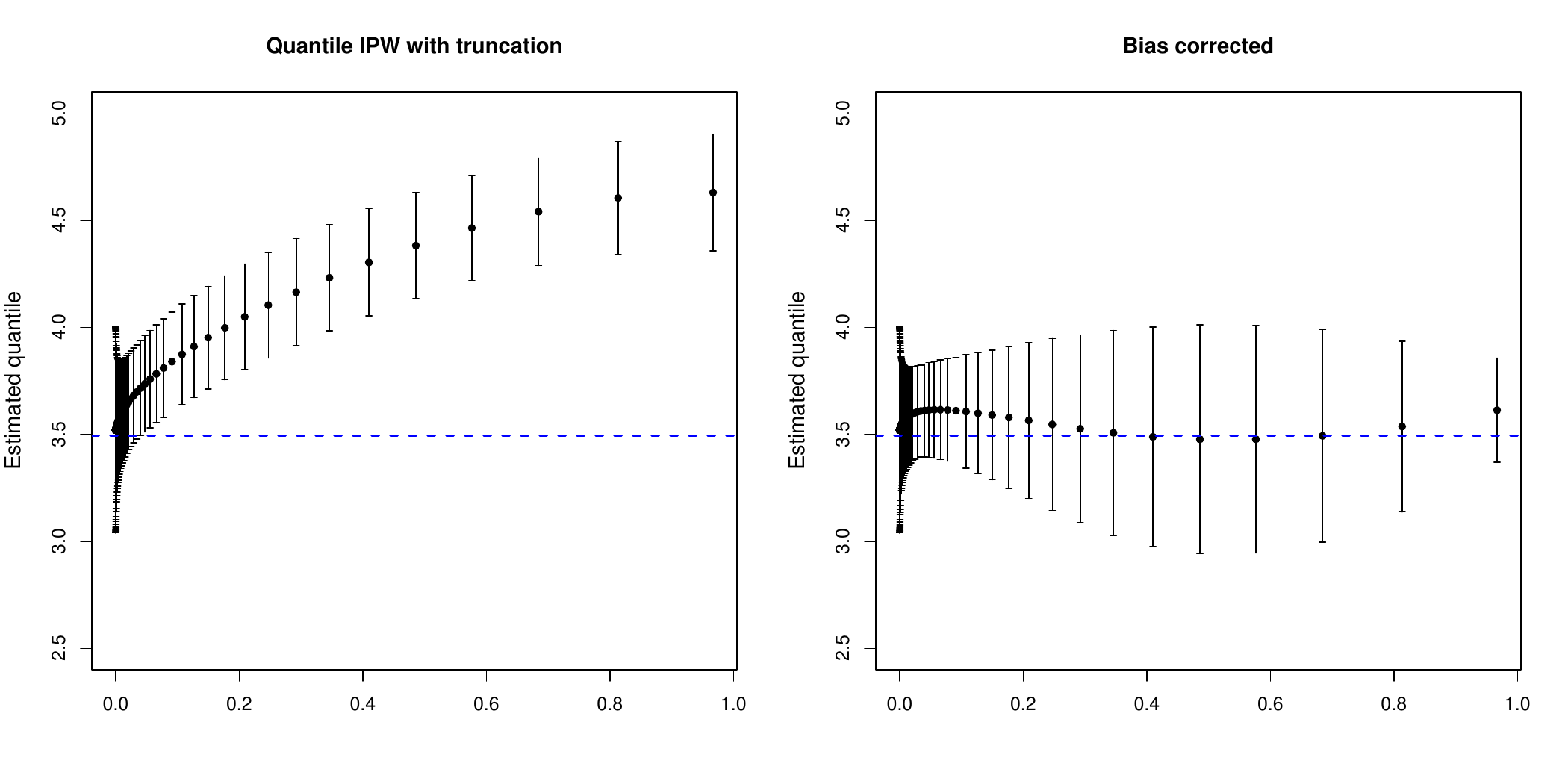}
          \renewcommand{\baselinestretch}{.8}
          \caption{ {\footnotesize The left plot shows the performance of the estimator $\hat q_1(\tau,b_n)$ an exponentially increasing grid  from $b_{n,1}=\min_{1\leq i \leq n}e(X_i,\hat\beta)$ to $b_{n,L}=\max_{1\leq i \leq n}e(X_i,\hat\beta)$ and grid size $L=100$. The plot reports the mean of $\hat q_1(\tau, b_{n,l})$ over 1000 replications plus/minus 2  standard deviations.  The  right plot shows the results over the same 1000 replications for the additive bias correction estimator \eqref{eq:debiased_q} based on polynomials of order $k=2$. The blue line reports the target population  of approximately $3.5$. } }
    \label{fig:QTT_bias_correction}
\end{figure}

Guided by the form of the asymptotic mean of the limiting distribution (see \eqref{e:ass.bias}, which is a power of $\theta$), we propose a simple additive bias correction approach that first fits 
\begin{equation*}
(\hat\vartheta_0,\hat\vartheta_1,\dots,\hat\vartheta_k)=\underset{\vartheta_0,\vartheta_1,\dots,\vartheta_k}{\argmin}\,\sum_{\ell=1}^L(\hat{q}_{1}(\tau,b_{n,\ell})-\sum_{j=0}^k\vartheta_jb_{n,\ell}^j)^2
\end{equation*}
and then obtains the bias corrected estimators
\begin{equation}
    \label{eq:debiased_q}
  \hat{q}_{1}^{\mathsf{bc}}(\tau,b_{n,\ell})=  \hat{q}_{1}(\tau,b_{n,\ell})-\sum_{j=1}^k\hat\vartheta_jb_{n,\ell}^j.
\end{equation}
To illustrate this procedure we consider model (a)  described in Section \ref{sec:numexp}.  In these experiments we take $k=2$, i.e., a quadratic bias correction which worked very well. Figure \ref{fig:QTT_bias_correction} illustrates the performance of this method over 1000 replications   with $n=2000$, $\gamma_1=1.5$, $\tau=0.9$, and  $L=100$.  The true 0.9-quantile is approximately 3.5.

\subsection{Data-driven selection of the truncation parameter}
\label{sec:data_driven_selection}
We propose a practical data-driven method for selecting $b_n^*\in G_L$ using the  bias-corrected estimators \eqref{eq:debiased_q} and an intuitive bootstrap subsampling approach. Given a sample $(Y_i,D_i,X_i)$, $i=1,\dots,n$ and  a grid $G_L$, we suggest the following simple steps for finding some $b_n^*\in G_L$.
\begin{enumerate}
    \item Compute $n_B$ bootstrap subsamples of size $m\leq n$.   The subsamples are drawn with replacement.
    \item Compute the bias-corrected estimators $ \hat{q}_{1,j}^{\mathsf{bc}}(\tau,b_{n,\ell})$ for each of the $j\in [n_B]$ subsamples. For this one needs to recompute the estimated propensity scores, the estimated truncated IPW quantiles and the bias correction as in \eqref{eq:debiased_q}.
    \item Define the bootstrap bias-corrected estimators $\bar{\hat{q}}_{1}^{\mathsf{bc}}(\tau,b_{n,\ell})=\frac{1}{n_B}\sum_{j=1}^{n_B}\hat{q}_{1,j}^{\mathsf{bc}}(\tau,b_{n,\ell})$ and find the $b_n$ that minimizes the estimated MSE as
    \begin{equation*}
        b_n^*=\argmin_{b_{n,\ell}\in G_L}\frac{1}{n_B}\sum_{j=1}^{n_B}(\hat{q}_{1,j}^{\mathsf{bc}}(\tau,b_{n,\ell})-\bar{\hat{q}}_{1}^{\mathsf{bc}}(\tau,b_{n,\ell}))^2=\argmin_{b_{n,\ell}\in G_L}\widehat{\textrm{MSE}}(b_{n,\ell}).
    \end{equation*}
\end{enumerate}
A straightforward modification of the above algorithm will also give a practical way to select the truncation level for the quantile treatment to the untreated and the quantile treatment effect estimators. Figure \ref{fig:sub_QTT_dist} illustrates the shape of a bootstrap MSE validation curve that is minimized by our data driven criterion. The plot also reports the corresponding bootstrap subsampled distribution of the corresponding  debiased quantile estimator for the same one realization of model (a) in Section \ref{sec:numexp}.  

We note that it is well known that the standard bootstrap fails to give valid inference for the sample mean when the variance is infinite \citep{athreya1987,knight1989} and subsampling is a well-established procedure that can overcome this shortcoming \citep{arcones:gine1989,politis:romano1994,athreya:lahiri:wu1998,romano:wolf1999}. Subsampling was also considered in the context of extreme inverse probability weighting in \cite{ma:wang2020,heiler:kazak2021}.   The need for the use of a bootstrapped subsample in our context is discussed further in Appendix B (see Section \ref{subsample}).

\begin{figure}[h!]
    \centering
                   \includegraphics[scale=0.35]{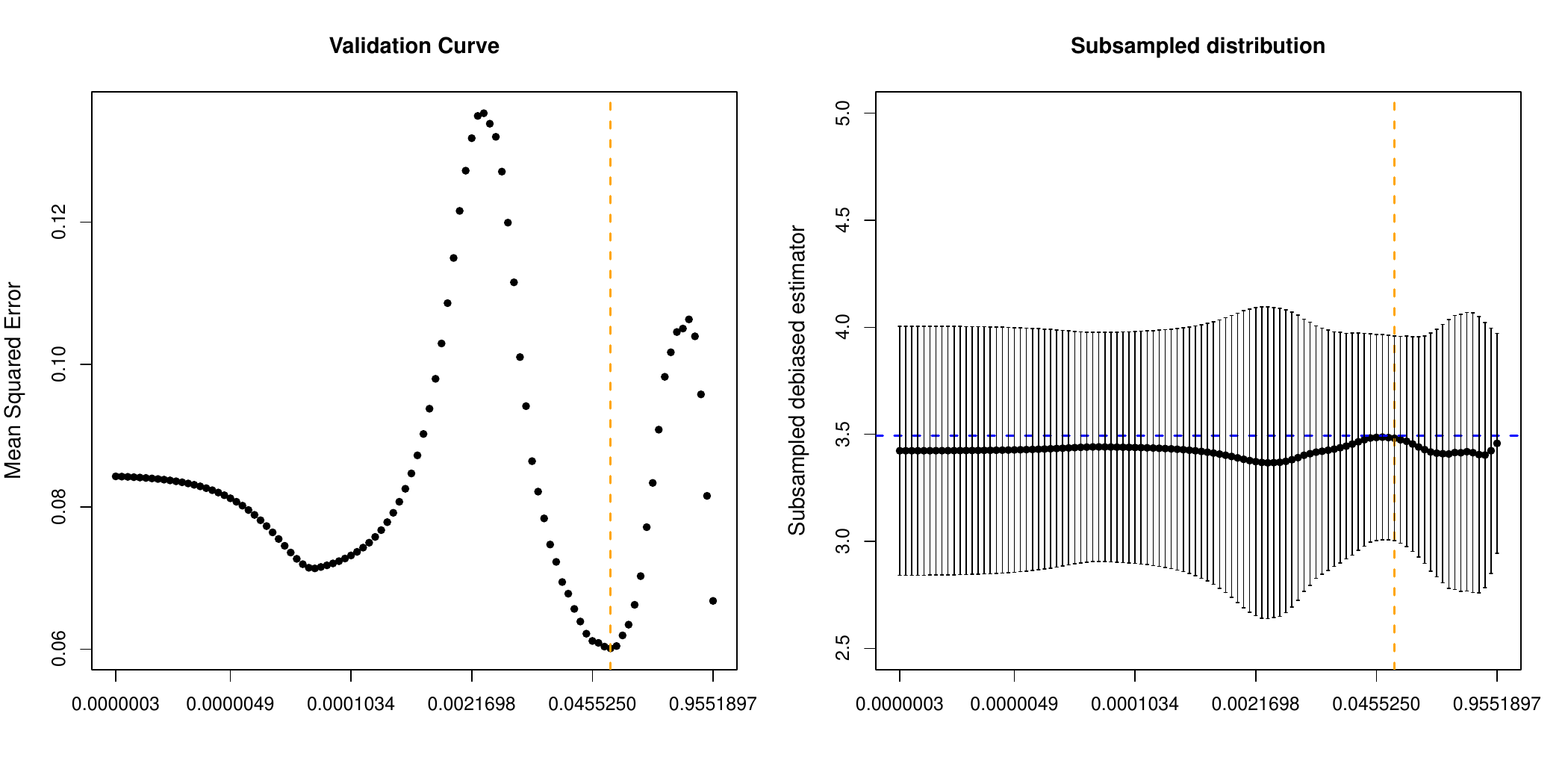}

   \renewcommand{\baselinestretch}{.8}
              \caption{ {\footnotesize The left plot shows an estimated $\widehat{\textrm{MSE}}(b_{n,\ell})$ obtained by $n_B=1000$ bootstrap  subsamples of size $m=1000$ for one realization of model $(a)$ with $n=2000$ and grid length $L=100$. The vertical orange line indicates the minimizing truncation value. The corresponding subsampling distribution of $\{\hat{q}_{1,j}^{\mathsf{bc}}(\tau,b_{n,\ell})\}_{j=1}^{1000}$ is illustrated on  the right plot. More precisely, the plot reports the  means $\bar{\hat{q}}_{1}^{\mathsf{bc}}(\tau,b_{n,\ell})$ plus/minus two  standard deviations. The horizontal blue line is the target population value, while the vertical orange one indicates the subsampling distribution of the selected debiased estimator.}}
    \label{fig:sub_QTT_dist}
\end{figure}

\begin{figure}[h]
    \centering
    \begin{minipage}{0.24\textwidth}
        \centering
        \includegraphics[width=\textwidth]{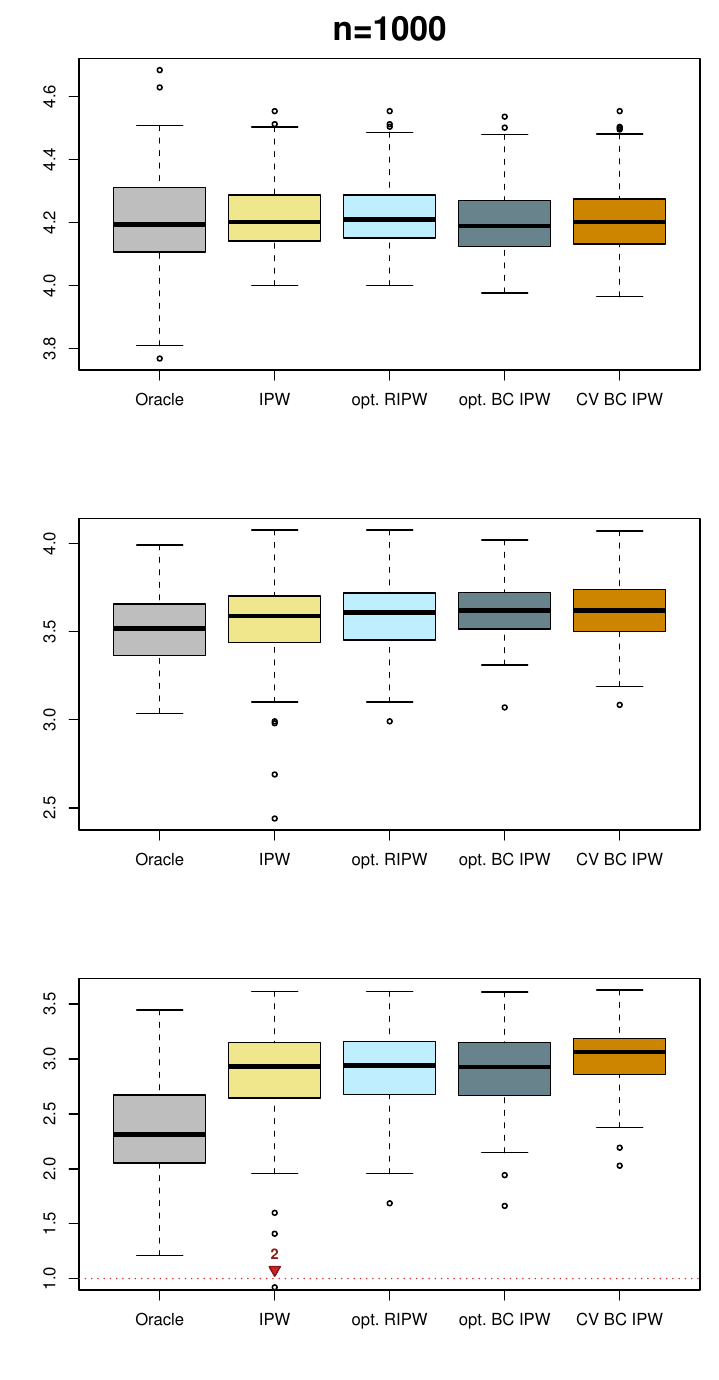}
    \end{minipage}
    \hfill
    \begin{minipage}{0.24\textwidth}
        \centering
        \includegraphics[width=\textwidth]{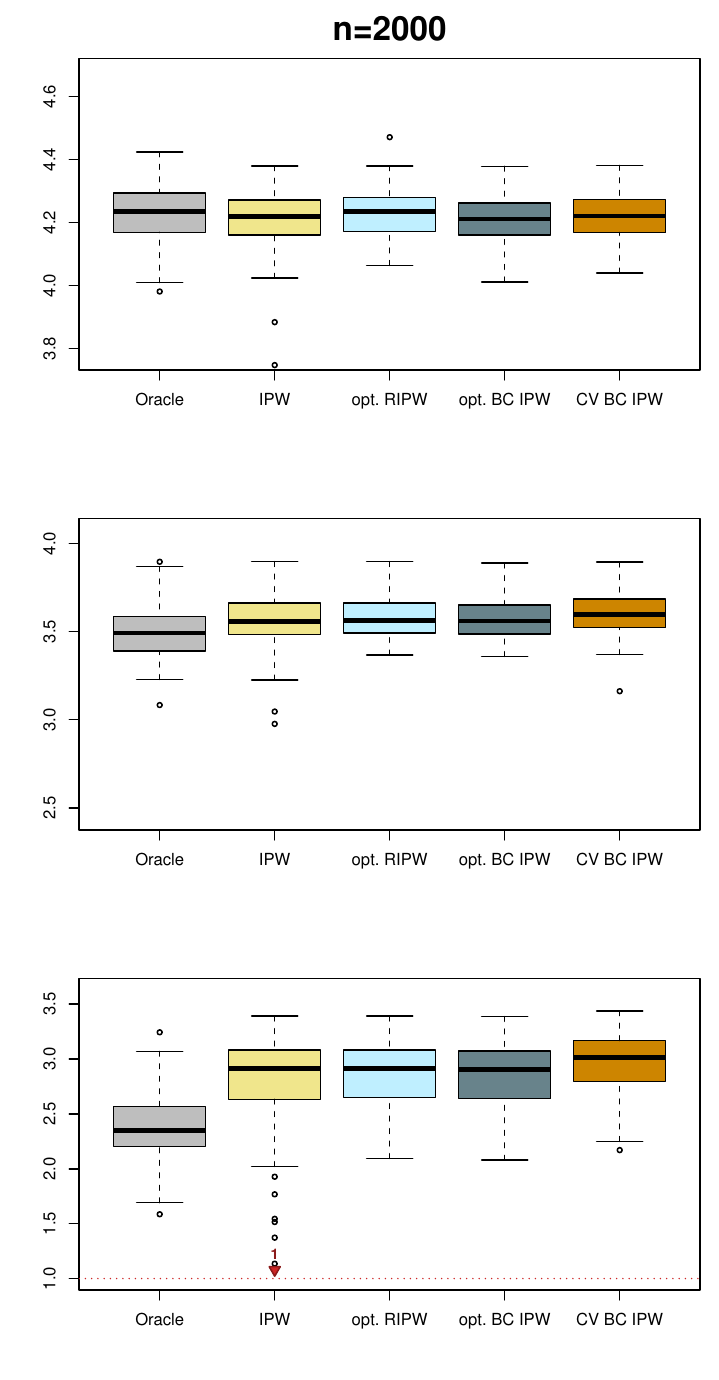}
    \end{minipage}
    \hfill
    \begin{minipage}{0.24\textwidth}
        \centering
        \includegraphics[width=\textwidth]{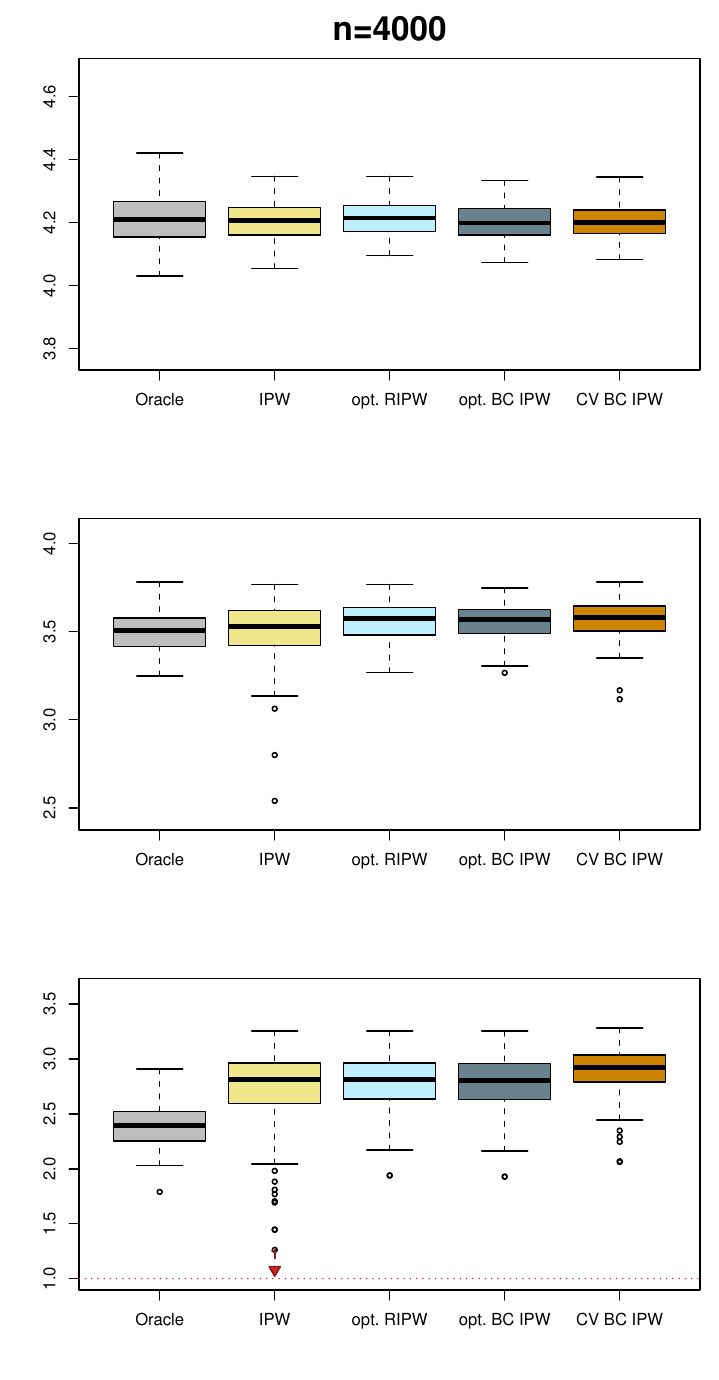}
    \end{minipage}
    \hfill
    \begin{minipage}{0.24\textwidth}
        \centering
        \includegraphics[width=\textwidth]{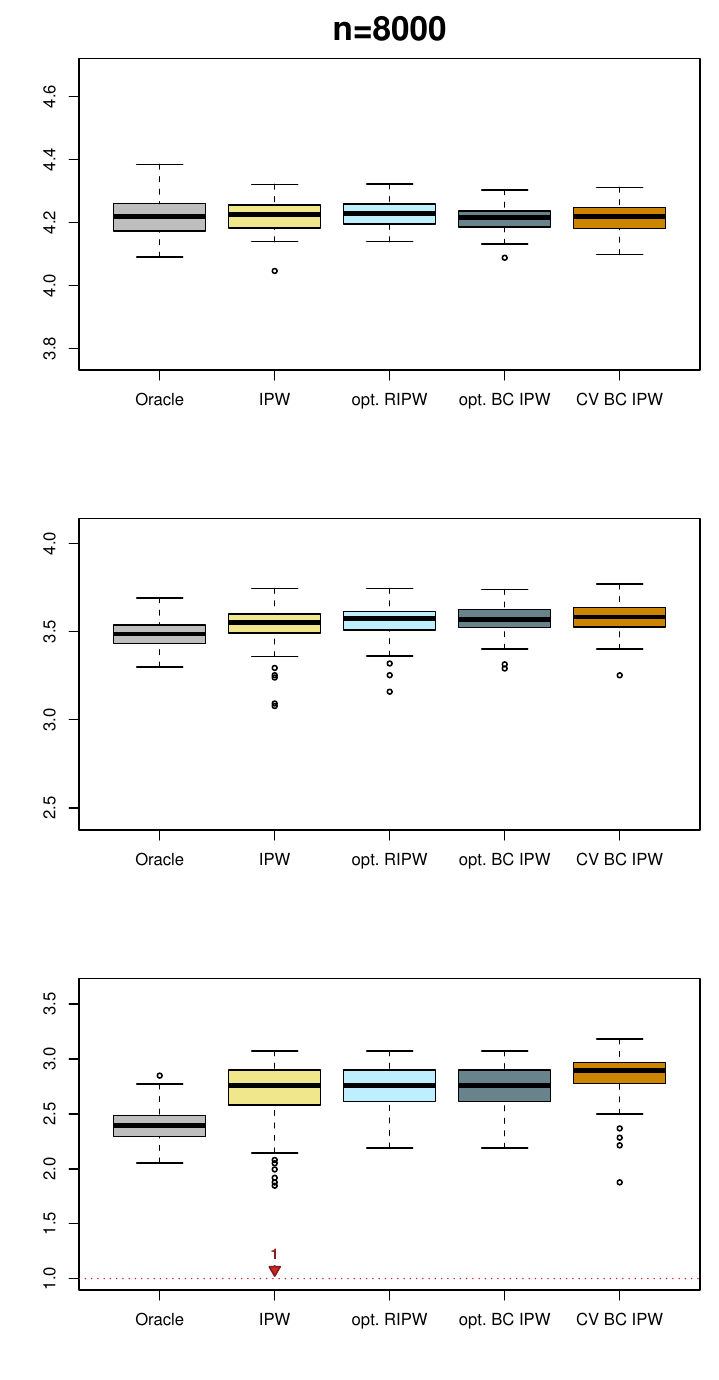}
    \end{minipage}
    \renewcommand{\baselinestretch}{.8}
    \caption{ \footnotesize
    Estimators of the $0.9$-quantile of $Y(1)$ based on model $(a)$ i.e., Gaussian responses. The  tails of $1/e(X_i)$ become heavier from top to bottom, while $n$ increases from left to right. 
           The boxplots show 100 realizations of an oracle estimator that sees  a random subset of  $\lfloor n \mathbb E[e(X_i)]\rfloor$ of the  $Y_i(1)$'s (light gray), the standard IPW estimator (yellow), the truncated IPW estimator that achieves the best empirical MSE performance (light blue), the debiased truncated IPW estimator that achieves the best empirical MSE performance (dark blue)  and the debiased truncated IPW estimator with our proposed  data-driven choice of the truncation parameter (dark orange).  In the last row a  triangle and number at this boundary marks an estimator with a corresponding realizations below the displayed range.
           }\label{fig:Q1_Gaussian}
\end{figure}
\subsection{Numerical experiments} \label{sec:numexp}

{ We consider a setting inspired by the example given in Section 3.3 of \cite{crump:hotz:imbens:mitnik2009}. The propensity score is assumed to follow a beta distribution with parameters $\gamma_1-1$  and $\gamma_0-1$, so the mean is $(\gamma_1-1)/(\gamma_1+\gamma_0-2)\in[0,1]$. 
We want to illustrate the performance of various IPW estimators when $e(X_1)\sim Beta(\gamma_1-1,\gamma_0-1)$. In particular, this implies that the standard overlap assumptions are violated. For simplicity, 
we will consider a univariate covariate and a well-specified logistic regression model such that $\{D_i|X_i\}{\sim}Bernoulli(e(X_i,\beta))$, where $\{e(X_i,\beta)\}=\{e^{\beta X_i}/(e^{\beta X_i}+1)\}\overset{iid}{\sim}Beta(\gamma_1-1,\gamma_0-1)$ and $\beta=1$. In the simulations we consider $\gamma_1 \in\{2,1.5, 1.2\}$ and $\gamma_0=2$. 
We only observe the response when $D_i=1$. With $h(x)=1/(1+e^{-30-x})$, we consider the  models:} 
\begin{enumerate}[(a)]
        \item $Y_i=1+X_ih(X_i)+Z_i$ and  $\{Z_i\}\overset{iid}{\sim}N(0,1)$.
        \item $Y_i=Z_ie^{X_ih(X_i)}$, where $\{Z_i\}$ are i.i.d. Fr\'echet with shape parameter  $\alpha=2$.  
\end{enumerate}

\begin{figure}[h]
    \centering
    \renewcommand{\baselinestretch}{1}
    \begin{minipage}{0.24\textwidth}
        \centering
        \includegraphics[width=\textwidth]{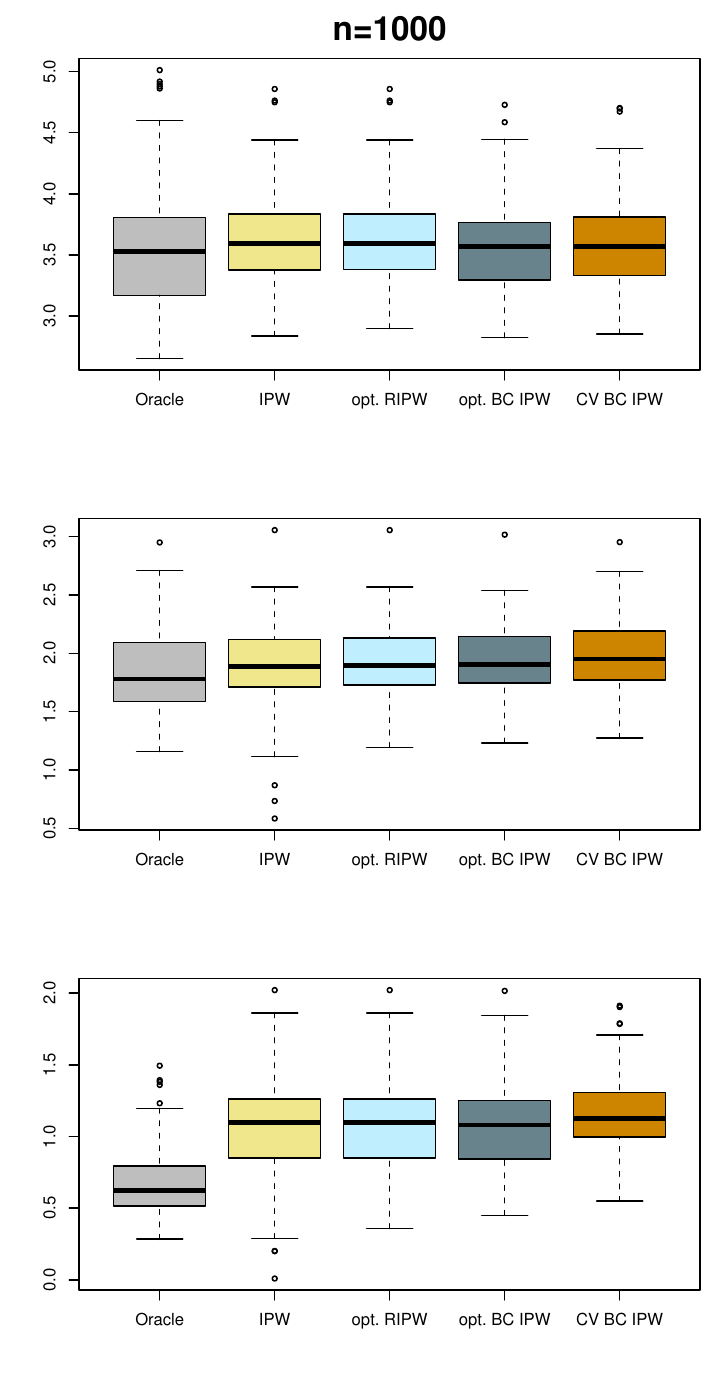}
    \end{minipage}
    \hfill
    \begin{minipage}{0.24\textwidth}
        \centering
        \includegraphics[width=\textwidth]{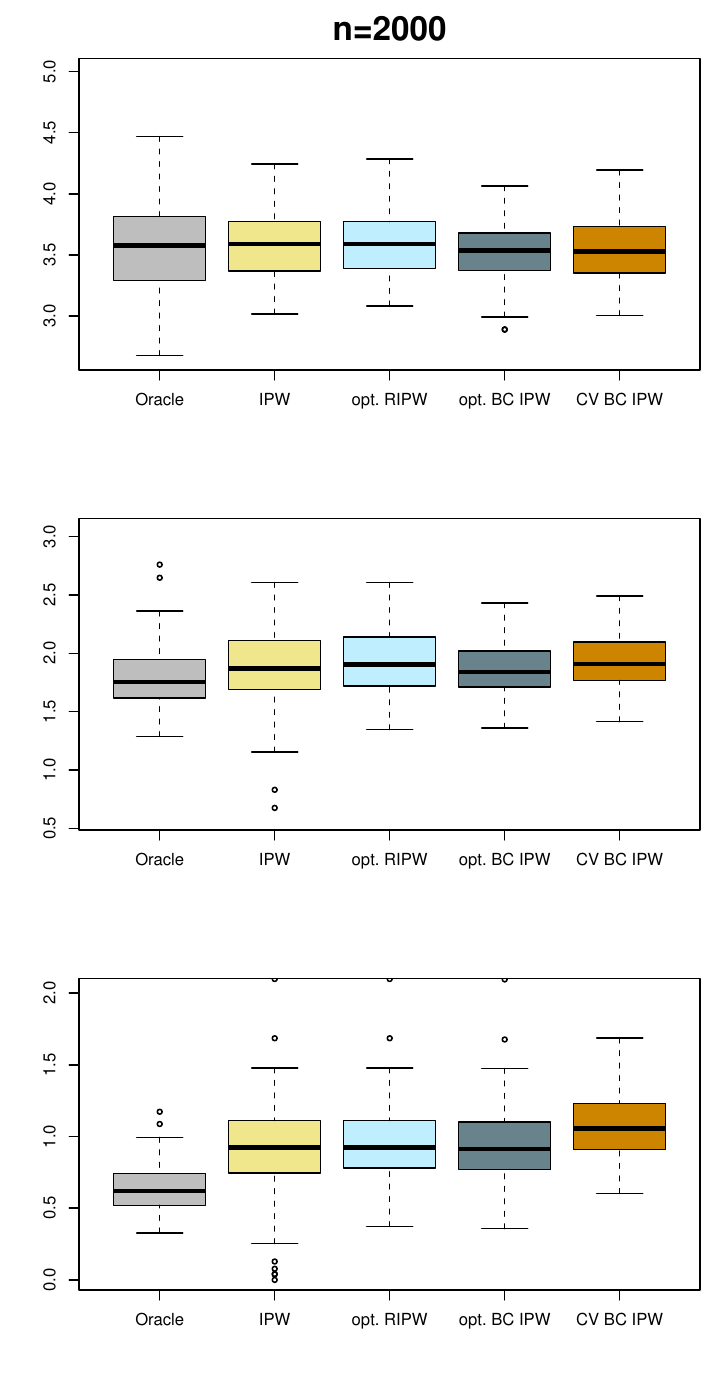}
    \end{minipage}
    \hfill
    \begin{minipage}{0.24\textwidth}
        \centering
        \includegraphics[width=\textwidth]{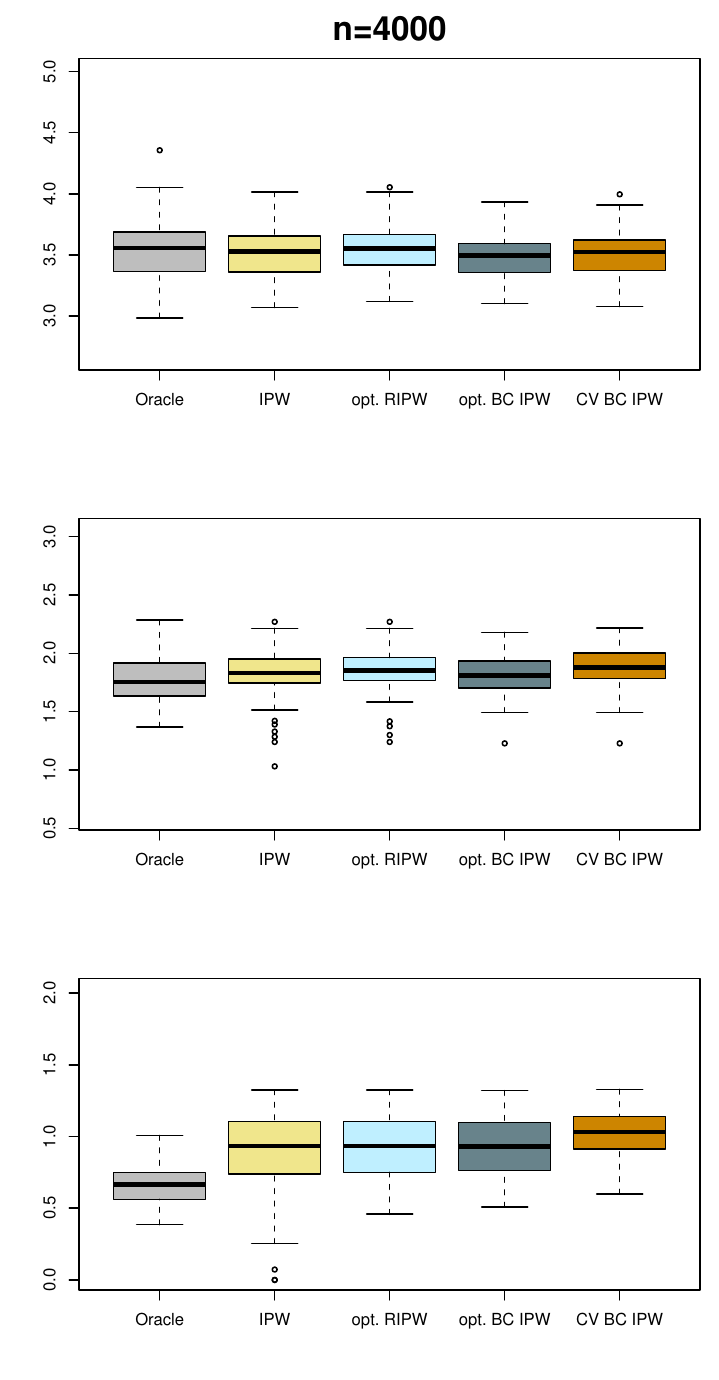}
    \end{minipage}
    \hfill
    \begin{minipage}{0.24\textwidth}
        \centering
        \includegraphics[width=\textwidth]{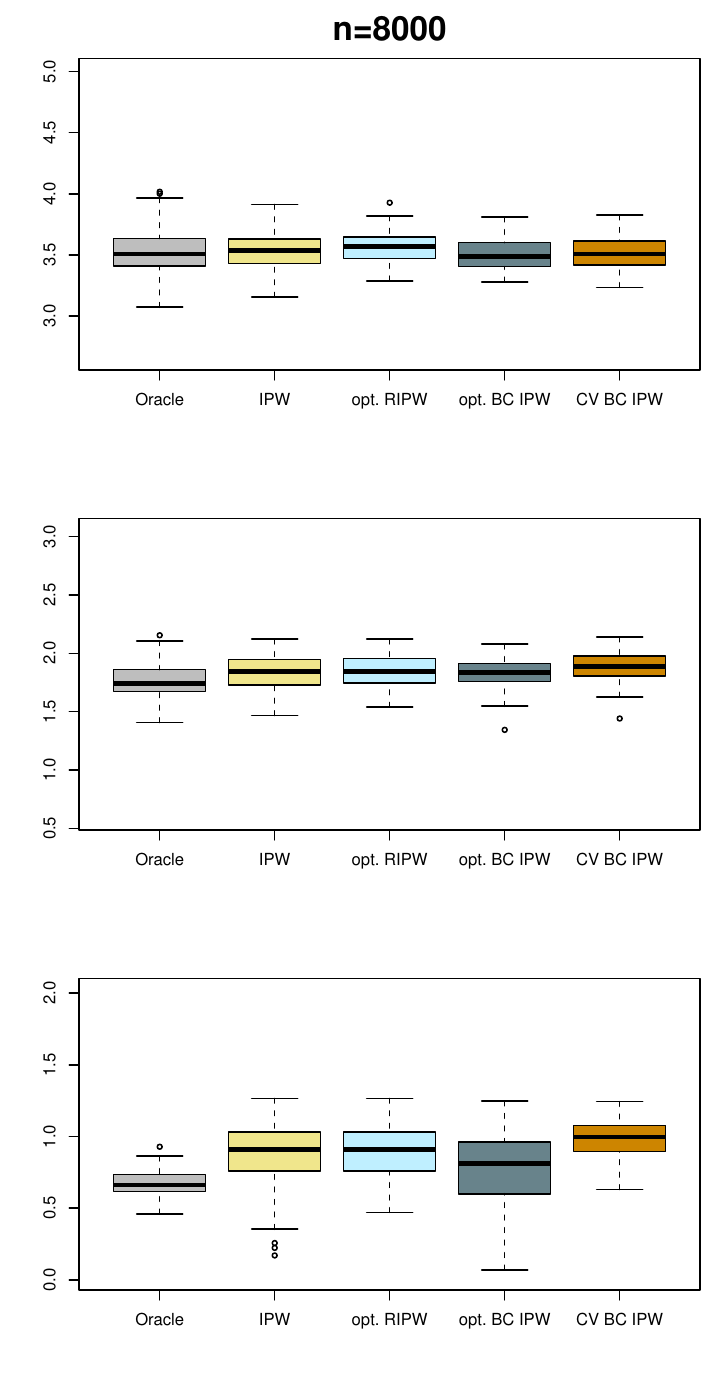}
    \end{minipage}
   \renewcommand{\baselinestretch}{1} \caption{  
           \footnotesize  Estimators of the $0.9$-quantile of $Y(1)$ based on model $(b)$ i.e. Fr\'echet responses. The description given for Figure \ref{fig:Q1_Gaussian} also applies here.  
           }
    \label{fig:Q1_Frechet}
\end{figure}
Note that the limiting conditional distribution \eqref{e:small.p} is $N(1,1)$ for model $(a)$ and Fr{\'e}chet with shape parameter $\alpha=2$ for model $(b)$. We estimate their corresponding $\tau$-quantiles of $Y_i$ which is equivalent to estimating $q_1(\tau)$ based on the estimator $\hat{q}_n(\tau)$ defined in  \eqref{e:argmin.n} using estimated propensity scores $e(X_i,\hat\beta)$ obtained by fitting a logistic regression.  We find a data-driven truncation parameter $b_n$ using the procedure described in Section \ref{sec:data_driven_selection} with an exponentially increasing grid of $100$ candidate parameters ranging  from $b_{n,1}=\min_{1\leq i \leq n}e(X_i,\hat\beta)$ to $b_{n,L}=\max_{1\leq i \leq n}e(X_i,\hat\beta)$, $m=\lfloor n/2 \rfloor$ and $n_B=1000$. We see from Figures \ref{fig:Q1_Gaussian} and \ref{fig:Q1_Frechet}  that the benefits of truncation in terms of MSE are more pronounced when the inverse propensity scores are very heavy-tailed. It is worth noting that the truncated IPW estimator that achieves the best empirical MSE (light blue) and the debiased truncated IPW estimator that achieves the best empirical MSE (dark blue) are virtually the same.  Neither of these estimators can be used in practice, since they require knowledge of the 0.9 quantile in advance in order to choose the {\it optimal} $b_n$.  In the first case, the non-debiased estimator essentially overcomes the bias by choosing the $b_n$ that minimizes the MSE.  So the light blue and dark blue boxplots are nearly the same.  The data-driven estimate that uses the method of Section 6.2 to debias the truncated IPW estimate essentially matches the performance as the other two estimators, but without having knowledge of the true quantile.  This is a remarkable performance.

\section{Empirical illustration}\label{sec:empirical}

We revisit the National Supported Work (NSW) program data that has been analyzed in multiple papers since the work of \cite{lalonde1986}, including \cite{dehejia:wahba1999,dehejia:wahba2002,ma:wang2020}.  The NSW  dataset consists of data from a labor training program implemented in the 1970s that provided work experience to a group of selected individuals. We follow the data preprocessing of \cite{ma:wang2020}. Therefore our sample consists of the treated individuals in the NSW experimental group ($D = 1$, sample size $n_1 = 185$), and a nonexperimental control group from the Panel Study of Income Dynamics (PSID, $D = 0$, sample size $n_0 = 1157$). The outcome response variable is the post-intervention earning measured in 1978. The covariates include information on age, education, marital status, ethnicity, and earnings in 1974 and 1975. Interested readers can find more details about the definition of these variables in \cite{dehejia:wahba1999,dehejia:wahba2002}.
\begin{figure}[h]
     \centering
     \includegraphics[scale=.4]{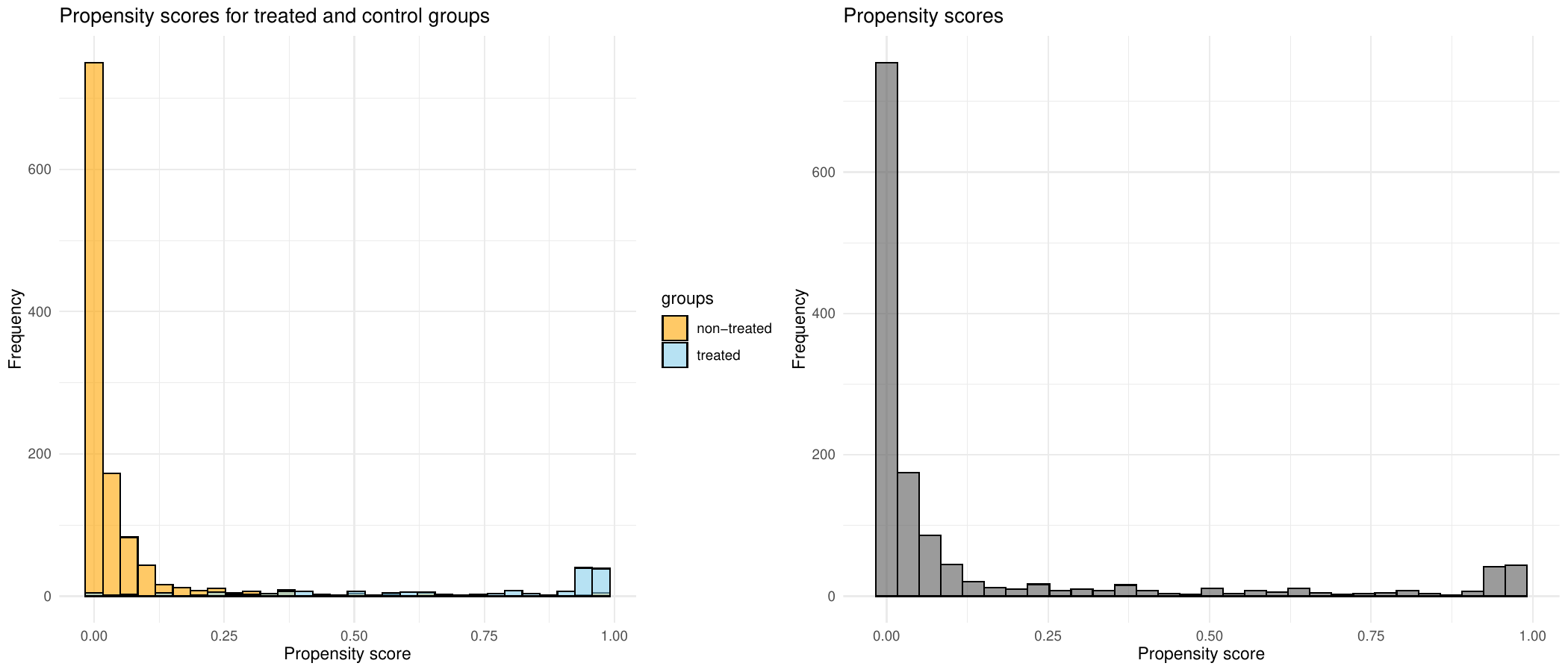}           \caption{ {\footnotesize Distribution of estimated propensity score for NSW-PSID data.}}
     \label{fig:Lalonde_prop}
 \end{figure}
 \begin{figure}[h]
     \centering
\includegraphics[scale=.37]{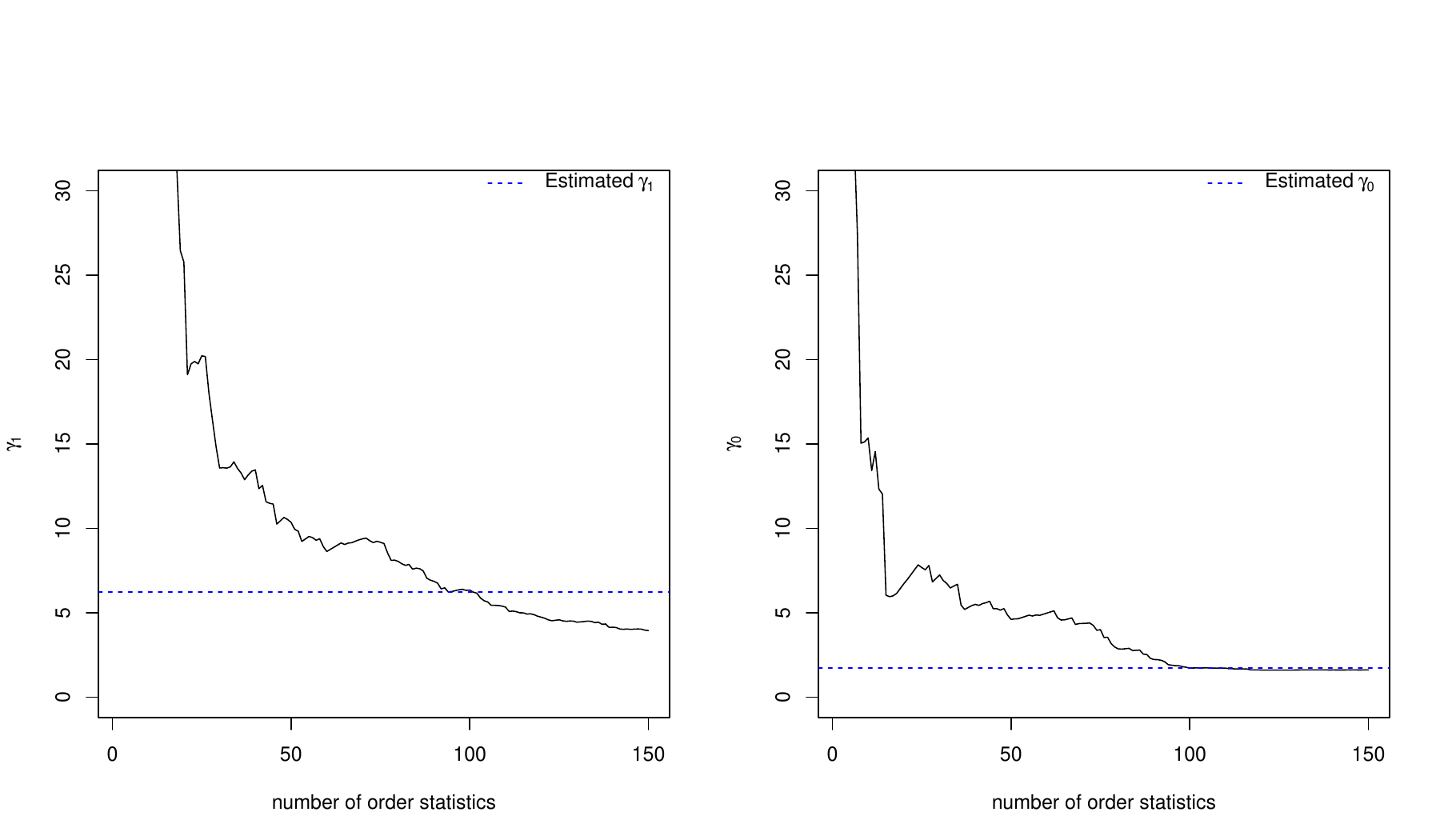}   
\renewcommand{\baselinestretch}{1}\caption{ {\footnotesize  Hill estimates of $\gamma_1$ and  $\gamma_0$ based on $\{1/\hat{e}(X_i)\}$ and $\{1/(1-\hat{e}(X_i))\}$ respectively.}}
\label{fig:Lalonde_Hill_plot}
 \end{figure}
 \begin{figure}[h]
     \centering
\renewcommand{\baselinestretch}{1}
     \includegraphics[scale=.45]{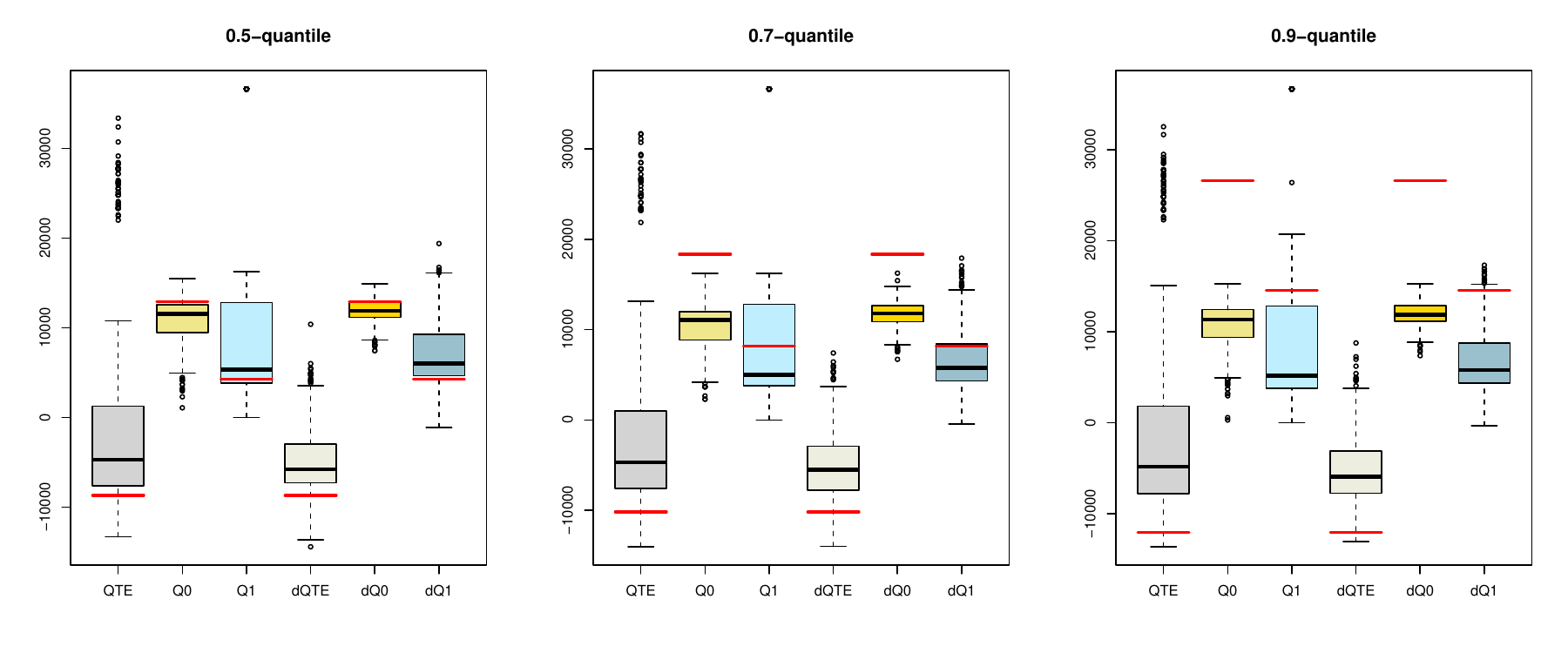}           \caption{ {\footnotesize The plot  shows the subsampled distribution of the estimated QTE and the corresponding quantiles of $Y(1)$ and $Y(0)$ for the quantile levels $\tau=\{0.5,0.7,0.9\}$, using the NSW-PSID data. The boxplots show $500$ subsampled realizations of size $\lfloor n/\log(n)\rfloor$ for the standard IPW estimators and our  debiased and truncated counterparts where the truncation parameter was chosen using the procedure described in Section \ref{sec:data_driven_selection}. The red lines show the corresponding
naive unweighted empirical quantile estimates using all the data.}}
\label{fig:Lalonde_QTE}
 \end{figure}
 
Figure \ref{fig:Lalonde_prop} shows the poor estimated overlap obtained with the estimated logistic regression reported in \cite{ma:wang2020} i.e., taking the variables \texttt{age, education, earn1974, earn1975}  and their squares as covariates, 
 as well as three dummy variables 
 for \texttt{married, black, hispanic} and an interaction term \texttt{black $\times$ u74}, where \texttt{u74} is the unemployment status in 1974. Using these propensity scores we can estimate the tail indexes of  $\{1/\hat{e}(X_i)\}$ or $\{1/(1-\hat{e}(X_i))\}$ using a Hill estimator. Note that $\hat\gamma_1$ and $\hat\gamma_0$ are $1$ plus the corresponding estimated regular variation index of $\{1/\hat{e}(X_i)\}$ and $\{1/(1-\hat{e}(X_i))\}$  respectively. Selecting the number of order statistics $k$ of the Hill estimator using  the minimum distance order procedure proposed in \cite{clauset:shalizi:newman2009} and analyzed in \cite{drees:janssen:resnick:wang2020}, we obtain $\hat{\gamma}_1\approx 7.958$ and $\hat{\gamma}_0\approx 4.343$. However, the Hill plots in Figure \ref{fig:Lalonde_Hill_plot} based on $100$ order statistics, corresponding to $7.5$\% of the data, suggest  that $\hat\gamma_1\approx 6.23$ and $\hat\gamma_0\approx1.73$ are reasonably good estimators. This indicates that the standard Gaussian limits for the estimated quantiles of $Y(0)$ do not apply. This would theoretically be reflected in a larger variance of $\hat{q}_0(\tau)$ relative to $\hat{q}_1(\tau)$, but  here the much smaller number of treated units makes $\hat{q}_1(\tau)$ have the larger variance.  

 Figure \ref{fig:Lalonde_QTE} shows the bootstrap subsampled distribution of the estimated quantile treatment effect $\delta(\tau)=q_1(\tau)-q_0(\tau)$ for $\tau=\{0.5,0.7,0.9\}$ over $500$ subsamples of size $n_B=\lfloor \frac{n}{\log n}\rfloor$.  Figure \ref{fig:Lalonde_QTE} also shows our debiased truncated IPW quantile estimator, where the truncation parameter was estimated using the three-step procedure described in Section \ref{sec:data_driven_selection}. It shows that as suggested in the simulations of the previous section, the  truncated estimator has a smaller variance than the standard IPW estimator. Furthermore, the resulting truncated estimates tend to be closer to the benchmark fully experimental data shown in purple.  
  \begin{figure}[h]
     \centering
     \includegraphics[scale=.45]{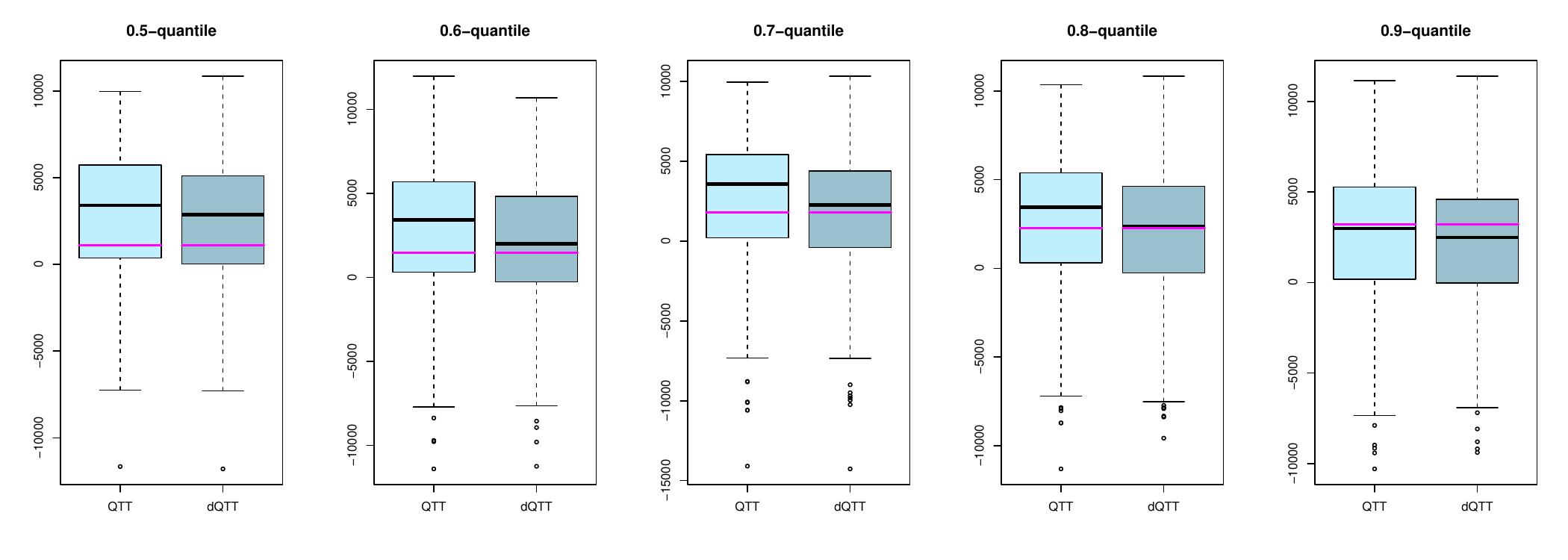}   
     \renewcommand{\baselinestretch}{1}\caption{ {\footnotesize Subsampled distributions of the estimated QTT using the standard IPW approach and our debiased estimator based on a data-driven selection of the truncation parameter.  The boxplots shows $500$ subsampled realizations of size $\lfloor n/\log(n)\rfloor$. The purple lines show the corresponding treatment effect based on the benchmark purely experimental NSW data.}}
     \label{fig:Lalonde_QTT}
 \end{figure}
 Finally, Figure \ref{fig:Lalonde_QTT} shows the estimated quantile treatment to the treated (QTT) based on standard IPW estimators and our data-driven truncated and debiased approach. Letting $q_{Y(j)|D=1}(\tau)$ denote the $\tau$-quantile of $Y(j)|D=1$ for $j\in\{0,1\}$, the QTT is defined as $\mathrm{QTT}(\tau)=q_{Y(1)|D=1}(\tau)-q_{Y(0)|D=1}(\tau)$. Note that the standard IPW estimator of the QTT is 
 \begin{equation*}
     \widehat{\mathrm{QTT}}^{\mathrm{IPW}}(\tau)=\hat{q}_{Y(1)|D=1}(\tau)-\hat{q}^{\mathrm{IPW}}_{Y(0)|D=1}(\tau),
 \end{equation*}
where 
\begin{equation*}
  \hat{q}_{Y(1)|D=1}(\tau)=\mbox{argmin}_q\sum_{i:D_i=1}\rho_\tau(Y_i-q) \quad \mbox{ and } \quad  \hat{q}^{\mathrm{IPW}}_{Y(0)|D=1}(\tau)=\mbox{argmin}_q\sum_{i:D_i=0}\frac{e(X_i)}{1-e(X_i)}\rho_\tau(Y_i-q).
\end{equation*}
In Figure \ref{fig:Lalonde_QTT} we report the IPW QTT estimator with estimated propensity scores $\hat{e}(X_i)$ and a truncated and debiased counterpart where truncation is applied to $\frac{1}{1-\hat{e}(X_i)}$.  The goal of this exercise is to assess how well these methods estimate the treatment effect relative to the benchmark purely experimental NSW data. Similar numerical illustrations have been considered in the literature, but for the average treatment to the treated (ATT) \citep{dehejia:wahba1999,dehejia:wahba2002,ma:wang2020}. This exercise reveals that overall, the estimated QTT based on the NSW-PSID data is close to the randomized benchmark NSW quantiles and that the estimates based on truncation seem to be slightly better than the standard IPW estimates in this case.

\section{Appendix A: Proof of main results}
\label{sec:appendix}

\subsection{Fixed quantiles}

We continue in the notation established in Remark \ref{rk:structure}.

We now switch to analyzing the limit behavior of the
processes $(Z^{(1)}_{n,\tau}(u), \,
u\in\bbr)$, $0<\tau<1$, in \eqref{e:split.Zn}.   The following proposition establishes, for fixed numbers $\tau_j, \, j=1,\ldots, k$  in $(0,1)$,  a multivariate
limit theorem for the random ingredients in these processes. 
Since each of the processes $(Z^{(1)}_{n,\tau_j}(u), \,
u\in\bbr)$  is linear in $u$, the joint limiting
behavior of these processes is completely determined in this proposition.

We now put together Propositions \ref{pr:z2} and \ref{pr:z1} to complete the proof of Theorem \ref{th:Z1}.

\begin{proposition} \label{pr:z2}
 If the conditions of Theorem \ref{th:Z1} hold for a  $\tau\in (0,1)$,  then, for $\gamma_1\in(1,2)$
  $$
  Z_{n,\tau}^{(2)}(u)\cip \frac{u^2}{2}g_Y(q(\tau))~~\mbox {as} \ n\to\infty\,, 
  $$
  where $Z_{n,\tau}^{(2)}(u)$ is as defined in \eqref{e:split.Zn}.  For the case $\gamma_1> 2$ or if $\gamma_1=2$ and $\E(1/e(X_1))<\infty$, then the same result holds with $Z_{n,\tau}^{(2)}(u)$ where $h_n$ is now $n^{1/2}$.
  
\end{proposition}

\begin{proposition} \label{pr:z1}
Let $\tau_1,\ldots, \tau_k$ be distinct points in $(0,1)$, such that 
 the conditions of Theorem \ref{th:Z1} hold at these points. 
 
 \smallskip 
    \noindent  {\rm Case 1: $\gamma_1\in (1,2)$.} As $n\to\infty$, 
  $$
  \left( h_n^{-1} \sum_{i=1}^n \frac{D_i}{\max(e(X_i),b_n)}
  \bigl( \tau_j-\one(Y_i\leq q(\tau_j))\bigr), \, j=1,\ldots, k\right) 
 \Rightarrow \bigl( Z_{\tau_j}, \, j=1,\ldots, k\bigr)
$$
weakly in $\bbr^k$, where $\bigl( Z_{\tau_1},\ldots, Z_{\tau_k}\bigr)$ is an
infinitely divisible random vector 
without a Gaussian component, with the characteristic function
$$
\E\exp\left\{ i\sum_{j=1}^k a_j Z_{\tau_j}\right\} = \exp\left\{
  i\sum_{j=1}^k a_j
  m_{\tau_j} + \int_{\bbr^k} \left(e^{i\sum_{j=1}^k a_j x_j} 
  -1 -i\sum_{j=1}^k a_j x_j\right)\, \nu_{\tau_1,\ldots, \tau_k}(d\bx)\right\}, 
$$
$(a_1,\ldots, a_k)\in\bbr^k$.   In the case $\theta >0$ the L\'evy measure
$\nu_{\tau_1,\ldots,\tau_k}$  is given by 
  \begin{align*}
\nu_{\tau_1,\ldots,
    \tau_k}=\nu_{\tau_1,\ldots,\tau_k;1}+\nu_{\tau_1,\ldots,\tau_k;
    2},  \ \ &\nu_{\tau_1,\ldots,\tau_k; 1}= 
  \frac{\gamma_1-1}{\gamma_1}\theta^{\gamma_1}  \, 
                G_0\circ T_{\tau_1,\ldots, \tau_k;\theta}^{-1},
    \\
  \notag &\nu_{\tau_1,\ldots,\tau_k; 2}= \frac{\gamma_1-1}{\gamma_1} \int_\theta^\infty \gamma_1
           x^{\gamma_1-1} G_0\circ T_{\tau_1,\ldots,\tau_k; x}^{-1}\,  dx, 
\end{align*}
where for   $x>0$, $ T_{\tau_1,\ldots,\tau_k;x}:\, \bbr\to\bbr^k$ is given by 
\begin{align*}
T_{\tau_1,\ldots,\tau_k;x}(y)  =  \left( \frac{1}{x} \bigl(
  \tau_j-\one(y\leq q(\tau_j))\bigr), \, j=1,\ldots, k\right). 
\end{align*}

In the case $\theta=0$ the L\'evy measure
$\nu_{\tau_1,\ldots,\tau_k}$  is given by
\begin{align*} 
\nu_{\tau_1,\ldots,
    \tau_k} = \frac{\gamma_1-1}{\gamma_1} \int_0^\infty \gamma_1
           x^{\gamma_1-1} G_0\circ T_{\tau_1,\ldots,\tau_k; x}^{-1}\,
  dx.
 \end{align*}

Finally, the mean $m_{\tau_j}$ is given by
\begin{equation} \label{e:ass.bias}
m_{\tau_j} =-\frac{\theta^{\gamma_1-1}}{\gamma_1} \int_\bbr 
  \bigl( \tau_j-\one(y\leq q(\tau_j))\bigr)\, G_0(dy), \ j=1,\ldots, k
\end{equation}
if $\theta>0$, 
and the mean vanishes if $\theta=0$. 

\medskip  \noindent {\rm Case 2: $\gamma_1>2$ or $\gamma_1=2$ and $\E(1/e(X_1))<\infty$.}  As $n\to\infty$,
$$
  \left( n^{-1/2} \sum_{i=1}^n \frac{D_i}{\max(e(X_i),b_n)}
  \bigl( \tau_j-\one(Y_i\leq q(\tau_j))\bigr), \, j=1,\ldots, k\right) 
 \Rightarrow \bigl( Z_{\tau_j}, \, j=1,\ldots, k\bigr)\,,
$$
where $\bigl( Z_{\tau},~0<\tau<1 \bigr)$ is a Gaussian process with mean 0 and covariance given by
$$
{\rm Cov}(Z_{\tau_1},Z_{\tau_2})=\E\left[\left(1/e(X_1)(\tau_1-\one(Y_1\leq q(\tau_1)))(\tau_2-\one(Y_1\leq q(\tau_2))\right)\right].
$$

\end{proposition}

\begin{proof}[Proof of Theorem \ref{th:Z1}]
It follows from the two propositions and \eqref{e:split.Zn} that \eqref{e:big.conv} holds in finite-dimensional distributions. Since for each $0<\tau<1$ the processes $\bigl(Z_{n,\tau}(u)\bigr)$ have convex sample paths, the laws of each component process in 
\eqref{e:big.conv} are tight in $C(\R)$ and, hence, the the laws of the entire random field are tight in $C^k(\R)$. 
It follows that \eqref{e:big.conv} holds as weak convergence in $C^k(\R)$; see e.g. Remark 1 in \cite{davis1992m}).

The fact that the limiting process of each component in \eqref{e:big.conv} has strictly convex sample paths and, hence, a unique sample minimum, implies that the argmin functional in \eqref{e:exr.diff} is on the set of full measure with respect to the limiting law for each $0<\tau<1$; see Lemma 2.2 in \cite{davis1992m}). Therefore, the continuous mapping theorem implies the statement of the theorem. 
\end{proof}

 \subsection{Intermediate quantiles}

\begin{proposition} \label{prop:cip.inter}
    Under the conditions of Theorem \ref{t:quant.lim.interm}, as $n\to\infty$, 
    
\begin{equation*}
Z_n^{(2)}(u) \cip u^2/2~~\mbox{ for every }u\in\R\,. 
\end{equation*}
\end{proposition}

\begin{proposition} \label{pr:Z1.inter}
Under the conditions of Theorem \ref{t:quant.lim.interm}, as $n\to\infty$, 
$$
\bigl(g_Y(q(\tau_n))h_n\bigr)^{-1} \sum_{i=1}^n \frac{D_i}{\max(e(X_i),b_n)}
  \bigl( \tau_n-\one(Y_i\leq q(\tau_n))\bigr)  
 \Rightarrow Z 
$$
weakly in $\bbr$, where $Z$ is an
infinitely divisible random variable 
without a Gaussian component, with the characteristic function
$$
\E e^{iaZ}\ = \exp\left\{ iam + \int_{-\infty}^0 \bigl(e^{iax}-1
  -iax\bigr)\, \nu(dx)\right\},
$$
$a\in\bbr$. Here
\begin{equation*} \label{e:lim.mean.interm}
  m= \frac{\beta-1}{\gamma_1}\theta^{\gamma_1},\end{equation*}
while the L\'evy measure
$\nu$  is given by 
 \begin{align*}
\nu= \beta \frac{\gamma_1-1}{\gamma_1}
      \theta^{\gamma_1} \delta_{-1/\theta} + \beta(\gamma_1-1)
   \mu_{\gamma_1; \theta},
 \end{align*}
 with $ \mu_{\gamma_1; \theta}$ having the following density with respect to the
 Lebesgue measure $\lambda$: 
 $$
 \frac{d \mu_{\gamma_1; \theta}}{d\lambda}(x) = |x|^{-(\gamma_1+1)},
 \ -1/\theta<x<0.
 $$
 
\end{proposition}

 \begin{proof}[Proof of Theorem \ref{t:quant.lim.interm}]
The statement of the theorem follows from Propositions \ref{prop:cip.inter} and \ref{pr:Z1.inter} in the same way as the statement of Theorem \ref{thm:joint} follows from 
Propositions \ref{pr:z2} and \ref{pr:z1}. 
\end{proof}

\subsection{Quantile treatment effects}

\begin{proof}[Proof of Theorem \ref{thm:joint} for fixed quantiles]
The proof consists of applying the procedure of Theorem \ref{th:Z1} to both $\hat q_1(\tau)$
and $\hat q_0(\tau)$ at the same time. That is, we apply the decomposition in Remark \ref{rk:structure} to both estimators. The convergence in probability described in Proposition \ref{pr:z2} clearly continues to hold for both estimators. The version of the statement of Proposition \ref{pr:z1} that is needed now takes the form 
 \begin{align*} 
  & \left( (h_n^{(1)})^{-1} \sum_{i=1}^n \frac{D_i}{\max(e(X_i),b_n)}
  \bigl( \tau_j-\one(Y_i(1)\leq q_1(\tau_j))\bigr), \, j=1,\ldots, k, \right. \\
  \notag  &\left.  
  (h_n^{(0)})^{-1} \sum_{i=1}^n \frac{1-D_i}{\max(1-e(X_i),b_n)}
  \bigl( \tau_j-\one(Y_i(0)\leq q_0(\tau_j))\bigr), \, j=1,\ldots, k
  \right) \\
\notag \Rightarrow&  \bigl( Z_{1,\tau_j}, \, j=1,\ldots, k, \, Z_{2,\tau_j}, \, j=1,\ldots, k
\bigr), 
\end{align*}
where $\bigl( Z_{1,\tau_1},\ldots, Z_{1,\tau_k}\bigr)$ and $\bigl( Z_{0, \tau_1},\ldots, Z_{0, \tau_k}\bigr)$ are independent infinitely divisible random vectors, each with the distribution described in Proposition \ref{pr:z1}. To prove that, we need to check the three claims corresponding to \eqref{e:right.t}, \eqref{e:tr.m} and \eqref{e:tr.v}. The claim corresponding to \eqref{e:right.t} is
\begin{align} \label{e:vague.2sided}
&n\bbP\left[\left( (h_n^{(1)})^{-1}\frac{D_1}{\max(e(X_1),b_n)}
  \bigl( \tau_j-\one(Y_1(1)\leq q_1(\tau_j))\bigr), \, j=1,\ldots,    k, \, \right.\right. \\ 
\notag & \left.\left.   (h_n^{(0)})^{-1}\frac{1-D_1}{\max(1-e(X_1),b_n)}
  \bigl( \tau_j-\one(Y_1(0)\leq q_0(\tau_j))\bigr), \, j=1,\ldots,    k
  \right)\in \cdot\right] \vague \hat \nu_{\tau_1,\ldots, \tau_k}(\cdot) 
\end{align}
vaguely in $(\overline{\bbr})^{2k}\setminus \{0\}$, where $\hat \nu_{\tau_1,\ldots, \tau_k}$ is a Radon measure given by
\begin{equation*} \label{e:big.Lm}
\hat \nu_{\tau_1,\ldots, \tau_k} ((A\times B) \setminus \{0\}) = \nu^{(1)}_{\tau_1,\ldots, \tau_k} (A\setminus \{0\})\times \delta_0(B) +
\delta_0(A) \times \nu^{(0)}_{\tau_1,\ldots, \tau_k} (B\setminus \{0\})
\end{equation*}
for any Borel sets $A,B$ in $\bbr^k$. Here   $\nu^{(1)}_{\tau_1,\ldots, \tau_k}$ and $\nu^{(0)}_{\tau_1,\ldots, \tau_k}$ are the L\'evy measures described in Proposition \ref{pr:z1} for the cases $z\to 0$ and $z\to 1$, correspondingly. However, for any  Borel set $A$ in $\bbr^{2k}$ bounded away from the origin,
\begin{align*}
&n \bbP\left[\left( (h_n^{(1)})^{-1}\frac{D_1}{\max(e(X_1),b_n)}
  \bigl( \tau_j-\one(Y_1(1)\leq q_1(\tau_j))\bigr), \, j=1,\ldots,    k, \, \right.\right. \\ 
 & \left.\left.   (h_n^{(0)})^{-1}\frac{1-D_1}{\max(1-e(X_1),b_n)}
  \bigl( \tau_j-\one(Y_1(0)\leq q_0(\tau_j))\bigr), \, j=1,\ldots,    k
  \right)\in A\right] \\ 
 =&n\int_0^1 z\bbP\left[\left((h_n^{(1)})^{-1}\frac{1}{\max(z,b_n)}    \bigl( \tau_j-\one(Y_1(1)\leq q_1(\tau_j))\bigr), \, j=1,\ldots,    k\right) 
 \in A_{1,{\bf 0}}\Bigg| e(X_1)=z 
 \right] \\
 & \hskip 5.5in \LFe(dz)\\
 +& n\int_0^1 (1-z)\bbP\left[\left((h_n^{(0)})^{-1}\frac{1}{\max(1-z,b_n)}    \bigl( \tau_j-\one(Y_1(0)\leq q_0(\tau_j))\bigr), \, j=1,\ldots,    k\right) 
 \in A_{2,{\bf 0}}\Bigg| e(X_1)=z 
 \right] \\
 & \hskip 5.5in \LFe(dz), 
\end{align*}
where $A_{1,{\bf 0}}$ and $A_{2,{\bf 0}}$ are the sections of the set $A$ obtained by setting the first $k$ or the last $k$ coordinates equal to 0, correspondingly. However, by Fubini's theorem, both of these sets are bounded away from the origin Borel subsets of $\bbr^k$, and it follows from \eqref{e:right.t} that the first term in the right hand side converges to $\nu^{(1)}_{\tau_1,\ldots, \tau_k}(A_{1,{\bf 0}})$ as long as the latter set is a continuity set of $\nu^{(1)}_{\tau_1,\ldots, \tau_k}$. By the analog of \eqref{e:right.t} applied to the case $z\to 1$,  we see that the second term in the right hand side converges to $\nu^{(0)}_{\tau_1,\ldots, \tau_k}(A_{2,{\bf 0}})$, again as the long as the latter set is a continuity set. This proves \eqref{e:vague.2sided}.

The statement \eqref{e:tr.m} is a marginal convergence statement, so
its  version needed in the present case follows because we already
know that it holds separately when $z\to 0$ and when $z\to
1$. 
Finally, the statement \eqref{e:tr.v} is a two-dimensional
convergence statement. When both coordinates correspond to the case
$z\to 0$, or both correspond to the case $z\to 1$, its version needed
in the present case is already known. This leaves only the statement  
\begin{align*} 
&n\, \cov\left[ (h_n^{(1)})^{-1}\frac{D_1}{\max(e(X_1),b_n)}
  \bigl( \tau_{j_1}-\one(Y_1(1)\leq q(\tau_{j_1}))\bigr)  \one\bigl(
                               A_n(t)\bigr), \right. \\
\notag &\left. \hskip 0.4in (h_n^{(0)})^{-1}\frac{1-D_1}{\max(1-e(X_1),b_n)}
  \bigl( \tau_{j_2}-\one(Y_1(0)\leq q(\tau_{j_2}))\bigr)  \one\bigl(
                               A_n(t)\bigr)\right] \\
\notag &\hskip 1.5in 
 \to
\int_{\bbr^{2k}} x_{j_1}x_{j_2}\one\bigl( \|\bx\|_\infty\leq t\bigr)
\hat \nu_{\tau_1,\ldots, \tau_k}(d\bx)  
\end{align*}
for every $j_1,j_2=1,\ldots, k$ and $t>0$ of the type described in Proposition \ref{pr:z1}. However, by the definition of the measure $\hat \nu_{\tau_1,\ldots, \tau_k}$, the integral in the right hand sides vanishes. Furthermore, we have shown in the proof of the proposition that the product of the expectations part of the covariance vanishes in the limit. Now the validity of the above statement follows from the fact that either $D_1$ or $1-D_1$ always vanish. 
 \end{proof}

\begin{proof}[Proof of Theorem \ref{thm:joint} for intermediate quantiles.] Once again, the proof consists of applying the procedure used in the proof of Theorem \ref{t:quant.lim.interm} to both $\hat q_1(\tau_n)$ and $\hat q_0(\tau_n)$ at the same time. The convergence in probability in Proposition \ref{prop:cip.inter}, obviously, holds for both estimators at the same time. The version of Proposition \ref{pr:Z1.inter} we need now is 
\begin{align*} \label{e:interm.2sided.prop}
&\left(\bigl(g_Y(q(\tau_n))h_n^{(1)}\bigr)^{-1} \sum_{i=1}^n \frac{D_i}{\max(e(X_i),b_n)}
  \bigl( \tau_n-\one(Y_i(1)\leq q(\tau_n))\bigr),
\right. \\
\notag & \left. \bigl(g_Y(q(\tau_n))h_n^{(0)}\bigr)^{-1} \sum_{i=1}^n \frac{1-D_i}{\max(1-e(X_i),b_n)}
  \bigl( \tau_n-\one(Y_i(0)\leq q(\tau_n))\bigr) 
  \right) 
 \Rightarrow (Z_1,Z_0)
    \end{align*}
with $Z_1,Z_0$ being independent infinitely divisible random variables, each with a distribution described in Proposition \ref{pr:Z1.inter}, and for this we need to verify three claims corresponding to \eqref{e:vague.interm}, \eqref{e:tr.m.interm} and \eqref{e:tr.v.interm}. However, all of these three claims follow from their counterparts among 
\eqref{e:vague.interm}, \eqref{e:tr.m.interm} and \eqref{e:tr.v.interm} in the same way as in the case of fixed quantiles. 
 \end{proof}

\newpage

\section*{Appendix B: auxiliary results} \label{app:B} 

This appendix contains 
the proofs of the auxiliary technical results needed for the proofs of the main theorems in 
Sections \ref{sec:IPW-quantile}, \ref{sec:intermediate}, and \ref{sec:qte} as well as comments on the need for bootstrapped subsamples for approximating the sampling distribution of the estimated quantiles.

We start by establishing certain useful properties of moments of the propensity scores, followed by two subsections, dealing with fixed quantiles and intermediate quantiles, respectively.

\begin{lemma} \label{e:regvar.est}
 Under Assumption \ref{ass:RV_prop_score},  the following asymptotic results hold as $u\to 0$:
  \begin{equation} \label{e:small.z}
  \int_0^{u} z\LFe(dz) \sim \frac{\gamma_1-1}{\gamma_1} u 
    \bbP(e(X_1)\leq u),
 \end{equation}
 and for any $\beta>\gamma_1-1$,
 \begin{equation} \label{e:big.z}
 \int_{u}^1 z^{-\beta}\LFe(dz) \sim \frac{\gamma_1-1}{1+\beta-\gamma_1}
 u^{-\beta}   \bbP(e(X_1)\leq u).
 \end{equation}

\end{lemma} 
\begin{proof}
  Note that 
  \begin{align*}
      \int_0^{u} z\LFe(dz)&=\int_0^\infty \bbP(x<e(X)\le u)\,dx\\
     &= \int_0^u \ (\LFe(u)-\LFe(x))\,dx,
  \end{align*}
and 
 \eqref{e:small.z} follows by Karamata's theorem.

As for \eqref{e:big.z}, note that $W=1/e(X)$ is regularly varying at infinity with index $-(\gamma_1-1)$.  Hence, by  Karamata's theorem, as $u\to 0$, 
\begin{eqnarray*}
 \int_{u}^1 z^{-\beta}\LFe(dz)&=&\E\left (W^\beta\one(W\le 1/u)\right)\\
 &\sim& \frac{\gamma_1-1}{1+\beta-\gamma_1}u^{-\beta}\mathbb{P}(e(X)\le u),   
\end{eqnarray*}
as required.
\end{proof}

\subsection{Fixed quantiles} \label{subsec:fixed}

We continue in the notation established in Remark \ref{rk:structure}.

\subsubsection{Proof of Proposition \ref{pr:z2}}

\begin{proof}  We first consider the case $\gamma_1\in (1,2)$.  For $u>0$, we have
\begin{align*}
&\E Z^{(2)}_{n,\tau}(u) = \frac{n^2}{h_n^2}\E\left[   \frac{D_1}{\max(e(X_1),b_n)} \one
  (q(\tau)<Y_1\leq q(\tau)+uh_n/n) \bigl( uh_n/n-Y_1+q(\tau)\bigr)\right] \\
=&\frac{n^2}{h_n^2}\left[       \int_0^{b_n} \frac{z}{b_n} \E\bigl[  \one
  (q(\tau)<Y_1\leq q(\tau)+uh_n/n) \bigl( uh_n/n-Y_1+q(\tau)
   \big| e(X_1)=z\bigr]\, \LFe(dz)  \right. \\
  &\left.  + \int_{b_n}^1   \E\bigl[  \one
  (q(\tau)<Y_1\leq q(\tau)+uh_n/n) \bigl( uh_n/n-Y_1+q(\tau)
   \big| e(X_1)=z\bigr]\, \LFe(dz) 
\right] \\
=& A_n(u) +B_n(u). 
\end{align*}
Notice that by \eqref{e:key.lim} and assumption \eqref{e:small.p1}, 
\begin{align} \label{e:mean.A}
A_n(u) \leq& u\frac{n}{h_n} \int_0^{b_n} \frac{z}{b_n} \bbP
  \bigl(q(\tau)<Y_1\leq q(\tau)+uh_n/n  
   \big| e(X_1)=z\bigr)\, \LFe(dz) \\
\notag \leq& u\frac{n}{h_n} \bbP(e(X_1)\leq b_n)\sup_{z\leq b_n}\bbP
  \bigl(q(\tau)<Y_1\leq q(\tau)+uh_n/n  
   \big| e(X_1)=z\bigr) \to 0. 
\end{align}
Next,
\begin{align*} 
B_n(u) =& \frac{n^2}{h_n^2}E\bigl[  \one
  (q(\tau)<Y_1\leq q(\tau)+uh_n/n) \bigl(
  uh_n/n-Y_1+q(\tau)\bigr)\bigr]\\
\notag  -&
 \frac{n^2}{h_n^2} \int_0^{b_n}  \E\bigl[  \one
  (q(\tau)<Y_1\leq q(\tau)+uh_n/n) \bigl( uh_n/n-Y_1+q(\tau)
   \big| e(X_1)=z\bigr]\, \LFe(dz). 
\end{align*}
The same bound in \eqref{e:mean.A} can be applied to the second term in the above equation and hence it vanishes in the limit. Furthermore, by \eqref{e:dens,tau}, as $n\to\infty$, 
\begin{align} \label{e:mean_B1}
&\E\left\{\frac{n^2}{h_n^2}\bigl[  \one
  (q(\tau)<Y_1\leq q(\tau)+uh_n/n) \bigl(
  uh_n/n-Y_1+q(\tau)\bigr)\bigr]\right\} \\
\nonumber \sim & g_Y(q(\tau))
  \frac{n^2}{h_n^2}\int_{q(\tau)}^{q(\tau)+uh_n/n} (
  uh_n/n-y+q(\tau)\bigr)\, dy =\frac{u^2}{2}g_Y(q(\tau)). 
\end{align}
We conclude by \eqref{e:mean.A} and \eqref{e:mean_B1} that for $u>0$, 
\begin{equation} \label{e:lim.mean}
  \lim_{n\to\infty}\E Z^{(2)}_{n,\tau}(u) =
  \frac{u^2}{2}g_Y(q(\tau)), 
\end{equation}
and a similar argument shows that \eqref{e:lim.mean} holds for
$u\leq 0$ as well.

Next we show that for $\theta>0$ and $p=2$ or for $\theta=0$ and $p\in (1,\gamma_1)$
\begin{equation} \label{e:lim.var}
\lim_{n\to\infty}\E\Bigl|Z^{(2)}_{n,\tau}(u)-\E\left(Z^{(2)}_{n,\tau}(u)\right)\Bigr|^p =0,
\end{equation}
from which the conclusion of the proposition follows.

Again, we only spell out the argument for $u>0$.   We use the Marcinkiewicz-Zygmund inequality; see e.g. \cite{chow:teicher:1988}, p. 367, Theorem 2. 
For  $p\in(1,2]$, there exists a $C_p$ such that
\begin{align*}
  &\E\Bigl|Z^{(2)}_{n,\tau}(u) -\E Z^{(2)}_{n,\tau}(u)\Bigr|^p\\
\leq& C_p\frac{n^{p+1}}{h_n^{2p}}\E\biggl|  \frac{D_1}{\max(e(X_1),b_n)} \one
  (q(\tau)<Y_1\leq q(\tau)+uh_n/n) \bigl(
  uh_n/n-Y_1+q(\tau)\bigr)\\
  &\hskip 0.7in  -\E\left( \frac{D_1}{\max(e(X_1),b_n)} \one
  (q(\tau)<Y_1\leq q(\tau)+uh_n/n) \bigl(
  uh_n/n-Y_1+q(\tau)\bigr)\right)
  \biggr|^p\\
\leq & 2^p C_p\frac{n^{p+1}}{h_n^{2p}}\E\biggl|  \frac{D_1}{\max(e(X_1),b_n)} \one
  (q(\tau)<Y_1\leq q(\tau)+uh_n/n) \bigl(
  uh_n/n-Y_1+q(\tau)\bigr)\biggr|^p\\
 \leq& 2^p C_p u^p\frac{n}{h_n^p}\left[       \int_0^{b_n} \frac{z}{b_n^p} \bbP
  \bigl(q(\tau)<Y_1\leq q(\tau)+uh_n/n  
   \big| e(X_1)=z\bigr)\, \LFe(dz)  \right. \\
   & \left.  \hskip 1in + \int_{b_n}^1    \frac{1}{z^{p-1}} \bbP
  \bigl(q(\tau)<Y_1\leq q(\tau)+uh_n/n  
   \big| e(X_1)=z\bigr)\, \LFe(dz) 
\right]\\
 = & A^{(p)}_n(u) +B^{(p)}_n(u). 
\end{align*}
Now suppose $\theta>0$ and choose $p=2$.  Then we have by definition of $h_n$, \eqref{e:key.lim}, and assumption \eqref{e:small.p1}
\begin{align} 
\notag A^{(2)}_n(u) \leq  & 2^2 C_2 u^2\frac{n}{h_n^2b_n} \bbP(e(X_1)\leq  b_n)\sup_{z\le b_n}\bbP
  \bigl(q(\tau)<Y_1\leq q(\tau)+uh_n/n  
   \big| e(X_1)=z\bigr)\\
\notag =& 2^2 C_2 u^2\frac{n}{h_nb_n} h_n^{-1} \bbP(e(X_1)\leq  b_n) o(1)\\
\notag \sim& 2^2 C_2 u^2\theta^{\gamma_1-2}o(1) 
   \to   0\,.
\end{align}
Next, for a fixed $z_0\in (0,1)$, we have by \eqref{e:big.z}, 
\begin{align*} 
B^{(2)}_n(u) \leq& 2^2 C_2 u^2\frac{n}{h_n^2} \left(\int_{b_n}^{z_0}
  \frac{1}{z}  \LFe(dz) \sup_{z\le z_0}\bbP
  \bigl(q(\tau)<Y_1\leq q(\tau)+uh_n/n 
   \big| e(X_1)=z\bigr)+\frac{1}{z_0}\right)\\
\notag \sim& 2^2 C_2 u^2\frac{\gamma_1-1}{2-\gamma_1}\frac{n}{h_n^2} 
b_n^{-1}\bbP(e(X_1)\leq b_n) \sup_{z\le z_0}\bbP
  \bigl(q(\tau)<Y_1\leq q(\tau)+uh_n/n 
   \big| e(X_1)=z\bigr) \\
\notag +&2^2C_2u^2\frac{n}{h_n^2}\frac{1}{z_0}.
\end{align*}
The second term above vanishes $n\to\infty$ because $h_n$ is regularly varying with exponent $1/\gamma_1>1/2$, after which the first term also vanishes by the assumption \eqref{e:small.p1} as we subsequently let $z_0\to 0$. We conclude that $B^{(2)}_n(u)\to0$ as desired.  

Turning to the case $\theta=0$ and choosing $p\in (1,\gamma_1)$, we have, if $b_n>0$,
\begin{align} 
\notag A^{(p)}_n(u) \leq  & 2^p C_p u^p\frac{n}{h_n^pb_n^{p-1}} \bbP(e(X_1)\leq  b_n)\\
\notag =& 2^p C_p u^p\frac{n}{h_n^{p-1}b_n^{p-1}} h_n^{-1} \bbP(e(X_1)\leq  h_n^{-1})\frac{\bbP(e(X_1)\leq  b_n)}{\bbP(e(X_1)\leq  h_n^{-1})} \\
\notag \sim& 2^p C_p u^p\frac{1}{(h_nb_n)^{p-1}}\frac{\bbP(e(X_1)\leq  b_n)}{\bbP(e(X_1)\leq  h_n^{-1})} \to   0\,,
\end{align}
since, by the regular variation Assumption \ref{ass:RV_prop_score} and by the Potter bounds, for any $\vep>0$, 
$\bbP(e(X_1)\leq  b_n)/\bbP(e(X_1)\leq h_n^{-1})\leq (1+\vep)(h_nb_n)^{\gamma_1-1-\vep}$. The claim follows by noticing that we can choose $\vep>0$ so small that $\gamma_1-1-\vep>p-1$ and recalling that $h_nb_n\to 0$. 

Furthermore, using the fact that $\int_{0}^1z^{-(p-1)}\LFe(dz)<\infty$ for $p<\gamma_1$, and  \eqref{e:small.p-2}, we have for a small enough $0<z_0<1$,
\begin{align*} 
B^{(p)}_n(u) \leq & \ C\frac{n}{h_n^p}  \int_{b_n}^{z_0}
  \frac{1}{z^{p-1}}   \LFe(dz)\sup_{z\le z_0}\bbP\bigl[
  q(\tau)<Y_1\leq q(\tau)+uh_n/n)  
   \big| e(X_1)=z\bigr] \\
   &\ +  C\frac{n}{h_n^p} z_0^{-(p-1)}\int_{0}^{1}
    \bbP\bigl[q(\tau)<Y_1\leq q(\tau)+uh_n/n)  
   \big| e(X_1)=z\bigr]\LFe(dz)\\
\notag \le&\  C\frac{1}{h_n^{p-1}} 
 +  C\frac{n}{h_n^{p}} z_0^{-(p-1)}\bbP
  (q(\tau)<Y_1\leq q(\tau)+uh_n/n) \to 0\,,
\end{align*}
where  $C$ is a constant which may change value from line to line.  Therefore, \eqref{e:lim.var} holds for $\theta=
0$ as well.

Finally, if $b_n=0$ is also covered by the calculations above.  In this case, the terms $A_n(u)$ and $A_n^{(p)}(u)$ simply vanish. 

 Now turning to the case, $\gamma_1>2$ or $\gamma_1=2$ and $\E(1/e(X_1))<\infty$, the same argument as above shows that $\E Z_{n,\tau}^{(2)}(u)\to \frac{u^2}{2}g_Y(q(\tau))$.  To complete the proof, it sufficies to show
$\E(Z_{n,\tau}^{(2)}(u)-\E Z_{n,\tau}^{(2)}(u))^2\to 0$ and $n\to \infty$. However, by the preceding argument with $p=2$ and $C$ a finite constant, we have the bound,
\begin{align*}
  &\E\left(Z^{(2)}_{n,\tau}(u) -\E Z^{(2)}_{n,\tau}(u)\right)^2\\
\leq & ~C\frac{n^{3}}{n^{2}} \E\left(  \frac{D_1}{\max(e(X_1),b_n)} \one
  (q(\tau)<Y_1\leq q(\tau)+ u/n^{1/2}) \bigl(
  u/n^{1/2}-Y_1+q(\tau)\bigr) \right)^2\\
  \leq & ~C u^2\E\left( \frac{D_1}{\max(e(X_1),b_n)} \one
  (q(\tau)<Y_1\leq q(\tau)+ u/n^{1/2})  \right)^2\\
  \leq & ~C u^2\E\left(e(X_1)\left(  \frac{1}{\max(e(X_1),b_n)} \one
  (q(\tau)<Y_1\leq q(\tau)+ u/n^{1/2})\right)^2 \right)\\
  \leq & ~C u^2\E\left(1/e(X_1) \one
  (q(\tau)<Y_1\leq q(\tau)+ u/n^{1/2})\right)\\
  \to & ~0,
\end{align*}
by the dominating convergence theorem ($\E(1/e(X_1)<\infty$).  This completes the proof.

\end{proof}

\subsubsection{Proof of Proposition \ref{pr:z1}} 

\begin{proof}  Case 1  ($\gamma_1\in(1,2)$).
We use Theorem 15.28 in \cite{kallenberg:2002}, according to which we
have to prove the following three statements:
\begin{align} \label{e:right.t}
&n\bbP\left[\left( h_n^{-1}\frac{D_1}{\max(e(X_1),b_n)}
  \bigl( \tau_j-\one(Y_1\leq q(\tau_j))\bigr), \, j=1,\ldots,
                                  k \right)\in \cdot\right] \vague 
                                  \nu_{\tau_1,\ldots, \tau_k}(\cdot) 
\end{align}
vaguely in $(\overline{\bbr})^k\setminus \{0\}$, 
 \begin{align} \label{e:tr.m}
&n\E\left[ h_n^{-1}\frac{D_1}{\max(e(X_1),b_n)}
  \bigl( \tau_j-\one(Y_1\leq q(\tau_j))\bigr) \one\bigl( A_n(t)\bigr)\right] \\
\notag &\hskip 1.5in \to
m_{\tau_j} - \int_{\bbr^k} x_j\one\bigl( \|\bx\|_\infty>t\bigr)
\nu_{\tau_1,\ldots, \tau_k}(d\bx)  
 \end{align}
 for any $j=1,\ldots, k$, for
 some $t>0$  such that
\begin{equation} \label{e:cty.pt}
\nu_{\tau_1,\ldots, \tau_k}\bigl( \bigl\{ \bx\in\bbr^{k}:\,
\|\bx\|_\infty=t\bigr\}\bigr)=0, 
\end{equation}
where 
 \begin{equation*} \label{e:A.n}
A_n(t) = \left\{ h_n^{-1}\frac{D_1}{\max(e(X_1),b_n)}
  \bigl| \tau_j-\one(Y_1\leq q(\tau_j))\bigr|\leq t \ \text{for} \ \
  j=1,\ldots, k\right\}. 
\end{equation*}
Finally, we need to show that 
\begin{align} \label{e:tr.v}
&n\, \cov\left[ h_n^{-1}\frac{D_1}{\max(e(X_1),b_n)}
  \bigl( \tau_{j_1}-\one(Y_1\leq q(\tau_{j_1}))\bigr)  \one\bigl(
                               A_n(t)\bigr), \right. \\
\notag &\left. \hskip 0.4in h_n^{-1}\frac{D_1}{\max(e(X_1),b_n)}
  \bigl( \tau_{j_2}-\one(Y_1\leq q(\tau_{j_2}))\bigr)  \one\bigl(
                               A_n(t)\bigr)\right] \\
\notag &\hskip 1.5in 
 \to
\int_{\bbr^k} x_{j_1}x_{j_2}\one\bigl( \|\bx\|_\infty\leq t\bigr)
\nu_{\tau_1,\ldots, \tau_k}(d\bx)  
\end{align}
for every $j_1,j_2=1,\ldots, k$ and the same $t>0$ as above. 

We start with the case $\theta>0$. 
Observe that only the points of the type $t=\tau_j/\theta$ and
$t=(1-\tau_j )/\theta$, $j=1,\ldots, k$, may 
have positive masses in \eqref{e:cty.pt}, 
so in the sequel our choices of $t$ avoid these
points. 
For reasons that will be clear in the sequel, we also
avoid rational multiples of these points. 

We start by proving \eqref{e:right.t}. It suffices to prove vague
convergence in each quadrant of $\bbr^k$ separately, and we will spell
out the argument for the nonnegative quadrant only. It is enough to
prove that for any choice of positive numbers $t_1,\ldots, t_k$ none of
which is a rational multiple of the points above, we have 
\begin{align} \label{eq:nuA}
&n\bbP\left[\left( h_n^{-1}\frac{D_1}{\max(e(X_1),b_n)}
  \bigl( \tau_j-\one(Y_1\leq q(\tau_j))\bigr), \, j=1,\ldots,
                                  k \right)\in A\right] \to 
                                  \nu_{\tau_1,\ldots, \tau_k}(A)
\end{align}
with
$$
A = \bigcup_{j=1,\ldots, k} \bigl\{ \bx:\, x_d\geq 0, \, d=1,\ldots,
k, \, x_j\geq t_j\bigr\}.
$$

We have: 
\begin{align*} 
&  n\bbP\left[\left( h_n^{-1}\frac{D_1}{\max(e(X_1),b_n)}
  \bigl( \tau_j-\one(Y_1\leq q(\tau_j))\bigr), \, j=1,\ldots,
                                  k \right)\in A\right]\\
\notag =&n\int_0^{b_n} z\bbP\bigl(   Y_1> \max_{j=1,\ldots, k} q(\tau_j)
          \big| e(X_1)=z\bigr) \one \bigl( \tau_j\geq t_jh_nb_n \
          \text{for some} \ \ j=1,\ldots, k\bigr) \LFe(dz)\\ 
  \notag +& n\int_{b_n}^1 z\bbP\bigl(   Y_1> \max_{j=1,\ldots, k} q(\tau_j)
          \big| e(X_1)=z\bigr) \one \bigl( \tau_j\geq t_jh_nz \
          \text{for some} \ \ j=1,\ldots, k\bigr) \LFe(dz)  \\
\notag =&: I_{1,n}+I_{2,n}.
\end{align*}
It follows immediately from \eqref{e:moderate.ass}, \eqref{e:key.lim},\eqref{e:small.p},  
the choice of $(t_j)$ and
\eqref{e:small.z}, that 
\begin{align} \label{e:trea.I1n}
 I_{1,n} \sim& \frac{\gamma_1-1}{\gamma_1} n b_n 
   \bbP(e(X_1)\leq b_n)  G_0\bigl[\bigl( \max_{j=1,\ldots, k} q(\tau_j),
               \infty)\bigr] \one \bigl( \tau_j\geq t_j\theta \
          \text{for some} \ \ j=1,\ldots, k\bigr)
\\
   \notag  \to& \frac{\gamma_1-1}{\gamma_1}\theta^{\gamma_1}  
G_0\bigl[\bigl( \max_{j=1,\ldots, k} q(\tau_j),
               \infty)\bigr] \one \bigl( \tau_j\geq t_j\theta \
          \text{for some} \ \ j=1,\ldots, k\bigr)
= \nu_{\tau_1,\ldots,\tau_k;1}(A). 
\end{align}

We now address the behavior of $I_{2,n}$. We start with the obvious
observations that for all $M$ large enough (specifically, for
$M>\max_j\tau_j/(t_j\theta)$), 
\begin{align} \label{I.2n.M}
\  \lim_{n\to\infty} n  \int_{Mb_n}^1 z\bbP\bigl(   Y_1> \max_{j=1,\ldots, k} q(\tau_j)
          \big| e(X_1)=z\bigr) \one \bigl( \tau_j\geq t_jh_nz \
          \text{for some} \ \ j=1,\ldots, k\bigr) \LFe(dz)  =0
\end{align}
(in fact, the integral vanishes for all $n$ large enough). 

Fix now a large integer $M>1$. We have 
 by \eqref{e:moderate.ass}, \eqref{e:key.lim},\eqref{e:small.p}, \eqref{e:key.lim},  and
\eqref{e:small.z} that 
\begin{align} \label{e:I.2n.sm}
  &n \int_{b_n}^{Mb_n}  z\bbP\bigl(   Y_1> \max_{j=1,\ldots, k} q(\tau_j)
          \big| e(X_1)=z\bigr) \one \bigl( \tau_j\geq t_jh_nz \
          \text{for some} \ \ j=1,\ldots, k\bigr) \LFe(dz)   \\
\notag \sim& \, G_0\bigl[\bigl( \max_{j=1,\ldots, k} q(\tau_j),
               \infty)\bigr] n \int_{b_n}^{\min(Mb_n,h_n^{-1}\max_j
             \tau_j/t_j)}z \LFe(dz)  \\
\notag \sim& \, G_0\bigl[\bigl( \max_{j=1,\ldots, k} q(\tau_j),
               \infty)\bigr] n \int_{b_n}^{b_n \theta^{-1}\max_j
             \tau_j/t_j}z \LFe(dz) \\
\notag \to&\, G_0\bigl[\bigl( \max_{j=1,\ldots, k} q(\tau_j),
               \infty)\bigr] \frac{\gamma_1-1}{\gamma_1} \, 
\left[ \left(\max_{j=1,\ldots, k}
            \frac{\tau_j}{t_j}\right)^{\gamma_1}-\theta^{\gamma_1}\right]_+ \\
  \notag =&   (\gamma_1-1)
\int_\theta^{\infty}  x^{\gamma_1-1} \one\left( x<\max_{j=1,\ldots,
            k}\frac{\tau_j}{t_j}\right) dx  \   G_0\bigl[\bigl( \max_{j=1,\ldots, k} q(\tau_j),
               \infty)\bigr]. 
\end{align}
Combining \eqref{I.2n.M} and \eqref{e:I.2n.sm} we conclude that
\begin{align} \label{e:I.2n.final}
I_{2,n} \to&   (\gamma_1-1) \int_\theta^{\infty}  x^{\gamma_1-1} \one\left( x<\max_{j=1,\ldots,
            k}\frac{\tau_j}{t_j}\right) dx 
  \   G_0((
             \max_{j=1,\ldots, k} q(\tau_j),\infty))=
            \nu_{\tau_1,\ldots,\tau_k;2}(A).  
\end{align}
Now \eqref{e:right.t} follows from \eqref{e:trea.I1n} and
\eqref{e:I.2n.final}. 

Next we prove \eqref{e:tr.m}. We may and will choose 
$$
t>\theta^{-1} \max_{j=1,\ldots, k}\bigl(\tau_j\vee 1-\tau_j \bigr). 
$$
Write
\begin{align*} 
& n\E\left[ h_n^{-1}\frac{D_1}{\max(e(X_1),b_n)}
  \bigl( \tau_j-\one(Y_1\leq q(\tau_j))\bigr) \one\bigl( A_n(t)\bigr)\right] \\
=&- nh_n^{-1}\E\left[ \frac{D_1}{e(X_1)}\bigl( \tau_j-\one(Y_1\leq
   q(\tau_j))\bigr) \one\bigl( A_n(t)^c\bigr)\right]  \\
  +&  nh_n^{-1} \int_0^{b_n} \left( \frac{1}{b_n}-\frac{1}{z}\right)
     z\, \LFe(dz) \, \E\left[ \bigl( \tau_j-\one(Y_1\leq q(\tau_j))\bigr)
           \one(A_n(t))     \Big|    e(X_1)=z\right]\\
=&-nh_n^{-1}\E\left[ \frac{D_1}{\max(e(X_1),b_n)}\bigl( \tau_j-\one(Y_1\leq
   q(\tau_j))\bigr) \one\bigl( A_n(t)^c\bigr)\right]   \\
   +& nh_n^{-1} \int_0^{b_n} \left( \frac{1}{b_n}-\frac{1}{z}\right)
     z\, \LFe(dz) \, \E\left[ \bigl( \tau_j-\one(Y_1\leq q(\tau_j))\bigr)
               \Big|    e(X_1)=z\right]
=: I_{1,n}^{(m)}+  I_{2,n}^{(m)},
\end{align*}
because, by the choice of $t$, for large $n$ the event $A_n(t)$ always occurs if $z\leq b_n$. 
Notice that $ I_{1,n}^{(m)}$ is the integral of a function with
respect to the Radon measure on $\bbr\setminus \{0\}$ whose vague
convergence to $\nu_{\tau_1,\ldots,\tau_k}$ was proved in \eqref{e:right.t}. The
function is bounded, has compact support, and, by the choice of $t$, 
continuous on a set of full measure $\nu_{\tau_1,\ldots,\tau_k}$. We conclude that
\begin{equation} \label{e:mean.1n}
  I_{1,n}^{(m)} \to - \int_{\bbr^k} x_j\one\bigl( \|\bx\|_\infty>t\bigr)
\nu_{\tau_1,\ldots, \tau_k}(d\bx).
\end{equation}
Furthermore, by \eqref{e:small.p}, \eqref{e:small.z} and
\eqref{e:key.lim},
\begin{align} \label{e:mean.2n}
I_{2,n}^{(m)} \sim \int_\bbr 
  \bigl( \tau_j-\one(y\leq q(\tau_j)))\, G_0(dy)\, 
  nh_n^{-1} \left( \frac{\gamma_1-1}{\gamma_1}-1\right) \bbP(e(X_1)\leq
  b_n) \to m_{\tau_j}. 
\end{align}
Now \eqref{e:tr.m} follows from \eqref{e:mean.1n} and
\eqref{e:mean.2n}.

It remains to prove \eqref{e:tr.v}. Write
\begin{align*}
& n\, \cov\left[ h_n^{-1}\frac{D_1}{\max(e(X_1),b_n)}
  \bigl( \tau_{j_1}-\one(Y_1\leq q(\tau_{j_1}))\bigr)  \one\bigl(
                               A_n(t)\bigr), \right. \\
\notag &\left. \hskip 0.4in h_n^{-1}\frac{D_1}{\max(e(X_1),b_n)}
  \bigl( \tau_{j_2}-\one(Y_1\leq q(\tau_{j_2}))\bigr)  \one\bigl(
                               A_n(t)\bigr)\right] \\
=&n\E\left[  \left(h_n^{-1}\frac{D_1}{\max(e(X_1),b_n)}\one\bigl(
                               A_n(t)\bigr)\right)^2 
  \bigl( \tau_{j_1}-\one(Y_1\leq q(\tau_{j_1}))\bigr)  \bigl(
   \tau_{j_2}-\one(Y_1\leq q(\tau_{j_2}))\bigr)  \right] \\
-&n\left\{\E\left[  h_n^{-1}\frac{D_1}{\max(e(X_1),b_n)}
  \bigl( \tau_{j_1}-\one(Y_1\leq q(\tau_{j_1}))\bigr)  \one\bigl(
                               A_n(t)\bigr)\right] \right. \\
& \left.\E\left[  h_n^{-1}\frac{D_1}{\max(e(X_1),b_n)}
  \bigl( \tau_{j_2}-\one(Y_1\leq q(\tau_{j_2}))\bigr)  \one\bigl(
                               A_n(t)\bigr)\right] \right\} =: I_{1,n}^{(v)} + I_{2,n}^{(v)}. 
\end{align*}
Note that  $I_{2,n}^{(v)}$ is equal to the product of the expressions in
the left hand side of \eqref{e:tr.m} for $j_1$ and $j_2$ 
divided by $n$. Since that
expression converges to a finite limit, we conclude that
\begin{equation} \label{e:var.2n}
  I_{2,n}^{(v)} \to  0.
\end{equation}
Next, write for a small $\vep>0$, 
\begin{align*}
  I_{1,n}^{(v)} =& n\E\left[  \left(h_n^{-1}\frac{D_1}{\max(e(X_1),b_n)}\one\bigl(
                               A_n(\vep)\bigr)\right)^2 
  \bigl( \tau_{j_1}-\one(Y_1\leq q(\tau_{j_1}))\bigr)  \bigl(
  \tau_{j_2}-\one(Y_1\leq q(\tau_{j_2}))\bigr)  \right] \\
+& 
 n\E\left[
   \left(h_n^{-1}\frac{D_1}{\max(e(X_1),b_n)}\one\bigl(A_n(\vep)^c\cap 
                               A_n(t)\bigr)\right)^2 
  \bigl( \tau_{j_1}-\one(Y_1\leq q(\tau_{j_1}))\bigr)  \bigl(
   \tau_{j_2}-\one(Y_1\leq q(\tau_{j_2}))\bigr)  \right] \\
=& I_{1,1,n}^{(v)} +I_{1,2,n}^{(v)}. 
\end{align*}
Once again, $ I_{1,2,n}^{(v)}$ is the integral of a function with
respect to the Radon measure on $\bbr\setminus \{0\}$ whose vague
convergence to $\nu_{\tau_1,\ldots,\tau_k}$ was proved in \eqref{e:right.t}. The
function is, once again, bounded, has compact support  and, for small $\vep>0$, 
continuous on a set of full measure $\nu_{\tau_1,\ldots,\tau_k}$. We conclude that
$$
  I_{1,2,n}^{(v)} \to  \int_{\bbr^k} x_{j_1}x_{j_2}\one\bigl(\vep< \|\bx\|_\infty\leq t\bigr)
\nu_{\tau_1,\ldots, \tau_k}(d\bx).
$$
On the other hand, for small $\vep>0$, 
\begin{align*}
I_{1,1,n}^{(v)} \leq&n\E\left(h_n^{-1}\frac{D_1}{\max(e(X_1),b_n)}\one\bigl(
                               A_n(\vep)\bigr)\right)^2 
                      \sim  nh_n^{-2} \int_{b_n/(\theta\vep)}^1 \frac1z \LFe(dz) \\
 =& nh_n^{-2}\left[ \int_0^{b_n} \frac{z}{b_n^2} 
\bbP\left(  \frac{1}{h_nb_n} \max_{j=1,\ldots, k} |\tau_j-\one(Y_1\leq q(\tau_j))|\leq \vep)\Big|e(X_1)=z\right)F_{e(X)}(dz) \right. \\
 &\left. \hskip 0.3in + \int_{b_n}^{1} \frac{1}{z} 
\bbP\left(  \frac{1}{h_nz } \max_{j=1,\ldots, k} |\tau_j-\one(Y_1\leq q(\tau_j))|\leq \vep) \Big|e(X_1)=z\right)F_{e(X)}(dz)
 \right].
 \end{align*}
The first integral vanishes for large $n$ as long as $\vep>0$ is small enough. In second integral, for large $n$, the values of $z$ for which the integrand does not vanish are contained in the set
$$
\bigl[ \min_{j=1,\ldots, k} (\tau_j\wedge 1-\tau_j) /(h_n\vep), 1\bigr] \subset [cb_n/(\theta\vep),1\bigr],
$$
where $c=\min_{j=1,\ldots, k} (\tau_j\wedge 1-\tau_j)/2$. We conclude that 
 \begin{align*}                  
 \limsup _{n\to\infty}I_{1,1,n}^{(v)}  \leq 
 \lim_{n\to\infty} nh_n^{-2}\int_{cb_n/(\theta\vep)}^{1} \frac{1}{z} 
  F_{e(X)}(dz) = 
 C\vep^{2-\gamma_1}
  \end{align*}
  by \eqref{e:big.z} and \eqref{e:key.lim}.
  Letting first $n\to\infty$ and then $\vep\to 0$ shows that
  \begin{equation} \label{e:var.1n}
  I_{1,n}^{(v)} \to  \int_{\bbr^k} x_{j_1}x_{j_2}\one\bigl(\|\bx\|_\infty\leq t\bigr)
\nu_{\tau_1,\ldots, \tau_k}(d\bx).
\end{equation}
Now \eqref{e:tr.v} follows from \eqref{e:var.1n} and \eqref{e:var.2n},
and so the proof of the proposition is complete in the case
$\theta>0$.

We now consider the case $\theta =0$. Once again, we prove
\eqref{e:right.t} first, and we only consider the nonnegative
quadrant. We use the same notation as in the case $\theta>0$. The same
argument as in the case $\theta>0$ shows that now $I_{1,n}\to
0$. Regarding $I_{2,n}$, it is clear that \eqref{I.2n.M} still holds. 
 Furthermore, for large $M$, we proceed as in \eqref{e:I.2n.sm} to see that 
\begin{align*}  
  &n \int_{b_n}^{Mh_n^{-1}}  z\bbP\bigl(   Y_1> \max_{j=1,\ldots, k} q(\tau_j)
          \big| e(X_1)=z\bigr) \one \bigl( \tau_j\geq t_jh_nz \
          \text{for some} \ \ j=1,\ldots, k\bigr) \LFe(dz)   \\
\notag =\sim& \, G_0\bigl[\bigl( \max_{j=1,\ldots, k} q(\tau_j),
               \infty)\bigr] n \int_{b_n}^{h_n^{-1}\max_j
          \tau_j/t_j}z \LFe(dz). 
\end{align*}    
Since $\theta=0$, this is also 
 \begin{align*} \label{e:I.2n.sm0}         
   \sim& \, G_0\bigl[\bigl( \max_{j=1,\ldots, k} q(\tau_j),
               \infty)\bigr] n \int_{0}^{h_n^{-1}\max_j
          \tau_j/t_j}z \LFe(dz) \\
\notag \to&\, G_0\bigl[\bigl( \max_{j=1,\ldots, k} q(\tau_j),
               \infty)\bigr] \frac{\gamma_1-1}{\gamma_1} \, 
  \left(\max_{j=1,\ldots, k}
            \frac{\tau_j}{t_j}\right)^{\gamma_1}  \\
  \notag =&   (\gamma_1-1)
\int_0^{\infty}  x^{\gamma_1-1} \one\left( x<\max_{j=1,\ldots,
            k}\frac{\tau_j}{t_j}\right) dx  \   G_0\bigl[\bigl( \max_{j=1,\ldots, k} q(\tau_j),
               \infty)\bigr] = \nu_{\tau_1,\ldots,\tau_k}(A), 
\end{align*}
where we have used  \eqref{e:small.z}  and the definition of $h_n$. Therefore, \eqref{e:right.t} follows.

The statement \eqref{e:tr.m} 
for $\theta=0$ is proven in the same way as in the case $\theta>0$, except that now the event $A_n(t)$ does not occur for large $n$ if $z\leq b_n$, so the term $I_{2,n}^{(m)}$ vanishes for large $n$, and the argument for 
\eqref{e:tr.v} for $\theta=0$ is the same was as
in the case $\theta>0$. Therefore, the proof of the proposition in the
case $\theta=0$ is now complete as well. 

\medskip 
\noindent  Case 2  ($\gamma_1>2$ or $\gamma_1=2$ and $\E(1/e(X_1))<\infty$). We first show that truncation in the finite variance case plays no role in the limit.  Set  
\begin{equation*}
 \tilde Z_{n,\tau}^{(1)}=-un^{-1/2} \sum_{i=1}^n \frac{D_i}{e(X_i)}
  \bigl( \tau-\one(Y_i\leq q(\tau))\bigr) 
\end{equation*}
and note that with $Z_{n,\tau}^{(1)}$ as defined in \eqref{e:split.Zn} with $h_n=n^{1/2}$, we have
\begin{eqnarray*}
   {\rm var} \left(\tilde Z_{n,\tau}^{(1)}- Z_{n,\tau}^{(1)}\right)&\le& cu^2
   \E\left(\frac{D_1}{e(X_1)}-\frac{D_1}{\max\{e(X_1),b_n\}}\right)^2 \\
   &\le &\E\left[\left(\frac{1}{e(X_1)}-\frac{1}{b_n}\right)^2 e(X_1)\one(e(X_1)\le b_n)\right]\\
  &\le &\E\left[\left(\frac{1}{e(X_1)}\right)^2e(X_1) \one(e(X_1)\le b_n)\right]\\
  &\to& 0
\end{eqnarray*}
since $1/e(X_1)$ is integrable and $b_n\to 0$. A direct application of the central limit theorem shows that 
$$
\tilde Z_{n,\tau}^{(1)}\Rightarrow Z_{\tau}\,,
$$
where $Z_{\tau}$ is normally distributed with mean zero and variance 
$\E\left[(1/e(X_1)\left(\tau-\one(Y_1\le q(\tau))\right)^2\right]$ and hence this convergence is passed on to $Z_{n,\tau}^{(1)}$ as claimed. The joint convergence for multiple $\tau$'s follows from an application of the multivariate CLT applied to  $\tilde Z_{n,\tau}^{(1)}$ at multiple $\tau$'s.

\end{proof}

\subsection{Intermediate quantiles} \label{subsec:intermediate}
We now use the notation in Remark \ref{e:split.Zn.mod}. 
We start with a version of Proposition \ref{pr:z2} in the case of
intermediate quantiles.

\subsubsection{Proof of Proposition \ref{prop:cip.inter}}
\begin{proof}
The argument is similar to that in Proposition \ref{pr:z2}. 
We have for $u>0$: 
\begin{align*}
&\E Z^{(2)}_{n}(u) = \frac{n^2}{g_Y(q(\tau_n))^2h_n^2}\E\left[   \frac{D_1}{\max(e(X_1),b_n)} \one
  (q(\tau_n)<Y_1\leq q(\tau_n)+uh_n/n) \bigl( uh_n/n-Y_1+q(\tau_n)\bigr)\right] \\
=&\frac{n^2}{g_Y(q(\tau_n))h_n^2}\left[       \int_0^{b_n} \frac{z}{b_n}  \E\bigl[  \one
  (q(\tau_n)<Y_1\leq q(\tau_n)+uh_n/n) \bigl( uh_n/n-Y_1+q(\tau_n)
   \big| e(X_1)=z\bigr]\, \LFe(dz)  \right. \\
  &\left.  + \int_{b_n}^1    \E\bigl[  \one
  (q(\tau_n)<Y_1\leq q(\tau_n)+uh_n/n) \bigl( uh_n/n-Y_1+q(\tau_n)
   \big| e(X_1)=z\bigr]\, \LFe(dz) 
\right] \\
=& A_n(u) +B_n(u). 
\end{align*}
Notice that, first  by the assumption \eqref{e:small.p.interm}, and
then by the assumption \eqref{e:theta.n}, we have  
\begin{align*}
A_n(u) \leq& \frac{n^2}{g_Y(q(\tau_n))h_n^2}  \\
&\int_0^{b_n} \left[
  \int_{q(\tau_n)}^{q(\tau_n)+uh_n/n} \bigl( \bbP(Y_1\leq
  w|e(X_1)=z)-\bbP(Y_1\leq q(\tau_n)|e(X_1)=z)\bigr)\, dw\right]\, \LFe(dz) \\
=& \frac{n^2}{g_Y(q(\tau_n))h_n^2}  \\
&\int_0^{b_n} \left[
  \int_{q(\tau_n)}^{q(\tau_n)+uh_n/n} \bigl( (\beta +o(1))\bbP(Y_1\leq
  w)-(\beta+o(1))\bbP(Y_1\leq q(\tau_n))\bigr)\, dw\right]\, \LFe(dz) \\
=&o(1)\tau_n \frac{h_n}{n}\frac{n^2}{g_Y(q(\tau_n))h_n^2}\bbP(e(X_1)\leq b_n)\\
+& \beta \frac{n^2}{g_Y(q(\tau_n))h_n^2} 
\int_0^{b_n} \left[
  \int_{q(\tau_n)}^{q(\tau_n)+uh_n/n}  \bbP(q(\tau_n)<Y_1\leq
  w) \, dw\right]\, \LFe(dz)\\
=& o(1)\tau_n  \frac{n\bbP(e(X_1)\leq b_n)}{g_Y(q(\tau_n))h_n}
+ (\beta+o(1)) g_Y(q(\tau_n))\frac{n^2}{g_Y(q(\tau_n))h_n^2} 
  u^2\frac{h_n^2}{2n^2} \bbP(e(X_1)\leq b_n) \to 0
\end{align*}
by \eqref{e:express.n}, \eqref{e:bn.ass.interm} and regular
variation. 

Next, we write
\begin{align*} \label{e:mean.B.interm}
B_n(u) =& \frac{n^2}{g_Y(q(\tau_n))h_n^2}\E\bigl[  \one
  (q(\tau_n)<Y_1\leq q(\tau_n)+uh_n/n) \bigl(
  uh_n/n-Y_1+q(\tau_n)\bigr)\bigr]\\
\notag  -&
 \frac{n^2}{g_Y(q(\tau_n))h_n^2} \int_0^{b_n}  \E\bigl[  \one
  (q(\tau_n)<Y_1\leq q(\tau_n)+uh_n/n) \bigl( uh_n/n-Y_1+q(\tau_n)
   \big| e(X_1)=z\bigr]\, \LFe(dz). 
\end{align*}
We already know that the second term vanishes in the
limit. Furthermore, 
due to the assumptions  \eqref{e:hn.thetan} and \eqref{e:theta.n} the first term converges to $u^2/2$. 
It follows that
\begin{equation} \label{e:Zn2.mean.interm}
\lim_{n\to\infty}\E Z_n^{(2)}(u)= u^2/2, \ u\in\bbr.
\end{equation}

Next,
\begin{align*}
&\var(Z_n^{(2)}(u)) = \frac{n^3}{g_Y(q(\tau_n))^2h_n^4}\\
&\var\left[   \frac{D_1}{\max(e(X_1),b_n)} \one
  (q(\tau_n)<Y_1\leq q(\tau_n)+uh_n/n) \bigl( uh_n/n-Y_1+q(\tau_n)\bigr)\right] \\
\leq& \frac{n^3}{g_Y(q(\tau_n))^2h_n^4}\E\left[   \frac{D_1}{\max(e(X_1),b_n)} \one
  (q(\tau_n)<Y_1\leq q(\tau_n)+uh_n/n) \bigl( uh_n/n-Y_1+q(\tau_n)\bigr)\right]^2 \\
=& \frac{n^3}{g_Y(q(\tau_n))^2h_n^4} \\
&\left[       \int_0^{b_n} \frac{z}{b_n^2} \E\bigl[  \one
  (q(\tau_n)<Y_1\leq q(\tau_n)+uh_n/n) \bigl( uh_n/n-Y_1+q(\tau_n))^2
   \big| e(X_1)=z\bigr]\, \LFe(dz)  \right. \\
  &\left.  + \int_{b_n}^1  \frac1z \E\bigl[  \one
  (q(\tau_n)<Y_1\leq q(\tau_n)+uh_n/n) \bigl( uh_n/n-Y_1+q(\tau_n))^2
   \big| e(X_1)=z\bigr]\, \LFe(dz) 
\right] \\
=& A_n^{(2)}(u) +B_n^{(2)}(u). 
\end{align*}
Once again, first  by the assumption \eqref{e:small.p.interm}, and
then by assumption \eqref{e:theta.n}, 
\begin{align*}
A_n^{(2)}(u) \leq& \frac{n^3}{g_Y(q(\tau_n))^2h_n^4}  \frac{1}{b_n}\\
&\int_0^{b_n} \left[ 2\int_0^{uh_n/n} w \bigl( \bbP(Y_1\leq
  q(\tau_n)+uh_n/n-w|e(X_1)=z)-\bbP(Y_1\leq q(\tau_n)|e(X_1)=z)\bigr)\, dw\right]\\
&\hskip 4in  \LFe(dz) \\
\leq& \frac{n^3}{g_Y(q(\tau_n))^2h_n^4b_n} 2\frac{uh_n}{n} \\
&\int_0^{b_n} \left[
  \int_{q(\tau_n)}^{q(\tau_n)+uh_n/n} \bigl( \bbP(Y_1\leq
  w|e(X_1)=z)-\bbP(Y_1\leq q(\tau_n)|e(X_1)=z)\bigr)\, dw\right]\, \LFe(dz) \\
=& \frac{2un^2}{g_Y(q(\tau_n))^2h_n^3b_n}  \\
&\int_0^{b_n} \left[
  \int_{q(\tau_n)}^{q(\tau_n)+uh_n/n} \bigl( (1+o(1))\bbP(Y_1\leq
  w)-(1+o(1))\bbP(Y_1\leq q(\tau_n))\bigr)\, dw\right]\, \LFe(dz) \\
=&o(1)\tau_n \frac{h_n}{n}\frac{n^2}{g_Y(q(\tau_n))^2h_n^3b_n}\bbP(e(X_1)\leq b_n)\\
+&   \frac{2un^2}{g_Y(q(\tau_n))^2h_n^3b_n} 
\int_0^{b_n} \left[
  \int_{q(\tau_n)}^{q(\tau_n)+uh_n/n}  \bbP(q(\tau_n)<Y_1\leq
  w) \, dw\right]\, \LFe(dz)\\
=& o(1)\tau_n  \frac{n\bbP(e(X_1)\leq b_n)}{g_Y(q(\tau_n))^2h_n^2b_n}
+  (1+o(1)) g_Y(q(\tau_n))\frac{2un^2}{g_Y(q(\tau_n))^2h_n^3b_n} 
  u^2\frac{h_n^2}{2n^2} \bbP(e(X_1)\leq b_n)\to 0,
\end{align*}
again by \eqref{e:express.n}, \eqref{e:bn.ass.interm} and regular
variation.

Next, for a large $M>0$ we write 
\begin{align*}
B_n^{(2)}(u) = \frac{n^3}{g_Y(q(\tau_n))^2h_n^4} \left(
  \int_{b_n}^{Mb_n} + \int_{Mb_n}^{1} \right)=: B_n^{(2,1)}(u)+ 
  B_n^{(2,2)}(u).
\end{align*}
The same argument as the one used above for $A_n^{(2)}(u)$ shows that
$B_n^{(2,1)}(u)\to 0$ for every fixed $M$ (just replace $b_n$ by
$Mb_n$). On the other hand,  by \eqref{e:theta.n}, and arguing as
above, 
 \begin{align*}
B_n^{(2,2)}(u) \leq&  \frac{n^3}{g_Y(q(\tau_n))^2h_n^4Mb_n}\E\bigl[  \one
  (q(\tau_n)<Y_1\leq q(\tau_n)+uh_n/n) \bigl( uh_n/n-Y_1+q(\tau_n)\bigr)^2
   \bigr]\\
=& (1+o(1)) \frac{u^3}{3g_Y(q(\tau_n)) h_nMb_n}
\to 0 
\end{align*}
as $n\to\infty$ and then $M\to\infty$. We conclude that
\begin{equation} \label{e:Zn2.var.interm}
\lim_{n\to\infty} \var(Z_n^{(2)}(u))= 0, \ u\in\bbr.
\end{equation}
It follows from \eqref{e:Zn2.mean.interm} and \eqref{e:Zn2.var.interm}
that
\begin{equation*} \label{e:Zn2.prob.interm}
Z_n^{(2)}(u)\to u^2/2 \ \ \text{in probability}, \ u\in\bbr.
\end{equation*}
\end{proof}

We proceed with a version of Proposition \ref{pr:z1} appropriate to the case of intermediate quantiles. 

\subsection{Proof of Proposition \ref{pr:Z1.inter} }

\begin{proof}
As in the proof of Proposition \ref{pr:z1} we need to show the following 3
statements. 
\begin{align} \label{e:vague.interm}
&n\bbP\left( \bigl(g_Y(q(\tau_n))h_n\bigr)^{-1}\frac{D_1}{\max(e(X_1),b_n)}
  \bigl( \tau_n-\one(Y_1\leq q(\tau_n))\bigr) \in \cdot\right)\vague 
                                  \nu (\cdot) 
\end{align}
vaguely in $\overline{\bbr}\setminus \{0\}$, 
 \begin{align} \label{e:tr.m.interm}
&n\E\left[ \bigl(g_Y(q(\tau_n))h_n\bigr)^{-1}\frac{D_1}{\max(e(X_1),b_n)}
  \bigl( \tau_n-\one(Y_1\leq q(\tau_n))\bigr) \one\bigl( A_n(t)\bigr)\right] \\
\notag &\hskip 1.5in \to
m- \int_{\bbr} x\one\bigl( |x|>t\bigr)
\nu (dx)  
 \end{align}
 for  
 any $t>0, \, t\not=1/\theta$, where 
 \begin{equation*} \label{e:A.n1}
A_n(t) = \left\{ \bigl(g_Y(q(\tau_n))h_n\bigr)^{-1}\frac{D_1}{\max(e(X_1),b_n)}
  \bigl| \tau_n-\one(Y_1\leq q(\tau_n))\bigr|\leq t \right\}. 
\end{equation*}

Finally, we need to show that for every $t$ as above, 
\begin{align} \label{e:tr.v.interm}
&n\, \var\left[ \bigl(g_Y(q(\tau_n))h_n\bigr)^{-1}\frac{D_1}{\max(e(X_1),b_n)}
  \bigl( \tau_{n}-\one(Y_1\leq q(\tau_{n}))\bigr)  \one\bigl(
                               A_n(t)\bigr)\right] \\
\notag &\hskip 1.5in 
 \to
\int_{\bbr} x^2\one\bigl( |x|\leq t\bigr)
\nu (dx).  
\end{align}

  We have for $t>0$,
  \begin{align*} \label{e:split.neg.interm}
& n\bbP\left( \bigl(g_Y(q(\tau_n))h_n\bigr)^{-1}\frac{D_1}{\max(e(X_1),b_n)}
  \bigl( \tau_n-\one(Y_1\leq q(\tau_n))\bigr) \leq -t \right)\\
\notag =&n\int_0^{b_n} z\bbP\bigl(   Y_1\leq    q(\tau_n)
          \big| e(X_1)=z\bigr) \one \bigl( (1-\tau_n)\geq t g_Y(q(\tau_n))h_nb_n
          \bigr) \LFe(dz)\\ 
  \notag +& n\int_{b_n}^1 z\bbP\bigl(   Y_1\leq    q(\tau_n)
          \big| e(X_1)=z\bigr) \one \bigl( (1-\tau_n)\geq t g_Y(q(\tau_n))h_nz
          \bigr)  \LFe(dz)  \\
\notag =&: I_{1,n}+I_{2,n}.
\end{align*}

If $0<t<1/\theta$, then using first  \eqref{e:small.p.interm}, then Lemma \ref{e:regvar.est}, 
\eqref{e:express.n} and, finally, \eqref{e:bn.ass.interm} and 
regular variation, we have 
\begin{align*}
I_{1,n} \sim& n\beta\bbP\bigl(   Y_1\leq    q(\tau_n)\bigr) 
  \int^{b_n}_0 z  \LFe(dz)
  \sim \beta n\tau_n \frac{\gamma_1-1}{\gamma_1} b_n \bbP(e(X_1)\leq b_n)  \\
\sim& \beta \frac{\gamma_1-1}{\gamma_1}  \frac{b_n \bbP(e(X_1)\leq b_n)}{\bigl(
g_Y(q(\tau_n))   h_n\bigr)^{-1}\bbP\bigl(e(X_1)\leq \bigl(
g_Y(q(\tau_n))   h_n\bigr)^{-1}\bigr)}\to \beta \frac{\gamma_1-1}{\gamma_1}
      \theta^{\gamma_1}. 
\end{align*}
On the other hand, if $t>1/\theta$, then both $I_{1,n}\to 0$ and $I_{2,n}\to 0$. 

When $0<t<1/\theta$, we treat the term $I_{2,n}$ in the same way as we treated the
analogous term in the case of fixed quantiles. It is still true that
$$
\int_{Mb_n}^1 z\bbP\bigl(   Y_1\leq    q(\tau_n)
          \big| e(X_1)=z\bigr) \one \bigl( (1-\tau_n)\geq t g_Y(q(\tau_n))h_nz
          \bigr)  \LFe(dz)=0
$$
for large $n$ if $M$ is large enough. On the other hand, for a fixed
large $M$,  first  by \eqref{e:small.p.interm}, then by \eqref{e:bn.ass.interm}, Lemma \ref{e:regvar.est}, 
\eqref{e:express.n} and, again, \eqref{e:bn.ass.interm}
\begin{align*}
&n\int_{b_n}^{Mb_n} z\bbP\bigl(   Y_1\leq    q(\tau_n)
          \big| e(X_1)=z\bigr) \one \bigl( (1-\tau_n)\geq t g_Y(q(\tau_n))h_nz
          \bigr)  \LFe(dz) \\
\sim& \beta\tau_n n \int_{b_n}^{\min(Mb_n,(g_Y(q(\tau_n))h_n)^{-1} 
             (1-\tau_n)/t)}z \LFe(dz) \\
  \sim & \beta\tau_n n \int_{b_n}^{b_n/(\theta t)}z \LFe(dz)
         \to \beta \frac{\gamma_1-1}{\gamma_1} \bigl( t^{-\gamma_1} -
         \theta^{\gamma_1}\bigr)\\
=& \beta(\gamma_1-1) \int_\theta^{\infty}  x^{\gamma_1-1} \one( x<1/t)dx.         
         \end{align*}         
We conclude that
 $$
I_{2,n} \to  \beta(\gamma_1-1) \int_\theta^{\infty}  x^{\gamma_1-1}
\one( x<1/t)dx
= \beta(\gamma_1-1) \int_0^{1/\theta} x^{-(\gamma_1+1)} \one(x>t)dx,
$$
and so for any $t>0$,
\begin{equation} \label{e:neg.jump.mod.f}
n\bbP\left( \bigl(g_Y(q(\tau_n))h_n\bigr)^{-1}\frac{D_1}{\max(e(X_1),b_n)}
  \bigl( \tau_n-\one(Y_1\leq q(\tau_n))\bigr) \leq -t \right) \to
\nu\bigl( (-\infty, -t]\bigr).
\end{equation}
On the other hand, by\eqref{e:bn.ass.interm}, for any $t>0$, since $\tau_n\to 0$
 \begin{align} \label{e:split.pos.interm}
& n\bbP\left( \bigl(g_Y(q(\tau_n))h_n\bigr)^{-1}\frac{D_1}{\max(e(X_1),b_n)}
  \bigl( \tau_n-\one(Y_1\leq q(\tau_n))\bigr) >t \right) =0
\end{align}
for all $n$ large enough. The statement \eqref{e:vague.interm} now
follows from \eqref{e:neg.jump.mod.f} and \eqref{e:split.pos.interm}.

We now prove \eqref{e:tr.m.interm}. Write
\begin{align*} \label{e:split.mean.interm}
&n\E\left[ \bigl(g_Y(q(\tau_n))h_n\bigr)^{-1}\frac{D_1}{\max(e(X_1),b_n)}
  \bigl( \tau_n-\one(Y_1\leq q(\tau_n))\bigr) \one\bigl( A_n(t)\bigr)\right] \\
\notag =&n\E\left[ \bigl(g_Y(q(\tau_n))h_n\bigr)^{-1}\frac{D_1}{\max(e(X_1),b_n)}
  \bigl( \tau_n-\one(Y_1\leq q(\tau_n))\bigr) \right] \\
 \notag -&n\E\left[ \bigl(g_Y(q(\tau_n))h_n\bigr)^{-1}\frac{D_1}{\max(e(X_1),b_n)}
  \bigl( \tau_n-\one(Y_1\leq q(\tau_n))\bigr) \one\bigl(
          A_n(t)^c\bigr)\right] \\
\notag =& I_{1,n}^{(m)} - I_{2,n}^{(m)}.
\end{align*}
As in the case of fixed quantiles, $ I_{2,n}^{(m)}$ is the integral of a function with
respect to a Radom measure on $\bbr\setminus \{0\}$ whose vague
convergence to $\nu$ was proved in \eqref{e:vague.interm}. The
function is bounded, with compact support, and, by the choice of $t$, 
continuous on a set of full measure $\nu$. Therefore, 
\begin{equation} \label{e:lim.mean.interm2}
  I_{2,n}^{(m)} \to \int_\bbr x\one(|x|>t)\, \nu(dx).
\end{equation}
Furthermore, by \eqref{e:small.p.interm}, Lemma \ref{e:regvar.est}, 
\eqref{e:express.n} and regular variation, 
\begin{align} \label{e:lim.mean.interm1}
 I_{1,n}^{(m)}=&\frac{n}{g_Y(q(\tau_n))h_n} \left[ \int_0^{b_n}
  \frac{z}{b_n}\E\bigl[ \bigl( \tau_n-\one(Y_1\leq
  q(\tau_n))\bigr)\big| e(X_1)=z\bigr] \LFe(dz) \right.\\
\notag &\left. + \int_{b_n}^1 
   \E\bigl[ \bigl( \tau_n-\one(Y_1\leq
  q(\tau_n))\bigr)\big| e(X_1)=z\bigr] \LFe(dz) \right]\\
\notag =& \frac{n}{g_Y(q(\tau_n))h_n}  \int_0^{b_n}
  \left( \frac{z}{b_n}-1\right)\E\bigl[ \bigl( \tau_n-\one(Y_1\leq
  q(\tau_n))\bigr)\big| e(X_1)=z\bigr] \LFe(dz) \\
\notag \sim& (1-\beta)  \frac{n\tau_n}{g_Y(q(\tau_n))h_n}  \int_0^{b_n}
  \left( \frac{z}{b_n}-1\right)  \LFe(dz)\\
\notag \sim& (1-\beta)  \frac{n\tau_n}{g_Y(q(\tau_n))h_n} \left(
             -\frac{1}{\gamma_1}\bbP(e(X_1)\leq b_n)\right) \to m. 
\end{align}
Now \eqref{e:tr.m.interm} follows from \eqref{e:lim.mean.interm2} and
\eqref{e:lim.mean.interm1}.

Finally, we prove \eqref{e:tr.v.interm}.  By the already proved
\eqref{e:tr.m.interm},
$$
n\left\{\E\left[ \bigl(g_Y(q(\tau_n))h_n\bigr)^{-1}\frac{D_1}{\max(e(X_1),b_n)}
  \bigl( \tau_n-\one(Y_1\leq q(\tau_n))\bigr) \one\bigl(
  A_n(t)\bigr)\right]\right\}^2\to 0.
$$
Furthermore,
\begin{align*}
&n\E\left[ \bigl(g_Y(q(\tau_n))h_n\bigr)^{-1}\frac{D_1}{\max(e(X_1),b_n)}
  \bigl( \tau_n-\one(Y_1\leq q(\tau_n))\bigr) \one\bigl( A_n(t)\bigr)\right]^2
=\frac{n}{g_Y(q(\tau_n))^2h_n^2}\times 
\\
\notag & \left[ \int_0^{b_n}
  \frac{z}{b_n^2}\E\bigl[ \bigl( \tau_n-\one(Y_1\leq
  q(\tau_n))\bigr)^2 \one\bigl( |\tau_n-\one(Y_1\leq
  q(\tau_n))|\leq t q(\tau_n)h_nb_n\bigr)
\big| e(X_1)=z\bigr] \LFe(dz) \right.\\
\notag &\left. + \int_{b_n}^1 
     \frac{1}{z}\E\bigl[ \bigl( \tau_n-\one(Y_1\leq
  q(\tau_n))\bigr)^2 \one\bigl( |\tau_n-\one(Y_1\leq
  q(\tau_n))|\leq t q(\tau_n)h_nz\bigr)
\big| e(X_1)=z\bigr] \LFe(dz) \right]\\
\notag =& I_{1,n}^{(v)}+ I_{2,n}^{(v)}.
 \end{align*}
 Then
 \begin{align*}
&I_{1,n}^{(v)}= \frac{n}{g_Y(q(\tau_n))^2h_n^2b_n^2} \times\\
&\left[ \int_0^{b_n} z\bbP(Y_1\leq q(\tau_n)|e(X_1)=z) (\tau_n-1)^2
   \one\bigl( (1-\tau_n )\leq t q(\tau_n)h_nb_n\bigr)
\LFe(dz) \right.\\
& \left.+ \int_0^{b_n} z\bbP(Y_1> q(\tau_n)|e(X_1)=z) \tau_n^2
   \one\bigl( \tau_n\leq t q(\tau_n)h_nb_n\bigr)
\LFe(dz) \right] \\
&=: I_{1,1,n}^{(v)}+I_{1,2,n}^{(v)}. 
 \end{align*}
 By   \eqref{e:small.p.interm}, \eqref{e:bn.ass.interm}, Lemma \ref{e:regvar.est}, again \eqref{e:small.p.interm}, 
\eqref{e:express.n} and regular variation, 
 \begin{align*}
 I_{1,1,n}^{(v)} \sim& \beta
   \frac{n\tau_n}{g_Y(q(\tau_n))^2h_n^2b_n^2}\one(1\leq t\theta)
    \int_0^{b_n} z \LFe(dz) \\
\sim& \beta\frac{\gamma_1-1}{\gamma_1} \frac{n\tau_n}{g_Y(q(\tau_n))^2h_n^2b_n}
\one(1\leq t\theta) \bbP(e(X_1)\leq b_n)\\
\to& \beta\frac{\gamma_1-1}{\gamma_1} \theta^{\gamma_1-2} \one(1\leq
     t\theta). 
 \end{align*}
 Since $\tau_n\to 0$, a similar argument gives us 
 \begin{align*}
 I_{1,2,n}^{(v)} \leq \frac{n\tau_n^2}{g_Y(q(\tau_n))^2h_n^2b_n^2}
     \int_0^{b_n} z \LFe(dz)  \to 0, 
  \end{align*}
  so that 
  \begin{equation*} \label{e:I1n.v.interm}
   I_{1,n}^{(v)} \to \beta\frac{\gamma_1-1}{\gamma_1} \theta^{\gamma_1-2} \one(1\leq
     t\theta). 
\end{equation*}

Similarly, 
 \begin{align*}
&I_{2,n}^{(v)}= \frac{n}{g_Y(q(\tau_n))^2h_n^2} \\
&\left[ \int_{b_n}^1 \frac1z\bbP(Y_1\leq q(\tau_n)|e(X_1)=z) (\tau_n-1)^2
   \one\bigl( (1-\tau_n )\leq t q(\tau_n)h_nz\bigr)
\LFe(dz) \right.\\
& \left.+ \int_{b_n}^1 \frac1z\bbP(Y_1> q(\tau_n)|e(X_1)=z) \tau_n^2
   \one\bigl( \tau_n\leq t q(\tau_n)h_nz\bigr)
\LFe(dz) \right] \\
&=: I_{2,1,n}^{(v)}+I_{2,2,n}^{(v)}. 
 \end{align*}
By Lemma \ref{e:regvar.est},   \eqref{e:express.n}, \eqref{e:bn.ass.interm}, regular variation and the fact that $\tau_n\to 0$, 
\begin{align*}
  I_{2,2,n}^{(v)} \leq& \frac{n\tau_n^2}{g_Y(q(\tau_n))^2h_n^2}
  \int_{b_n}^1 \frac1z \LFe(dz) \\
\sim& \frac{\gamma_1-1}{2-\gamma_1}
  \frac{n\tau_n^2}{g_Y(q(\tau_n))^2h_n^2b_n} \bbP(e(X_1)\leq b_n)\\
\sim& \frac{\gamma_1-1}{2-\gamma_1}
  \frac{\tau_n}{g_Y(q(\tau_n))h_nb_n} \frac{\bbP(e(X_1)\leq
      b_n)}{\bbP(e(X_1)\leq (g_Y(q(\tau_n))h_n)^{-1})} \to 0.
\end{align*}
Furthermore, let $z_0>0$ be such that \eqref{e:small.p.interm} holds
for $z<z_0$. Write 
\begin{align*}
  I_{2,1,n}^{(v)} = \frac{n}{g_Y(q(\tau_n))^2h_n^2}\left( \int_{b_n}^{z_0} +
\int_{z_0}^1 \right) =: I_{2,1,1,n}^{(v)}+ 
  I_{2,1,2,n}^{(v)}. 
\end{align*}
Notice that by \eqref{e:express.n}, 
\begin{align*}
I_{2,1,2,n}^{(v)} \leq& \frac{1}{z_0} \frac{n}{g_Y(q(\tau_n))^2h_n^2}
\int_{z_0}^1 \bbP(Y_1\leq q(\tau_n)|e(X_1)=z) \LFe(dz) \\
\leq& \frac{1}{z_0} \frac{n}{g_Y(q(\tau_n))^2h_n^2} \bbP(Y_1\leq q(\tau_n))
= \frac{1}{z_0} \frac{n\tau_n}{g_Y(q(\tau_n))^2h_n^2} \\
\sim& \frac{1}{z_0} \frac{1}{g_Y(q(\tau_n))h_n}\frac{ 1}{\bbP(e(X_1)\leq
      (g_Y(q(\tau_n))h_n)^{-1})}  \to 0
\end{align*}
since $\gamma_1<2$ and $g_Y(q(\tau_n))h_n\to\infty$. On the other
hand, if $0<t<1/\theta$, using \eqref{e:bn.ass.interm}, Lemma \ref{e:regvar.est}, 
\eqref{e:express.n}, again \eqref{e:bn.ass.interm} and regular variation, 
\begin{align*}
  I_{2,1,1,n}^{(v)} \sim&\beta \frac{n\tau_n}{g_Y(q(\tau_n))^2h_n^2} 
\int_{b_n}^{z_0} \frac1z  \one\bigl( (1-\tau_n )\leq t q(\tau_n)h_nz\bigr)
\LFe(dz) \\
\sim& \beta \frac{n\tau_n}{g_Y(q(\tau_n))^2h_n^2} 
\int_{b_n/(t\theta)}^{z_0} \frac1z \LFe(dz) \\
\sim& \beta\frac{\gamma_1-1}{2-\gamma_1}
      \frac{n\tau_n}{g_Y(q(\tau_n))^2h_n^2}
      \bigl(b_n/(t\theta)\bigr)^{-1}P\bigl(e(X_1)\leq
      b_n/(t\theta)\bigr)
      \to \beta\frac{\gamma_1-1}{2-\gamma_1}t^{2-\gamma_1}, 
\end{align*}
while for $t>1/\theta$, in a similar manner,
we obtain $
I_{2,1,1,n}^{(v)} \to
\beta\frac{\gamma_1-1}{2-\gamma_1}\theta^{-(2-\gamma_1)}$.
We conclude that 
\begin{align*}
&n\, \var\left[ \bigl(g_Y(q(\tau_n))h_n\bigr)^{-1}\frac{D_1}{\max(e(X_1),b_n)}
  \bigl( \tau_{n}-\one(Y_1\leq q(\tau_{n}))\bigr)  \one\bigl(
                               A_n(t)\bigr)\right] \\
\to& \beta\frac{\gamma_1-1}{\gamma_1} \theta^{\gamma_1-2} \one(1\leq
     t\theta)
+ \beta\frac{\gamma_1-1}{2-\gamma_1}\min(t,1/\theta)^{2-\gamma_1}
= \int_{\bbr} x^2\one\bigl( |x|\leq t\bigr)
\nu (dx),
\end{align*}
proving \eqref{e:tr.v.interm} and completing the proof. 
\end{proof}

\subsection{Comments on using subsampling in constructing bootstrap replicates.}  \label{subsample}

 It is well known that in many problems where bootstrapping is used to approximate the sampling distribution of a statistic which has a non-normal limit,   subsampling is required.  That is, a subsample of size $m=m_n\to\infty$ with $m/n\to0$ must be drawn to obtain an asymptotically correct bootstrap procedure.  This is especially true for {\it heavy-tailed} data; see, for example, \cite{davis1997bootstrapping}, who consider bootstrapping parameter estimates for regression and autoregressive models.  Similar ideas can be applied in the context of this paper.  Consider a bootstrap subsample, $((e(X_i^*),D_i^*, Y_i^*), \,i=1,\ldots,m)$, of the observed data $\bU_n:=((e(X_i),D_i, Y_i), \,i=1,\ldots,n)$.  In the first step of the argument, we take $e(x)$ to be the true propensity score and then show that it can be replaced  with the estimated version.  The limit distribution of the companion bootstrapped quantile is found by optimizing the corresponding objective function given by $Z_{m,\tau}^*(u)$ with all the relevant terms replaced by their bootstrapped counterparts. To illustrate the idea in a simpler setting,  take $b_n=0$.  Then convergence of the first term in $\frac{m}{h_m^2}Z_{m,\tau}^*$,  denoted by $Z_{m,\tau}^{(1)*}$, just requires (see \eqref{e:right.t})
\begin{equation*}
\nu_n^*(\cdot)=m\bbP(h_{m}^{-1}\frac{D_1^*}{e(X_1^*)}\left(\tau-\one(Y_1^*\le \tau)\right)\in \cdot\,\mid\, \bU_n)\cip \nu_\tau(\cdot)\,,
\end{equation*}
where $\nu_\tau$ is the L\'evy measure specified in Proposition \ref{pr:z1}, and convergence is with respect to the vague topology.
Specifically, we need to show that for any bounded Borel $A$  in $\bar{\R}\setminus\{0\}$, bounded away from 0, 
\begin{equation}  \label{eq:nustar}
\nu_n^*(A)=\frac{m}{n}\sum_{i=1}^n\one\left[h_m^{-1}\frac{D_i}{e(X_i)}\left(\tau-\one(Y_i\le \tau)\right)\in A\right]\cip \nu_\tau(A)\,.
\end{equation}
We have from \eqref{eq:nuA}
\begin{equation*}
  \E( \nu_n^*(A))=m\bbP(h_{m}^{-1}\frac{D_1}{e(X_1)}\left(\tau-\one(Y_1\le \tau)\right)\in A)\to \nu_\tau(A) 
\end{equation*}
and
\begin{equation*}
  \mbox{Var}( \nu_n^*(A))= \frac{m}{n}\left(\E( \nu_n^*(A))-m^{-1}\E^2( \nu_n^*(A))\right)\to0
\end{equation*}
if and only if $m\to\infty$ and $m/n\to0$, which proves \eqref{eq:nustar}.  This computation illustrates why subsampling is needed here. Of course, to prove the full consistency of the subsampling bootstrap, there are other terms to consider that require  a more in-depth theoretical development.  We do not consider those issues further in the present paper.

\newpage

\bibliographystyle{plainnat}
\bibliography{biblio}

\end{document}